\documentclass[11pt]{article}\UseRawInputEncoding
\usepackage{amssymb,amsfonts,amsmath,latexsym,epsf,tikz,url}
\usepackage{graphics}
\usepackage[usenames,dvipsnames]{pstricks}
\usepackage{pstricks-add}
\usepackage{epsfig}
\usepackage{pst-grad} 
\usepackage{pst-plot} 
\usepackage[space]{grffile} 
\usepackage{etoolbox} 
\makeatletter 
\patchcmd\Gread@eps{\@inputcheck#1 }{\@inputcheck"#1"\relax}{}{}
\makeatother

\newtheorem{theorem}{Theorem}[section]
\newtheorem{proposition}[theorem]{Proposition}

\newtheorem{corollary}[theorem]{Corollary}

\newcommand{\qed}{\hfill $\square$\medskip}

\begin{document}

\title{On the Graovac-Ghorbani and atom-bond connectivity indices of graphs from primary subgraphs}

\author{
Nima Ghanbari
}

\date{\today}

\maketitle

\begin{center}
Department of Informatics, University of Bergen, P.O. Box 7803, 5020 Bergen, Norway\\
\bigskip

{\tt Nima.Ghanbari@uib.no }
\end{center}


\begin{abstract}
Let $G=(V,E)$ be a finite simple graph. The
Graovac-Ghorbani index of a graph G is defined as 
$ABC_{GG}(G)=\sum_{uv\in E(G)}\sqrt{\frac{n_u(uv,G)+n_v(uv,G)-2}{n_u(uv,G)n_v(uv,G)}},$
where $n_u(uv,G)$ is the number of vertices closer to vertex $u$ than vertex $v$ of the edge $uv\in E(G)$. $n_v(uv,G)$ is defined analogously. The
atom-bond connectivity index of a graph G is defined as 
$ABC(G)=\sum_{uv\in E(G)}\sqrt{\frac{d_u+d_v-2}{d_ud_v}},$
where $d_u$ is the degree of vertex $u$ in $G$. Let $G$ be a connected graph constructed from pairwise disjoint connected graphs $G_1,\ldots ,G_k$ by selecting a vertex of $G_1$, a vertex of $G_2$, and identifying these two
  vertices. Then continue in this manner inductively. We say that $G$ is obtained by point-attaching from $G_1, \ldots ,G_k$ and that $G_i$'s are the primary subgraphs of $G$. 
  In this paper, we give some lower and upper
bounds on Graovac-Ghorbani and atom-bond connectivity indices for these graphs. Additionally, we consider some  particular cases  of these graphs that  are  of importance in chemistry  and study their Graovac-Ghorbani and atom-bond connectivity indices. 
\end{abstract}


\medskip
\noindent{\bf AMS Subj.\ Class.:} 05C09, 05C12, 05C92.

\section{Introduction}

A molecular graph is a simple graph such that its vertices correspond to the atoms and the edges to the bonds of a molecule. 
 Let $G = (V, E)$ be a finite, connected, simple graph.  
 A topological index of $G$ is a real number related to $G$. It does not depend on the labeling or pictorial representation of a graph. The Wiener index $W(G)$ is the first distance based topological index defined as $W(G) = \sum_{\{u,v\}\subseteq G}d(u,v)=\frac{1}{2} \sum_{u,v\in V(G)} d(u,v)$ with the summation runs over all pairs of vertices of $G$ \cite{20}.
 The topological indices and graph invariants based on distances between vertices of a graph are widely used for characterizing molecular graphs, establishing relationships between structure and properties of molecules, predicting biological activity of chemical compounds, and making their chemical applications.  The
 Wiener index is one of the most used topological indices with high correlation with many physical and chemical indices of molecular compounds \cite{20}. 
 In 2010, Graovac et al. \cite{Gra} introduced a new bond-additive structural invariant as a quantitative refinement of the distance nonbalancedness and also a measure of peripherality  in  graphs.  They  used  the  name Graovac-Ghorbani index  for  this  invariant  which  is  defined  as 
 $$ABC_{GG}(G)=\sum_{uv\in E(G)}\sqrt{\frac{n_u(uv,G)+n_v(uv,G)-2}{n_u(uv,G)n_v(uv,G)}},$$
    where $n_u(uv,G)$ is  the number of vertices of $G$ closer to $u$ than to $v$, and similarly, $n_v(uv,G)$ is the number of vertices closer to $v$ than to $u$. Equidistant vertices from $u$ and $v$ are not taken
into account to compute $n_u(uv,G)$ and $n_v(uv,G)$. They determined  some bounds on this index. Graovac et al. in \cite{Gra2} computed that for some nanostar dendrimers.  Some other upper and lower bounds on the $ABC_{GG}$ index and also characterizing the extremal graphs was studied by  Das \cite{15}. Ghorbani et al. in \cite{Gho} calculated the $ABC_{GG}$ of an infinite family of fullerenes. More results on this index can be found in \cite{Dim,Fur,Pac,Ros1,Ros}.

 Graovac and Ghorbani  defined  $ABC_{GG}(G)$ \cite{Gra} which motivated by the definition of atom-bond connectivity index. Initially, the atom-bond connectivity index  of a graph $G$, $ABC(G)$, was defined \cite{Est1} as:
$$ABC(G)=\sqrt{2}\sum_{uv\in E(G)}\sqrt{\frac{d_u+d_v-2}{d_ud_v}},$$
but later on, this index was very slightly redefined \cite{Est} by dropping the factor $\sqrt{2}$. We refer the reader to \cite{Ali} for a complete review of the atom-bond connectivity index.

Cactus graphs
were first known as Husimi tree, they appeared in the scientific literature more than sixty years ago in papers
by Husimi and Riddell concerned with cluster integrals in the theory of condensation in statistical
mechanics \cite{Hara,Husi2,Ridd}. We refer the reader to \cite{Alikhani1,Chel,Sombor,Moster,Filomat,Sade,Gutindex} for some aspects of parameters of cactus
graphs.

 \medskip
 
 In this paper, we consider the Graovac-Ghorbani and atom-bond connectivity indices of graphs from primary subgraphs. For convenience, the definition of these kind of graphs will be given in the next  section.  In Section 2,  we obtain some lower and upper bounds for Graovac-Ghorbani and atom-bond connectivity indices of graphs from primary subgraphs. In Section 3, we obtain the  Graovac-Ghorbani and atom-bond connectivity indices  of
 families of graphs that are of importance in chemistry.

\section{Main results}

Let $G$ be a connected graph constructed from pairwise disjoint connected graphs
$G_1,\ldots ,G_k$ as follows. Select a vertex of $G_1$, a vertex of $G_2$, and identify these two vertices. Then continue in this manner inductively.  Note that the graph $G$ constructed in this way has a tree-like structure, the $G_i$'s being its building stones (see Figure \ref{Figure1}).

\begin{figure}
	\begin{center}
		\psscalebox{0.6 0.6}
		{
			\begin{pspicture}(0,-4.819607)(13.664668,2.90118)
			\pscircle[linecolor=black, linewidth=0.04, dimen=outer](5.0985146,1.0603933){1.6}
			\pscustom[linecolor=black, linewidth=0.04]
			{
				\newpath
				\moveto(11.898515,0.66039336)
			}
			\pscustom[linecolor=black, linewidth=0.04]
			{
				\newpath
				\moveto(11.898515,0.26039338)
			}
			\pscustom[linecolor=black, linewidth=0.04]
			{
				\newpath
				\moveto(12.698514,0.66039336)
			}
			\pscustom[linecolor=black, linewidth=0.04]
			{
				\newpath
				\moveto(10.298514,1.0603933)
			}
			\pscustom[linecolor=black, linewidth=0.04]
			{
				\newpath
				\moveto(11.098515,-0.9396066)
			}
			\pscustom[linecolor=black, linewidth=0.04]
			{
				\newpath
				\moveto(11.098515,-0.9396066)
			}
			\pscustom[linecolor=black, linewidth=0.04]
			{
				\newpath
				\moveto(11.898515,0.66039336)
			}
			\pscustom[linecolor=black, linewidth=0.04]
			{
				\newpath
				\moveto(11.898515,-0.9396066)
			}
			\pscustom[linecolor=black, linewidth=0.04]
			{
				\newpath
				\moveto(11.898515,-0.9396066)
			}
			\pscustom[linecolor=black, linewidth=0.04]
			{
				\newpath
				\moveto(12.698514,-0.9396066)
			}
			\pscustom[linecolor=black, linewidth=0.04]
			{
				\newpath
				\moveto(12.698514,0.26039338)
			}
			\pscustom[linecolor=black, linewidth=0.04]
			{
				\newpath
				\moveto(14.298514,0.66039336)
				\closepath}
			\psbezier[linecolor=black, linewidth=0.04](11.598515,1.0203934)(12.220886,1.467607)(12.593457,1.262929)(13.268515,1.0203933715820312)(13.943572,0.7778577)(12.308265,0.90039337)(12.224765,0.10039337)(12.141264,-0.69960666)(10.976142,0.5731798)(11.598515,1.0203934)
			\psbezier[linecolor=black, linewidth=0.04](4.8362556,-3.2521083)(4.063277,-2.2959895)(4.6714916,-1.9655427)(4.891483,-0.99004078729821)(5.111474,-0.014538889)(5.3979383,-0.84551746)(5.373531,-1.8452196)(5.349124,-2.8449216)(5.6092343,-4.208227)(4.8362556,-3.2521083)
			\psbezier[linecolor=black, linewidth=0.04](8.198514,-2.0396066)(6.8114076,-1.3924998)(6.844908,-0.93520766)(5.8785143,-1.6996066284179687)(4.9121203,-2.4640057)(5.6385145,-3.4996066)(6.3385143,-2.8396065)(7.0385146,-2.1796067)(9.585621,-2.6867135)(8.198514,-2.0396066)
			\pscircle[linecolor=black, linewidth=0.04, dimen=outer](7.5785146,-3.6396067){1.18}
			\psdots[linecolor=black, dotsize=0.2](11.418514,0.7403934)
			\psdots[linecolor=black, dotsize=0.2](9.618514,1.5003934)
			\psdots[linecolor=black, dotsize=0.2](6.6585145,0.7403934)
			\psdots[linecolor=black, dotsize=0.2](3.5185144,0.96039337)
			\psdots[linecolor=black, dotsize=0.2](5.1185145,-0.51960665)
			\psdots[linecolor=black, dotsize=0.2](5.3985143,-2.5796065)
			\psdots[linecolor=black, dotsize=0.2](7.458514,-2.4596066)
			\rput[bl](8.878514,0.42039338){$G_i$}
			\rput[bl](7.478514,-4.1196065){$G_j$}
			\psbezier[linecolor=black, linewidth=0.04](0.1985144,0.22039337)(0.93261385,0.89943534)(2.1385605,0.6900083)(3.0785143,0.9403933715820313)(4.0184684,1.1907784)(3.248657,0.442929)(2.2785144,0.20039338)(1.3083719,-0.042142253)(-0.53558505,-0.45864862)(0.1985144,0.22039337)
			\psbezier[linecolor=black, linewidth=0.04](2.885918,1.4892112)(1.7389486,2.4304078)(-0.48852357,3.5744174)(0.5524718,2.1502930326916756)(1.5934672,0.7261687)(1.5427756,1.2830372)(2.5062277,1.2429687)(3.46968,1.2029002)(4.0328875,0.5480146)(2.885918,1.4892112)
			\psellipse[linecolor=black, linewidth=0.04, dimen=outer](9.038514,0.7403934)(2.4,0.8)
			\psbezier[linecolor=black, linewidth=0.04](9.399693,1.883719)(9.770389,2.812473)(12.016343,2.7533927)(13.011008,2.856550531577144)(14.005673,2.9597082)(13.727474,2.4925284)(12.761896,2.2324166)(11.796317,1.9723049)(9.028996,0.9549648)(9.399693,1.883719)
			\end{pspicture}
		}
	\end{center}
	\caption{\label{Figure1} A graph with subgraph units  $G_1,\ldots , G_k$.}
\end{figure}
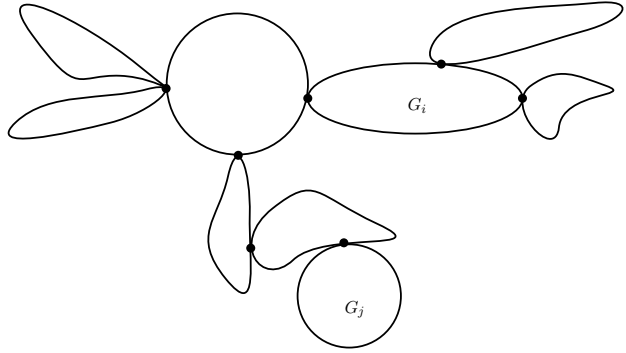

Usually  say that $G$ is obtained by point-attaching from $G_1,\ldots , G_k$ and that $G_i$'s are the primary subgraphs of $G$. A particular case of this construction is the decomposition of a connected graph into blocks (see \cite{Deutsch}). 
We consider some  particular cases  of these graphs  and study their atom-bond connectivity   index. 
As an example of point-attaching graph,   consider the graph $K_m$ and $m$ copies of  $K_n$. By definition, the graph $Q(m, n)$ is obtained by identifying each vertex of $K_m$ with a vertex of a unique $K_n$. The graph $Q(5,4)$ is shown in Figure \ref{qmn}.

	\begin{figure}[!h]\hspace{1.02cm}
		\begin{minipage}{7.5cm}
			\psscalebox{0.45 0.45}
			{
				\begin{pspicture}(0,-13.74)(11.956668,-7.22)
				\pscircle[linecolor=black, linewidth=0.08, dimen=outer](4.1333337,-8.02){0.8}
				\pscircle[linecolor=black, linewidth=0.08, dimen=outer](11.146667,-10.566667){0.8}
				\pscircle[linecolor=black, linewidth=0.08, dimen=outer](7.96,-8.033334){0.8}
				\pscircle[linecolor=black, linewidth=0.08, dimen=outer](0.8000002,-10.54){0.8}
				\psdots[linecolor=black, dotsize=0.4](4.2400002,-8.806666)
				\psdots[linecolor=black, dotsize=0.4](10.386667,-10.486667)
				\psdots[linecolor=black, dotsize=0.4](7.786667,-8.833333)
				\psdots[linecolor=black, dotsize=0.4](1.5600002,-10.526667)
				\rput[bl](5.6933336,-10.76){$\Large{K_m}$}
				\rput[bl](0.44666687,-10.706667){$\large{K_n}$}
				\rput[bl](3.9133337,-8.066667){$\large{K_n}$}
				\rput[bl](7.7000003,-8.08){$\large{K_n}$}
				\rput[bl](10.946667,-10.78){$\large{K_n}$}
				\psellipse[linecolor=black, linewidth=0.08, dimen=outer](5.9866667,-10.426666)(4.4533334,1.78)
				\psdots[linecolor=black, dotsize=0.1](9.533334,-12.246667)
				\psdots[linecolor=black, dotsize=0.1](8.493334,-12.633333)
				\psdots[linecolor=black, dotsize=0.1](7.306667,-12.86)
				\rput[bl](4.1200004,-9.273334){$u_1$}
				\rput[bl](7.4133334,-9.3){$u_2$}
				\rput[bl](1.7733335,-10.726666){$u_m$}
				\psdots[linecolor=black, dotsize=0.1](9.286667,-8.34)
				\psdots[linecolor=black, dotsize=0.1](9.906667,-8.64)
				\psdots[linecolor=black, dotsize=0.1](10.386667,-9.14)
				\rput[bl](9.706667,-10.58){$u_i$}
				\rput[bl](6.066667,-11.96){$u_j$}
				\rput[bl](11.726666,-9.82){v}
				\rput[bl](11.226666,-11.8){w}
				\pscircle[linecolor=black, linewidth=0.08, dimen=outer](6.0,-12.94){0.8}
				\psdots[linecolor=black, dotsize=0.4](5.9866667,-12.14)
				\psdots[linecolor=black, dotsize=0.1](4.046667,-12.84)
				\psdots[linecolor=black, dotsize=0.1](3.0266669,-12.58)
				\psdots[linecolor=black, dotsize=0.1](2.0266669,-12.08)
				\rput[bl](5.766667,-13.24){$K_n$}
				\psdots[linecolor=black, dotsize=0.4](11.566667,-9.96)
				\psdots[linecolor=black, dotsize=0.4](11.286667,-11.34)
				\end{pspicture}
			}
		\end{minipage}
		\hspace{1.02cm}
		\begin{minipage}{7.5cm} 
			\psscalebox{0.45 0.45}
			{
				\begin{pspicture}(0,-4.8)(6.8027782,1.202778)
				\psline[linecolor=black, linewidth=0.08](3.4013891,-0.5986108)(1.8013892,-1.7986108)(2.6013892,-3.3986108)(4.2013893,-3.3986108)(5.001389,-1.7986108)(3.4013891,-0.5986108)(3.4013891,-0.5986108)
				\psline[linecolor=black, linewidth=0.08](3.4013891,-0.5986108)(2.6013892,0.20138916)(3.4013891,1.0013891)(4.2013893,0.20138916)(3.4013891,-0.5986108)(3.4013891,-0.5986108)
				\psline[linecolor=black, linewidth=0.08](3.4013891,1.0013891)(3.4013891,-0.5986108)(3.4013891,-0.5986108)
				\psline[linecolor=black, linewidth=0.08](2.6013892,0.20138916)(4.2013893,0.20138916)(4.2013893,0.20138916)
				\psline[linecolor=black, linewidth=0.08](5.001389,-1.7986108)(5.801389,-0.99861085)(6.601389,-1.7986108)(5.801389,-2.5986109)(5.001389,-1.7986108)(6.601389,-1.7986108)(6.601389,-1.7986108)
				\psline[linecolor=black, linewidth=0.08](5.801389,-0.99861085)(5.801389,-2.5986109)(5.801389,-2.5986109)
				\psline[linecolor=black, linewidth=0.08](4.2013893,-3.3986108)(5.401389,-3.3986108)(5.401389,-4.598611)(4.2013893,-4.598611)(4.2013893,-3.3986108)(5.401389,-4.598611)(5.401389,-4.598611)
				\psline[linecolor=black, linewidth=0.08](5.401389,-3.3986108)(4.2013893,-4.598611)(4.2013893,-4.598611)
				\psline[linecolor=black, linewidth=0.08](2.6013892,-3.3986108)(2.6013892,-4.598611)(1.4013891,-4.598611)(1.4013891,-3.3986108)(2.6013892,-3.3986108)(1.4013891,-4.598611)(1.4013891,-3.3986108)(2.6013892,-4.598611)(2.6013892,-4.598611)
				\psline[linecolor=black, linewidth=0.08](1.8013892,-1.7986108)(1.0013891,-0.99861085)(0.20138916,-1.7986108)(1.0013891,-2.5986109)(1.8013892,-1.7986108)(0.20138916,-1.7986108)(1.0013891,-0.99861085)(1.0013891,-2.5986109)(1.0013891,-2.5986109)
				\psline[linecolor=black, linewidth=0.08](3.4013891,-0.5986108)(2.6013892,-3.3986108)(5.001389,-1.7986108)(1.8013892,-1.7986108)(4.2013893,-3.3986108)(3.4013891,-0.5986108)(3.4013891,-0.5986108)
				\psdots[linecolor=black, dotstyle=o, dotsize=0.4, fillcolor=white](3.4013891,-0.5986108)
				\psdots[linecolor=black, dotstyle=o, dotsize=0.4, fillcolor=white](4.2013893,0.20138916)
				\psdots[linecolor=black, dotstyle=o, dotsize=0.4, fillcolor=white](3.4013891,1.0013891)
				\psdots[linecolor=black, dotstyle=o, dotsize=0.4, fillcolor=white](2.6013892,0.20138916)
				\psdots[linecolor=black, dotstyle=o, dotsize=0.4, fillcolor=white](5.801389,-0.99861085)
				\psdots[linecolor=black, dotstyle=o, dotsize=0.4, fillcolor=white](5.001389,-1.7986108)
				\psdots[linecolor=black, dotstyle=o, dotsize=0.4, fillcolor=white](5.801389,-2.5986109)
				\psdots[linecolor=black, dotstyle=o, dotsize=0.4, fillcolor=white](6.601389,-1.7986108)
				\psdots[linecolor=black, dotstyle=o, dotsize=0.4, fillcolor=white](1.8013892,-1.7986108)
				\psdots[linecolor=black, dotstyle=o, dotsize=0.4, fillcolor=white](1.0013891,-2.5986109)
				\psdots[linecolor=black, dotstyle=o, dotsize=0.4, fillcolor=white](0.20138916,-1.7986108)
				\psdots[linecolor=black, dotstyle=o, dotsize=0.4, fillcolor=white](1.0013891,-0.99861085)
				\psdots[linecolor=black, dotstyle=o, dotsize=0.4, fillcolor=white](1.4013891,-3.3986108)
				\psdots[linecolor=black, dotstyle=o, dotsize=0.4, fillcolor=white](1.4013891,-4.598611)
				\psdots[linecolor=black, dotstyle=o, dotsize=0.4, fillcolor=white](2.6013892,-3.3986108)
				\psdots[linecolor=black, dotstyle=o, dotsize=0.4, fillcolor=white](2.6013892,-4.598611)
				\psdots[linecolor=black, dotstyle=o, dotsize=0.4, fillcolor=white](4.2013893,-3.3986108)
				\psdots[linecolor=black, dotstyle=o, dotsize=0.4, fillcolor=white](5.401389,-3.3986108)
				\psdots[linecolor=black, dotstyle=o, dotsize=0.4, fillcolor=white](5.401389,-4.598611)
				\psdots[linecolor=black, dotstyle=o, dotsize=0.4, fillcolor=white](4.2013893,-4.598611)
				\end{pspicture}
			}
		\end{minipage}
		\caption{The graph $Q(m,n)$ and $Q(5,4)$, respectively. }\label{qmn}
	\end{figure}
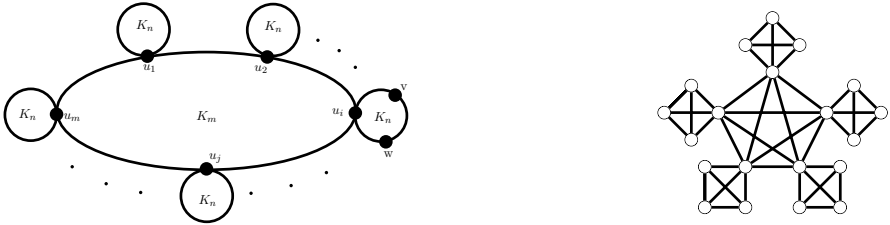

	\begin{theorem}
	For the graph $Q(m,n)$ (see Figure \ref{qmn}), and $n\geq 2$ we have:
	
\begin{itemize}
\item[(i)]
\begin{align*}
ABC(Q(m,n))&=\frac{m(m-1)}{2(m+n-2)}\sqrt{2(m+n-3)}+m(\frac{n}{2}-1)\sqrt{2(n-2)}\\
	&\quad +m(n-1)\sqrt{\frac{m+2n-5}{n^2+mn-m-3n+2}}.
	\end{align*}
\item[(ii)]	
		$$ABC_{GG}(Q(m,n))=\frac{m(m-1)}{2n}\sqrt{2n-2}+m(n-1)\sqrt{\frac{n(m-1)}{n(m-1)+1}}.$$	
\end{itemize}

	\end{theorem}

	\begin{proof}
	\begin{itemize}
\item[(i)]		There are $\frac{m(m-1)}{2}$ edges with endpoints of degree $m+n-2$. Also there are $m(n-1)$ edges with endpoints of degree $m+n-2$ and $n-1$, and there are $m(n-1)(\frac{n}{2}-1)$ edges with endpoints of degree $n-1$. Therefore 
	\begin{align*}
	ABC(Q(m,n))&=\frac{m(m-1)}{2}\sqrt{\frac{(m+n-2)+(m+n-2)-2}{(m+n-2)(m+n-2)}}\\
	&\quad+m(n-1)\sqrt{\frac{(m+n-2)+(n-1)-2}{(m+n-2)(n-1)}}\\
	&\quad +m(n-1)(\frac{n}{2}-1)\sqrt{\frac{(n-1)+(n-1)-2}{(n-1)(n-1)}},
	\end{align*}
	and we have the result.	

\item[(ii)]	
		First consider the edge $u_iu_j$ in $K_m$.  There are $n$ vertices which are closer to $u_i$ than $u_j$ (including $u_i$ itself), also there are  $n$ vertices closer to $u_j$ than $u_i$, and there are $\frac{m(m-1)}{2}$ edges like $u_iu_j$ in $Q(m,n)$. Now consider the edge $vw$ in the $i$-th $K_n$. There is one vertex which is closer to $v$ than $w$ and that is $v$ itself, and visa versa. Finally, consider the edge $u_iv$ in the $i$-th $K_n$. There are $n(m-1)+1$ vertices which are closer to $u_i$ than $v$ (including $u_i$), also there is  one vertex closer to $v$ than $u_i$ which is $v$, and there are $m(n-1)$ edges like $u_iv$ in $Q(m,n)$. Therefore we have the result.	\qed
	\end{itemize}
	\end{proof}

\subsection{Upper bounds}

By the definition of the atom-bond connectivity and Graovac-Ghorbani indices,  we have the  following easy result:

\begin{proposition}\label{pro-disconnect}
	Let $G$ be a disconnected graph with components $G_1$ and $G_2$. Then
	\begin{itemize}
	\item[(i)]
		$$ABC(G)=ABC(G_1)+ABC(G_2)$$
	\item[(ii)]
		$$ABC_{GG}(G)=ABC_{GG}(G_1)+ABC_{GG}(G_2)$$
	\end{itemize}
\end{proposition}

Now we  examine the effects on $ABC(G)$ and $ABC_{GG}(G)$ when $G$ is modified by deleting an  edge or vertex of $G$. 
	\begin{theorem}\label{pro(G-e)}
		Let $G=(V,E)$ be a graph and $e=uv\in E$ which is not a pendant edge. Also let $d_u$ be the degree of vertex $u$ in $G$, and  $n_u$ be  the number of vertices of $G$ closer to $u$ than to $v$. Then,
		\begin{itemize}
		\item[(i)]
				\begin{align*}
		ABC(G-e) \geq ABC(G) - max\{\frac{\sqrt{2d_u-2}}{d_v},\frac{\sqrt{2d_v-2}}{d_u}\}.
		\end{align*}
		\item[(ii)]
			\begin{align*}
		ABC_{GG}(G-e) \geq ABC_{GG}(G) - max\{\frac{\sqrt{2n_u-2}}{n_v},\frac{\sqrt{2n_v-2}}{n_u}\}.
		\end{align*}	
		\end{itemize}
	\end{theorem}
	
	\begin{proof}
	\begin{itemize}
		\item[(i)]
		First we remove edge $e$ and find $ABC(G-e)$. For every integer $a,b\geq 2$, we have $\sqrt{\frac{a+(b-1)-2}{a(b-1)}}\geq \sqrt{\frac{a+b-2}{ab}}$.
		Now Obviously, by adding edge $e$ to $G-e$ and $\sqrt{\frac{d_u+d_v-2}{d_ud_v}}$ to $ABC(G-e)$, then $ABC(G)$ is less than that or equal to it. So
		\begin{align*}
		ABC(G) &\leq ABC(G-e) + \sqrt{\frac{d_u+d_v-2}{d_ud_v}}\\
		&\leq ABC(G-e) + max\{\sqrt{\frac{d_u+d_u-2}{d_vd_v}},\sqrt{\frac{d_v+d_v-2}{d_ud_u}}\}\\
		&=ABC(G-e) + max\{\frac{\sqrt{2d_u-2}}{d_v},\frac{\sqrt{2d_v-2}}{d_u}\}, 
		\end{align*}
		and therefore we have the result.
		\item[(ii)] The proof is similar to Part (i). 		\qed
		\end{itemize}
	\end{proof}

By the same argument as the proof of Theorem \ref{pro(G-e)}, and deleting a vertex at the first step, we have:

	\begin{theorem}
		Let $G=(V,E)$ be a graph and $v\in V$. Also let $d_u$ be the degree of vertex $u$ in $G$. Then,
			\begin{itemize}
	\item[(i)]
		$$ABC(G-v) \geq ABC(G) - \sum_{uv\in E}max\{\frac{\sqrt{2d_u-2}}{d_v},\frac{\sqrt{2d_v-2}}{d_u}\}. $$
	\item[(ii)] 
	$$ABC_{GG}(G-v) \geq ABC_{GG}(G) - \sum_{uv\in E}max\{\frac{\sqrt{2n_u-2}}{n_v},\frac{\sqrt{2n_v-2}}{n_u}\}. $$
		\end{itemize}
	\end{theorem}

Here we study some  bounds on the atom-bond connectivity and Graovac-Ghorbani indices for links of graphs and circuits of
graphs.

	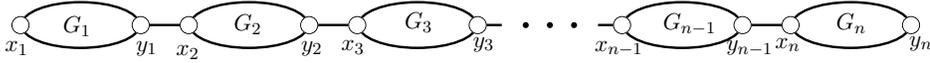
\begin{figure}
		\begin{center}
			\psscalebox{0.8 0.8}
			{
				\begin{pspicture}(0,-4.08)(15.436667,-3.12)
				\psellipse[linecolor=black, linewidth=0.04, dimen=outer](1.2533334,-3.52)(1.0,0.4)
				\psellipse[linecolor=black, linewidth=0.04, dimen=outer](4.0533333,-3.52)(1.0,0.4)
				\psellipse[linecolor=black, linewidth=0.04, dimen=outer](6.853334,-3.52)(1.0,0.4)
				\psellipse[linecolor=black, linewidth=0.04, dimen=outer](11.253334,-3.52)(1.0,0.4)
				\psellipse[linecolor=black, linewidth=0.04, dimen=outer](14.053333,-3.52)(1.0,0.4)
				\psline[linecolor=black, linewidth=0.04](2.2533333,-3.52)(3.0533335,-3.52)(3.0533335,-3.52)
				\psline[linecolor=black, linewidth=0.04](5.0533333,-3.52)(5.8533335,-3.52)(5.8533335,-3.52)
				\psline[linecolor=black, linewidth=0.04](12.253333,-3.52)(13.053333,-3.52)(13.053333,-3.52)
				\psdots[linecolor=black, dotstyle=o, dotsize=0.3, fillcolor=white](2.2533333,-3.52)
				\psdots[linecolor=black, dotstyle=o, dotsize=0.3, fillcolor=white](0.25333345,-3.52)
				\psdots[linecolor=black, dotstyle=o, dotsize=0.3, fillcolor=white](3.0533335,-3.52)
				\psdots[linecolor=black, dotstyle=o, dotsize=0.3, fillcolor=white](5.0533333,-3.52)
				\psdots[linecolor=black, dotstyle=o, dotsize=0.3, fillcolor=white](5.8533335,-3.52)
				\psdots[linecolor=black, dotstyle=o, dotsize=0.3, fillcolor=white](12.253333,-3.52)
				\psdots[linecolor=black, dotstyle=o, dotsize=0.3, fillcolor=white](13.053333,-3.52)
				\psdots[linecolor=black, dotstyle=o, dotsize=0.3, fillcolor=white](15.053333,-3.52)
				\rput[bl](0.0,-4.0133333){$x_1$}
				\rput[bl](2.8400002,-4.08){$x_2$}
				\rput[bl](5.5866666,-4.0){$x_3$}
				\rput[bl](2.1733334,-4.0266666){$y_1$}
				\rput[bl](4.92,-4.0266666){$y_2$}
				\rput[bl](7.7733335,-3.96){$y_3$}
				\rput[bl](0.9600001,-3.7066667){$G_1$}
				\rput[bl](3.8000002,-3.6666667){$G_2$}
				\rput[bl](6.64,-3.64){$G_3$}
				\psline[linecolor=black, linewidth=0.04](8.253333,-3.52)(7.8533335,-3.52)(7.8533335,-3.52)
				\psline[linecolor=black, linewidth=0.04](9.853333,-3.52)(10.253333,-3.52)(10.253333,-3.52)
				\psdots[linecolor=black, dotstyle=o, dotsize=0.3, fillcolor=white](7.8533335,-3.52)
				\psdots[linecolor=black, dotstyle=o, dotsize=0.3, fillcolor=white](10.253333,-3.52)
				\psdots[linecolor=black, dotsize=0.1](8.653334,-3.52)
				\psdots[linecolor=black, dotsize=0.1](9.053333,-3.52)
				\psdots[linecolor=black, dotsize=0.1](9.453334,-3.52)
				\rput[bl](12.8,-3.96){$x_n$}
				\rput[bl](15.026667,-3.9466667){$y_n$}
				\rput[bl](11.986667,-4.0266666){$y_{n-1}$}
				\rput[bl](10.933333,-3.68){$G_{n-1}$}
				\rput[bl](9.8,-4.04){$x_{n-1}$}
				\rput[bl](13.8133335,-3.6533334){$G_n$}
				\end{pspicture}
			}
		\end{center}
	\caption{Link of $n$ graphs $G_1,G_2, \ldots , G_n$.} \label{link-n}
	\end{figure}

\begin{theorem} \label{thm-link}
	Let $G_1,G_2, \ldots , G_k$ be a finite sequence of pairwise disjoint connected graphs and let
	$x_i,y_i \in V(G_i)$. Let $G$ be the link of graphs $\{G_i\}_{i=1}^k$ with respect to the vertices $\{x_i, y_i\}_{i=1}^k$ (see Figure \ref{link-n}) and suppose that $G_i\neq K_1 $. Then,
	\begin{itemize}
	\item[(i)]
		$$ABC(G)\leq \sum_{i=1}^{k}ABC(G_i)+\sum_{i=1}^{k-1}max\{\frac{\sqrt{2d_{x_{i+1}}-2}}{d_{y_i}},\frac{\sqrt{2d_{y_i}-2}}{d_{x_{i+1}}}\}.$$
			\item[(ii)]
		$$ABC_{GG}(G)\leq \sum_{i=1}^{k}ABC_{GG}(G_i)+\sum_{i=1}^{k-1}max\{\frac{\sqrt{2n_{x_{i+1}}-2}}{n_{y_i}},\frac{\sqrt{2n_{y_i}-2}}{n_{x_{i+1}}}\}.$$
	\end{itemize}

\end{theorem}

\begin{proof}
	\begin{itemize}
	\item[(i)]
	First we remove edge $y_1x_2$ (Figure \ref{link-n}). By Theorem \ref{pro(G-e)}, we have
	$$ABC(G) \leq ABC(G-y_1x_2) + max\{\frac{\sqrt{2d_{x_{2}}-2}}{d_{y_1}},\frac{\sqrt{2d_{y_1}-2}}{d_{x_{2}}}\}.$$
	Let $G^{\prime}$ be the link graph related to graphs $\{G_i\}_{i=2}^k$ with respect to the vertices $\{x_i, y_i\}_{i=2}^k$. Then by Proposition \ref{pro-disconnect} we have,
	$$ABC(G-y_1x_2)=ABC(G_1)+ABC(G^{\prime}),$$
	and therefore,
	$$ABC(G) \leq ABC(G_1)+ABC(G^{\prime}) + max\{\frac{\sqrt{2d_{x_{2}}-2}}{d_{y_1}},\frac{\sqrt{2d_{y_1}-2}}{d_{x_{2}}}\}.$$
	By continuing this process, we have the result.
	\item[(ii)] The proof is similar to Part (i). 		\qed
	\end{itemize}
\end{proof}

	\begin{figure}
		\begin{center}
			\psscalebox{0.85 0.85}
			{
				\begin{pspicture}(0,-7.6)(5.6,-2.0)
				\rput[bl](2.6533334,-4.48){$x_1$}
				\rput[bl](3.0533333,-4.92){$x_2$}
				\rput[bl](2.5733333,-5.4266667){$x_3$}
				\rput[bl](2.6,-3.1733334){$G_1$}
				\rput[bl](4.2933335,-4.9866667){$G_2$}
				\rput[bl](2.6133332,-6.7733335){$G_3$}
				\rput[bl](2.1733334,-4.9466667){$x_n$}
				\rput[bl](0.73333335,-4.9333334){$G_n$}
				\psellipse[linecolor=black, linewidth=0.04, dimen=outer](1.0,-4.8)(1.0,0.4)
				\psellipse[linecolor=black, linewidth=0.04, dimen=outer](4.6,-4.8)(1.0,0.4)
				\psellipse[linecolor=black, linewidth=0.04, dimen=outer](2.8,-3.0)(0.4,1.0)
				\psellipse[linecolor=black, linewidth=0.04, dimen=outer](2.8,-6.6)(0.4,1.0)
				\psline[linecolor=black, linewidth=0.04](2.0,-4.8)(2.8,-4.0)(3.6,-4.8)(2.8,-5.6)(2.8,-5.6)
				\psdots[linecolor=black, fillstyle=solid, dotstyle=o, dotsize=0.3, fillcolor=white](2.8,-4.0)
				\psdots[linecolor=black, fillstyle=solid, dotstyle=o, dotsize=0.3, fillcolor=white](3.6,-4.8)
				\psline[linecolor=black, linewidth=0.04, linestyle=dotted, dotsep=0.10583334cm](2.8,-5.6)(2.0,-4.8)(2.0,-4.8)
				\psdots[linecolor=black, fillstyle=solid, dotstyle=o, dotsize=0.3, fillcolor=white](2.0,-4.8)
				\psdots[linecolor=black, fillstyle=solid, dotstyle=o, dotsize=0.3, fillcolor=white](2.8,-5.6)
				\end{pspicture}
			}
		\end{center}
	\caption{Circuit of $n$ graphs $G_1,G_2, \ldots , G_n$.} \label{circuit-n}
	\end{figure}
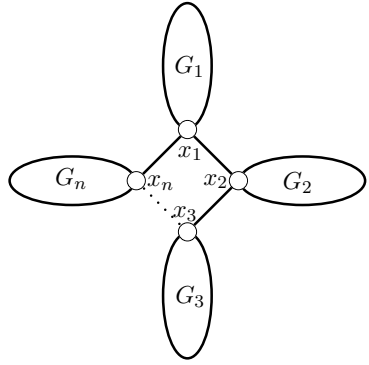

\begin{theorem} 
	Let $G_1,G_2, \ldots , G_k$ be a finite sequence of pairwise disjoint connected graphs and let
	$x_i \in V(G_i)$. Let $G$ be the circuit of graphs $\{G_i\}_{i=1}^k$ with respect to the vertices $\{x_i\}_{i=1}^k$ and obtained by identifying the vertex $x_i$ of the graph $G_i$ with the $i$-th vertex of the
	cycle graph $C_k$ (Figure \ref{circuit-n}) and suppose that $G_i\neq K_1 $. Then,
	\begin{itemize}
	\item[(i)]
		\begin{align*}
	ABC(G)& \leq max\{\frac{\sqrt{2d_{x_{1}}-2}}{d_{x_n}},\frac{\sqrt{2d_{x_n}-2}}{d_{x_{1}}}\}+\sum_{i=1}^{k}ABC(G_i)\\
	 &\quad+\sum_{i=1}^{k-1}max\{\frac{\sqrt{2d_{x_{i+1}}-2}}{d_{x_i}},\frac{\sqrt{2d_{x_i}-2}}{d_{x_{i+1}}}\}
	\end{align*}
		\item[(ii)]
		\begin{align*}
	ABC_{GG}(G)& \leq max\{\frac{\sqrt{2n_{x_{1}}-2}}{n_{x_n}},\frac{\sqrt{2n_{x_n}-2}}{n_{x_{1}}}\}+\sum_{i=1}^{k}ABC_{GG}(G_i)\\
	 &\quad+\sum_{i=1}^{k-1}max\{\frac{\sqrt{2n_{x_{i+1}}-2}}{n_{x_i}},\frac{\sqrt{2n_{x_i}-2}}{n_{x_{i+1}}}\}
	\end{align*}
	\end{itemize}
\end{theorem}

\begin{proof}
	\begin{itemize}
	\item[(i)]
	First we remove edge $x_nx_1$ (Figure \ref{circuit-n}). By Theorem \ref{pro(G-e)}, we have
	$$ABC(G) \leq ABC(G-x_nx_1) + max\{\frac{\sqrt{2d_{x_{1}}-2}}{d_{x_n}},\frac{\sqrt{2d_{x_n}-2}}{d_{x_{1}}}\}.$$
	Now we remove edge $x_1x_2$. Then,
	\begin{align*}
	ABC(G) &\leq ABC(G-\{x_nx_1,x_1x_2\}) +  max\{\frac{\sqrt{2d_{x_{1}}-2}}{d_{x_n}},\frac{\sqrt{2d_{x_n}-2}}{d_{x_{1}}}\}\\
	 &\quad + max\{\frac{\sqrt{2d_{x_{2}}-2}}{d_{x_1}},\frac{\sqrt{2d_{x_1}-2}}{d_{x_{2}}}\}.
	\end{align*}
	Let $G^{\prime}$ be the graph related to circuit graph with $\{G_i\}_{i=2}^k$ with respect to the vertices $\{x_i\}_{i=2}^k$ and removing the edge $x_nx_1$. Then by Proposition \ref{pro-disconnect} we have,
	$$ABC(G-\{x_nx_1,x_1x_2\})=ABC(G_1)+ABC(G^{\prime}),$$
	and therefore,
	\begin{align*}
	ABC(G) &\leq ABC(G_1)+ABC(G^{\prime}) + max\{\frac{\sqrt{2d_{x_{1}}-2}}{d_{x_n}},\frac{\sqrt{2d_{x_n}-2}}{d_{x_{1}}}\}\\
	 &\quad + max\{\frac{\sqrt{2d_{x_{2}}-2}}{d_{x_1}},\frac{\sqrt{2d_{x_1}-2}}{d_{x_{2}}}\}.
	\end{align*}
	By continuing this process, we have the result.
	\item[(ii)] The proof is similar to Part (i). 		\qed
	\end{itemize}
\end{proof}

\subsection{Some other Upper bounds for the  Graovac-Ghorbani index}

	In this subsection, we consider some special graphs from primary subgraphs and present upper bounds for the  Graovac-Ghorbani index of them.  	
	The following theorem is about the link of graphs.

	\begin{theorem}\label{thm-link-GG}
Let $G_1,G_2, \ldots , G_k$ be a finite sequence of pairwise disjoint connected graphs and let
	$x_i,y_i \in V(G_i)$. Let $G$ be the link of graphs $\{G_i\}_{i=1}^k$ with respect to the vertices $\{x_i, y_i\}_{i=1}^k$ (see Figure \ref{link-n}). Then,
		
		\begin{align*}
ABC_{GG}(G)&<\Big(|E(G)|-(n-1)\Big) + \sum_{i=1}^{n}ABC_{GG}(G_i)\\
&\quad +\sum_{i=1}^{n-1}\sqrt{\frac{|V(G)|-2}{\sum_{t=1}^{i}|V(G_t)|  \sum_{t=i+1}^{n}|V(G_t)|}}.\\
		\end{align*}
	\end{theorem}

	\begin{proof}
Consider the graph $G_i$ (Figure \ref{link-n}) and let  $n_u'(uv,G_i)$ be the number of vertices of $G_i$ closer to $u$ than $v$ in $G_i$, Also let  $n_u(uv,G_i)$ be the number of vertices of $G$ closer to $u$ than $v$ in $G$. By the definition of Graovac-Ghorbani index, we have:

		\begin{align*}
		ABC_{GG}(G)&= \sum_{uv\in E(G)}^{}\sqrt{\frac{n_u(uv,G)+n_v(uv,G)-2}{n_u(uv,G)n_v(uv,G)}}\\
		&= \sum_{i=1}^{n}\sum_{uv\in E(G_i)}^{}\sqrt{\frac{n_u(uv,G_i)+n_v(uv,G_i)-2}{n_u(uv,G_i)n_v(uv,G_i)}}\\
		\end{align*}

		\begin{align*}
				&\quad +\sum_{i=1}^{n-1}\sum_{y_ix_{i+1}\in E(G)}^{}\sqrt{\frac{n_{y_i}(y_ix_{i+1},G)+n_{x_{i+1}}(y_ix_{i+1},G)-2}{n_{y_i}(y_ix_{i+1},G)n_{x_{i+1}}(y_ix_{i+1},G)}}\\
		&=\sum_{i=1}^{n}\sum_{uv\in E(G_i), d(u,x_i)<d(v,x_i),  d(u,y_i)<d(v,y_i)}^{}\sqrt{\frac{n_u(uv,G_i)+n_v(uv,G_i)-2}{n_u(uv,G_i)n_v(uv,G_i)}}\\
		&\quad +\sum_{i=1}^{n}\sum_{uv\in E(G_i),d(u,x_i)<d(v,x_i), d(v,y_i)<d(u,y_i)}^{}\sqrt{\frac{n_u(uv,G_i)+n_v(uv,G_i)-2}{n_u(uv,G_i)n_v(uv,G_i)}}\\
		&\quad +\sum_{i=1}^{n}\sum_{uv\in E(G_i), d(u,x_i)<d(v,x_i),  d(u,y_i)=d(v,y_i)}^{}\sqrt{\frac{n_u(uv,G_i)+n_v(uv,G_i)-2}{n_u(uv,G_i)n_v(uv,G_i)}}\\
		&\quad +\sum_{i=1}^{n}\sum_{uv\in E(G_i), d(u,x_i)=d(v,x_i),  d(u,y_i)<d(v,y_i)}^{}\sqrt{\frac{n_u(uv,G_i)+n_v(uv,G_i)-2}{n_u(uv,G_i)n_v(uv,G_i)}}\\
		&\quad +\sum_{i=1}^{n}\sum_{uv\in E(G_i), d(u,x_i)=d(v,x_i),  d(u,y_i)=d(v,y_i)}^{}\sqrt{\frac{n_u(uv,G_i)+n_v(uv,G_i)-2}{n_u(uv,G_i)n_v(uv,G_i)}}\\
		&\quad +\sum_{i=1}^{n-1}\sum_{y_ix_{i+1}\in E(G)}^{}\sqrt{\frac{n_{y_i}(y_ix_{i+1},G)+n_{x_{i+1}}(y_ix_{i+1},G)-2}{n_{y_i}(y_ix_{i+1},G)n_{x_{i+1}}(y_ix_{i+1},G)}}\\
		&=\sum_{i=1}^{n}\sum_{uv\in E(G_i), d(u,x_i)<d(v,x_i),  d(u,y_i)<d(v,y_i)}^{}\\
		&\quad\quad\quad\quad\sqrt{\frac{n_u'(uv,G_i)+V(G)-V(G_i)+n_v'(uv,G_i)-2}{\Big( n_u'(uv,G_i)+V(G)-V(G_i)\Big) n_v'(uv,G_i)}}\\
		&\quad +\sum_{i=1}^{n}\sum_{uv\in E(G_i),d(u,x_i)<d(v,x_i), d(v,y_i)<d(u,y_i)}^{} \\
		&\quad\quad\quad\quad \sqrt{\frac{ n_u'(uv,G_i)+\sum_{t=1}^{i}|V(G_t)|+n_v'(uv,G_i)+\sum_{t=i+1}^{n}|V(G_t)|-2}{ \Big(n_u'(uv,G_i)+\sum_{t=1}^{i}|V(G_t)|\Big)\Big(n_v'(uv,G_i)-\sum_{t=i+1}^{n}|V(G_t)|\Big)}}\\
		&\quad +\sum_{i=1}^{n}\sum_{uv\in E(G_i), d(u,x_i)<d(v,x_i),  d(u,y_i)=d(v,y_i)}^{}\sqrt{\frac{n_u'(uv,G_i)+\sum_{t=1}^{i}|V(G_t)|+n_v'(uv,G_i)-2}{\Big(n_u'(uv,G_i)+\sum_{t=1}^{i}|V(G_t)|\Big)n_v'(uv,G_i)}}\\
		&\quad +\sum_{i=1}^{n}\sum_{uv\in E(G_i), d(u,x_i)=d(v,x_i),  d(u,y_i)<d(v,y_i)}^{}\sqrt{\frac{n_u'(uv,G_i)+n_v'(uv,G_i)+\sum_{t=i+1}^{n}|V(G_t)|-2}{n_u'(uv,G_i)\Big(n_v'(uv,G_i)+\sum_{t=i+1}^{n}|V(G_t)|\Big)}}\\
		&\quad +\sum_{i=1}^{n}\sum_{uv\in E(G_i), d(u,x_i)=d(v,x_i),  d(u,y_i)=d(v,y_i)}^{}\sqrt{\frac{n_u'(uv,G_i)+n_v'(uv,G_i)-2}{n_u'(uv,G_i)n_v'(uv,G_i)}}\\
		\end{align*}
		
		\begin{align*}		
			&\quad +\sum_{i=1}^{n-1}\sqrt{\frac{\sum_{t=1}^{i}|V(G_t)| + \sum_{t=i+1}^{n}|V(G_t)|-2}{\sum_{t=1}^{i}|V(G_t)|  \sum_{t=i+1}^{n}|V(G_t)|}}.\\	
		\end{align*}
		
Since for every non negative integers $a,b$ and $c$, we have: 

$$\sqrt{\frac{a+c+b-2}{(a+c)b}}<\sqrt{\frac{a+b-2}{ab}}+1,$$		

then,

		\begin{align*}		
		ABC_{GG}(G) &< \sum_{i=1}^{n}\sum_{uv\in E(G_i), d(u,x_i)<d(v,x_i),  d(u,y_i)<d(v,y_i)}^{} \sqrt{\frac{n_u'(uv,G_i)+n_v'(uv,G_i)-2}{n_u'(uv,G_i)n_v'(uv,G_i)}}+1\\
		&\quad +\sum_{i=1}^{n}\sum_{uv\in E(G_i),d(u,x_i)<d(v,x_i), d(v,y_i)<d(u,y_i)}^{} \sqrt{\frac{n_u'(uv,G_i)+n_v'(uv,G_i)-2}{n_u'(uv,G_i)n_v'(uv,G_i)}}+1\\
		&\quad +\sum_{i=1}^{n}\sum_{uv\in E(G_i), d(u,x_i)<d(v,x_i),  d(u,y_i)=d(v,y_i)}^{} \sqrt{\frac{n_u'(uv,G_i)+n_v'(uv,G_i)-2}{n_u'(uv,G_i)n_v'(uv,G_i)}}+1\\
		&\quad +\sum_{i=1}^{n}\sum_{uv\in E(G_i), d(u,x_i)=d(v,x_i),  d(u,y_i)<d(v,y_i)}^{} \sqrt{\frac{n_u'(uv,G_i)+n_v'(uv,G_i)-2}{n_u'(uv,G_i)n_v'(uv,G_i)}}+1\\
		&\quad +\sum_{i=1}^{n}\sum_{uv\in E(G_i), d(u,x_i)=d(v,x_i),  d(u,y_i)=d(v,y_i)}^{}\sqrt{\frac{n_u'(uv,G_i)+n_v'(uv,G_i)-2}{n_u'(uv,G_i)n_v'(uv,G_i)}}+1\\
		&\quad +\sum_{i=1}^{n-1}\sqrt{\frac{\sum_{t=1}^{i}|V(G_t)| + \sum_{t=i+1}^{n}|V(G_t)|-2}{\sum_{t=1}^{i}|V(G_t)|  \sum_{t=i+1}^{n}|V(G_t)|}}\\
		&= \Big(|E(G)|-(n-1)\Big) + \sum_{i=1}^{n}ABC_{GG}(G_i)\\
		&\quad +\sum_{i=1}^{n-1}\sqrt{\frac{\sum_{t=1}^{i}|V(G_t)| + \sum_{t=i+1}^{n}|V(G_t)|-2}{\sum_{t=1}^{i}|V(G_t)|  \sum_{t=i+1}^{n}|V(G_t)|}},
		\end{align*}
and therefore we have the result.
	\qed
	\end{proof}

	By the same argument similar to  the proof of the Theorem \ref{thm-link-GG}, we have the following theorem which is about the chain of graphs:
	
	\begin{theorem} 
		Let $G_1,G_2, \ldots , G_n$ be a finite sequence of pairwise disjoint connected graphs and let
		$x_i,y_i \in V(G_i)$. Let $C(G_1,...,G_n)$ be the chain of graphs $\{G_i\}_{i=1}^n$ with respect to the vertices $\{x_i, y_i\}_{i=1}^k$ which obtained by identifying the vertex $y_i$ with the vertex $x_{i+1}$ for $i=1,2,\ldots,n-1$ (Figure \ref{chain-n}). Then,

		\begin{align*}
				ABC_{GG}(C(G_1,...,G_n))&<|E(G)| + \sum_{i=1}^{n}ABC_{GG}(G_i)\\
&\quad +\sum_{i=1}^{n-1}\sqrt{\frac{|V(G)|-2}{\sum_{t=1}^{i}|V(G_t)|  \sum_{t=i+1}^{n}|V(G_t)|}}.\\
		\end{align*}
	\end{theorem}

	\begin{figure}
		\begin{center}
			\psscalebox{0.7 0.7}
			{
				\begin{pspicture}(0,-3.9483333)(12.236668,-2.8316667)
				\psellipse[linecolor=black, linewidth=0.04, dimen=outer](1.2533334,-3.4416668)(1.0,0.4)
				\psellipse[linecolor=black, linewidth=0.04, dimen=outer](3.2533333,-3.4416668)(1.0,0.4)
				\psellipse[linecolor=black, linewidth=0.04, dimen=outer](5.2533336,-3.4416668)(1.0,0.4)
				\psellipse[linecolor=black, linewidth=0.04, dimen=outer](8.853333,-3.4416668)(1.0,0.4)
				\psellipse[linecolor=black, linewidth=0.04, dimen=outer](10.853333,-3.4416668)(1.0,0.4)
				\psdots[linecolor=black, fillstyle=solid, dotstyle=o, dotsize=0.3, fillcolor=white](2.2533333,-3.4416666)
				\psdots[linecolor=black, fillstyle=solid, dotstyle=o, dotsize=0.3, fillcolor=white](0.25333345,-3.4416666)
				\psdots[linecolor=black, fillstyle=solid, dotstyle=o, dotsize=0.3, fillcolor=white](2.2533333,-3.4416666)
				\psdots[linecolor=black, fillstyle=solid, dotstyle=o, dotsize=0.3, fillcolor=white](4.2533336,-3.4416666)
				\psdots[linecolor=black, fillstyle=solid, dotstyle=o, dotsize=0.3, fillcolor=white](4.2533336,-3.4416666)
				\psdots[linecolor=black, fillstyle=solid, dotstyle=o, dotsize=0.3, fillcolor=white](9.853333,-3.4416666)
				\psdots[linecolor=black, fillstyle=solid, dotstyle=o, dotsize=0.3, fillcolor=white](9.853333,-3.4416666)
				\psdots[linecolor=black, fillstyle=solid, dotstyle=o, dotsize=0.3, fillcolor=white](11.853333,-3.4416666)
				\rput[bl](0.0,-3.135){$x_1$}
				\rput[bl](2.0400002,-3.2016668){$x_2$}
				\rput[bl](3.9866667,-3.1216667){$x_3$}
				\rput[bl](2.1733334,-3.9483335){$y_1$}
				\rput[bl](4.12,-3.9483335){$y_2$}
				\rput[bl](6.1733336,-3.8816667){$y_3$}
				\rput[bl](0.9600001,-3.6283333){$G_1$}
				\rput[bl](3.0,-3.5883334){$G_2$}
				\rput[bl](5.04,-3.5616667){$G_3$}
				\psdots[linecolor=black, fillstyle=solid, dotstyle=o, dotsize=0.3, fillcolor=white](6.2533336,-3.4416666)
				\psdots[linecolor=black, fillstyle=solid, dotstyle=o, dotsize=0.3, fillcolor=white](7.8533335,-3.4416666)
				\psdots[linecolor=black, dotsize=0.1](6.6533337,-3.4416666)
				\psdots[linecolor=black, dotsize=0.1](7.0533333,-3.4416666)
				\psdots[linecolor=black, dotsize=0.1](7.4533334,-3.4416666)
				\rput[bl](9.6,-3.0816667){$x_n$}
				\rput[bl](11.826667,-3.8683333){$y_n$}
				\rput[bl](9.586667,-3.9483335){$y_{n-1}$}
				\rput[bl](8.533334,-3.6016667){$G_{n-1}$}
				\rput[bl](7.4,-3.1616666){$x_{n-1}$}
				\rput[bl](10.613334,-3.575){$G_n$}
				\end{pspicture}
			}
		\end{center}
		\caption{Chain of $n$ graphs $G_1,G_2, \ldots , G_n$.} \label{chain-n}
	\end{figure}
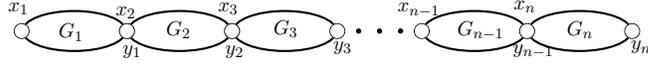

	With similar argument to the proof of the Theorem \ref{thm-link-GG}, we have the following theorem which is about the bouquet of graphs:

	\begin{theorem} 
		Let $G_1,G_2, \ldots , G_n$ be a finite sequence of pairwise disjoint connected graphs and let
		$x_i \in V(G_i)$. Let $B(G_1,...,G_n)$ be the bouquet of graphs $\{G_i\}_{i=1}^n$ with respect to the vertices $\{x_i\}_{i=1}^n$ and obtained by identifying the vertex $x_i$ of the graph $G_i$ with $x$ (see Figure \ref{bouquet-n}). Then,
			\begin{align*}
				ABC_{GG}(B(G_1,...,G_n))&<|E(G)| + \sum_{i=1}^{n}ABC_{GG}(G_i)\\
&\quad +\sum_{i=1}^{n-1}\sqrt{\frac{|V(G)|-2}{\sum_{t=1}^{i}|V(G_t)|  \sum_{t=i+1}^{n}|V(G_t)|}}.\\
		\end{align*}

					\end{theorem}

	\begin{figure}
		\begin{center}
			\psscalebox{0.7 0.7}
			{
				\begin{pspicture}(0,-6.76)(5.6,-1.16)
				\rput[bl](2.6133332,-3.64){$x_1$}
				\rput[bl](3.0533333,-4.0933332){$x_2$}
				\rput[bl](2.6533334,-4.5466666){$x_3$}
				\rput[bl](2.5866666,-1.8133334){$G_1$}
				\rput[bl](4.72,-4.1066666){$G_2$}
				\rput[bl](2.56,-6.2){$G_3$}
				\rput[bl](2.1333334,-4.04){$x_n$}
				\rput[bl](0.21333334,-4.0933332){$G_n$}
				\psellipse[linecolor=black, linewidth=0.04, dimen=outer](1.4,-3.96)(1.4,0.4)
				\psellipse[linecolor=black, linewidth=0.04, dimen=outer](2.8,-2.56)(0.4,1.4)
				\psellipse[linecolor=black, linewidth=0.04, dimen=outer](4.2,-3.96)(1.4,0.4)
				\psellipse[linecolor=black, linewidth=0.04, dimen=outer](2.8,-5.36)(0.4,1.4)
				\psdots[linecolor=black, dotsize=0.1](0.8,-4.76)
				\psdots[linecolor=black, dotsize=0.1](1.2,-5.16)
				\psdots[linecolor=black, dotsize=0.1](1.6,-5.56)
				\psdots[linecolor=black, dotstyle=o, dotsize=0.5, fillcolor=white](2.8,-3.96)
				\rput[bl](2.6533334,-4.04){$x$}
				\end{pspicture}
			}
		\end{center}
		\caption{Bouquet of $n$ graphs $G_1,G_2, \ldots , G_n$ and $x_1=x_2=\ldots=x_n=x$.} \label{bouquet-n}
	\end{figure}
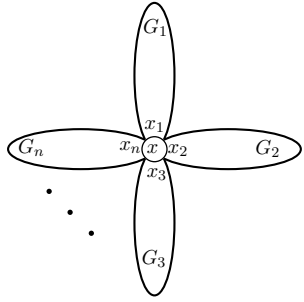

Now we consider to the circuit of graphs.

	\begin{theorem} 
		Let $G_1,G_2, \ldots , G_n$ be a finite sequence of pairwise disjoint connected graphs and let
		$x_i \in V(G_i)$. Let $G$ be the circuit of graphs $\{G_i\}_{i=1}^n$ with respect to the vertices $\{x_i\}_{i=1}^n$ and obtained by identifying the vertex $x_i$ of the graph $G_i$ with the $i$-th vertex of the
		cycle graph $C_n$ (Figure \ref{circuit-n}). Then,
		\begin{align*}	
		ABC_{GG}(G)&<(|E(G)|-n) + \sum_{i=1}^{n}ABC_{GG}(G_i)\\
		&\quad+\sqrt{\frac{|V(G)|-2}{|V(G_1)||V(G_{n})|}}+\sum_{i=1}^{n-1} \sqrt{\frac{|V(G)|-2}{|V(G_i)||V(G_{i+1})|}}\\
		\end{align*}
	\end{theorem}

	\begin{proof}
		First consider the edge $x_1x_n$. There are two cases, $n$ is even or odd. If $n=2t$ for some  $t\in \mathbb{N}$, then, the vertices in the graphs $G_1,G_2,G_3,\ldots,G_t$ are closer to $x_1$ than $x_n$, and the rest are closer to $x_n$ than $x_1$. So
		\begin{align*}
		\sqrt{\frac{n_{x_1}(x_1x_n,G)+n_{x_n}(x_1x_n,G)-2}{n_{x_1}(x_1x_n,G)n_{x_n}(x_1x_n,G)}}&=\sqrt{\frac{\sum_{i=1}^{t}|V(G_i)|+\sum_{i=1}^{t}|V(G_{t+i})|-2}{\sum_{i=1}^{t}|V(G_i)|\sum_{i=1}^{t}|V(G_{t+i})|}}\\
		&= \sqrt{\frac{|V(G)|-2}{\sum_{i=1}^{t}|V(G_i)|\sum_{i=1}^{t}|V(G_{t+i})|}}\\
		&< \sqrt{\frac{|V(G)|-2}{|V(G_1)||V(G_{2t})|}} \\
		&= \sqrt{\frac{|V(G)|-2}{|V(G_1)||V(G_{n})|}} . 
		\end{align*}
		By the same argument, for every $x_ix_{i+1}$, $1\leq i \leq n-1$, we have:

		$$\sqrt{\frac{n_{x_i}(x_ix_{i+1},G)+n_{x_{i+1}}(x_ix_{i+1},G)-2}{n_{x_i}(x_ix_{i+1},G)n_{x_{i+1}}(x_ix_{i+1},G)}}<\sqrt{\frac{|V(G)|-2}{|V(G_i)||V(G_{i+1})|}}.$$

		If $n=2t-1$ for some $t\in \mathbb{N}$, then, the vertices in the graphs $G_1,G_2,G_3,\ldots,G_{t-1}$ are closer to $x_1$ than $x_n$, and the vertices in the graphs $G_{t+1},G_{t+2},G_{t+3},\ldots,G_n$ are closer to $x_n$ than $x_1$. The vertices in the graph $G_{t}$ have the same distance to $x_1$ and $x_n$. So

		\begin{align*}
		\sqrt{\frac{n_{x_1}(x_1x_n,G)+n_{x_n}(x_1x_n,G)-2}{n_{x_1}(x_1x_n,G)n_{x_n}(x_1x_n,G)}}&=\sqrt{\frac{\sum_{i=1}^{t-1}|V(G_i)|+\sum_{i=1}^{t-1}|V(G_{t+i})|-2}{\sum_{i=1}^{t-1}|V(G_i)|\sum_{i=1}^{t-1}|V(G_{t+i})|}}\\
		&= \sqrt{\frac{|V(G)|-|V(G_{t})|-2}{\sum_{i=1}^{t-1}|V(G_i)|\sum_{i=1}^{t-1}|V(G_{t+i})|}}\\
		\end{align*}
		
		Therefore,
		
		\begin{align*}
		\sqrt{\frac{n_{x_1}(x_1x_n,G)+n_{x_n}(x_1x_n,G)-2}{n_{x_1}(x_1x_n,G)n_{x_n}(x_1x_n,G)}} &< \sqrt{\frac{|V(G)|-|V(G_{t})|-2}{|V(G_1)||V(G_{2t-1})|}} \\
		&= \sqrt{\frac{|V(G)|-|V(G_{t})|-2}{|V(G_1)||V(G_{n})|}}\\
		&< \sqrt{\frac{|V(G)|-2}{|V(G_1)||V(G_{n})|}} . 
		\end{align*}

		By the same argument, for every $x_ix_{i+1}$, $1\leq i \leq n-1$, we have:

		$$\sqrt{\frac{n_{x_i}(x_ix_{i+1},G)+n_{x_{i+1}}(x_ix_{i+1},G)-2}{n_{x_i}(x_ix_{i+1},G)n_{x_{i+1}}(x_ix_{i+1},G)}}<\sqrt{\frac{|V(G)|-2}{|V(G_i)||V(G_{i+1})|}}.$$

Now by the definition of Graovac-Ghorbani index and similar argument like the proof of the Theorem \ref{thm-link-GG}, we have the result.	\qed
	\end{proof}

\section{Chemical applications}\label{sec3}

In this section, we apply our previous results in order to obtain the atom-bond connectivity and Graovac-Ghorbani indices of
families of graphs that are of importance in chemistry.

    \subsection{Spiro-chains}
    Spiro-chains are defined in \cite{Diudea}. Making use of the concept of chain of graphs, a
    spiro-chain can be defined as a chain of cycles. We denote by $S_{q,h,k}$ the chain of $k$ cycles $C_q$ in which the distance between two consecutive contact vertices is $h$ (see spiro-chain  $S_{6,2,8}$ in Figure \ref{S628}). 
    
     \begin{figure}
    	\begin{center}
    		\psscalebox{0.45 0.45}
    		{
    			\begin{pspicture}(0,-5.2)(16.394232,-0.0057690428)
    			\psline[linecolor=black, linewidth=0.08](0.9971153,-2.6028845)(2.1971154,-2.6028845)(2.9971154,-3.8028846)(2.1971154,-5.0028844)(0.9971153,-5.0028844)(0.19711533,-3.8028846)(0.9971153,-2.6028845)(0.9971153,-2.6028845)
    			\psline[linecolor=black, linewidth=0.08](2.1971154,-0.20288453)(3.3971152,-0.20288453)(4.1971154,-1.4028845)(3.3971152,-2.6028845)(2.1971154,-2.6028845)(1.3971153,-1.4028845)(2.1971154,-0.20288453)(2.1971154,-0.20288453)
    			\psline[linecolor=black, linewidth=0.08](4.997115,-0.20288453)(6.1971154,-0.20288453)(6.997115,-1.4028845)(6.1971154,-2.6028845)(4.997115,-2.6028845)(4.1971154,-1.4028845)(4.997115,-0.20288453)(4.997115,-0.20288453)
    			\psline[linecolor=black, linewidth=0.08](6.1971154,-2.6028845)(7.397115,-2.6028845)(8.197115,-3.8028846)(7.397115,-5.0028844)(6.1971154,-5.0028844)(5.397115,-3.8028846)(6.1971154,-2.6028845)(6.1971154,-2.6028845)
    			\psline[linecolor=black, linewidth=0.08](8.997115,-2.6028845)(10.197115,-2.6028845)(10.997115,-3.8028846)(10.197115,-5.0028844)(8.997115,-5.0028844)(8.197115,-3.8028846)(8.997115,-2.6028845)(8.997115,-2.6028845)
    			\psline[linecolor=black, linewidth=0.08](10.197115,-0.20288453)(11.397116,-0.20288453)(12.197115,-1.4028845)(11.397116,-2.6028845)(10.197115,-2.6028845)(9.397116,-1.4028845)(10.197115,-0.20288453)(10.197115,-0.20288453)
    			\psline[linecolor=black, linewidth=0.08](12.997115,-0.20288453)(14.197115,-0.20288453)(14.997115,-1.4028845)(14.197115,-2.6028845)(12.997115,-2.6028845)(12.197115,-1.4028845)(12.997115,-0.20288453)(12.997115,-0.20288453)
    			\psline[linecolor=black, linewidth=0.08](14.197115,-2.6028845)(15.397116,-2.6028845)(16.197115,-3.8028846)(15.397116,-5.0028844)(14.197115,-5.0028844)(13.397116,-3.8028846)(14.197115,-2.6028845)(14.197115,-2.6028845)
    			\psdots[linecolor=black, dotsize=0.4](14.197115,-2.6028845)
    			\psdots[linecolor=black, dotsize=0.4](13.397116,-3.8028846)
    			\psdots[linecolor=black, dotsize=0.4](14.197115,-5.0028844)
    			\psdots[linecolor=black, dotsize=0.4](15.397116,-5.0028844)
    			\psdots[linecolor=black, dotsize=0.4](16.197115,-3.8028846)
    			\psdots[linecolor=black, dotsize=0.4](15.397116,-2.6028845)
    			\psdots[linecolor=black, dotsize=0.4](14.997115,-1.4028845)
    			\psdots[linecolor=black, dotsize=0.4](14.197115,-0.20288453)
    			\psdots[linecolor=black, dotsize=0.4](12.997115,-0.20288453)
    			\psdots[linecolor=black, dotsize=0.4](12.197115,-1.4028845)
    			\psdots[linecolor=black, dotsize=0.4](12.997115,-2.6028845)
    			\psdots[linecolor=black, dotsize=0.4](11.397116,-2.6028845)
    			\psdots[linecolor=black, dotsize=0.4](10.197115,-2.6028845)
    			\psdots[linecolor=black, dotsize=0.4](9.397116,-1.4028845)
    			\psdots[linecolor=black, dotsize=0.4](10.197115,-0.20288453)
    			\psdots[linecolor=black, dotsize=0.4](11.397116,-0.20288453)
    			\psdots[linecolor=black, dotsize=0.4](10.997115,-3.8028846)
    			\psdots[linecolor=black, dotsize=0.4](10.197115,-5.0028844)
    			\psdots[linecolor=black, dotsize=0.4](8.997115,-2.6028845)
    			\psdots[linecolor=black, dotsize=0.4](8.197115,-3.8028846)
    			\psdots[linecolor=black, dotsize=0.4](8.997115,-5.0028844)
    			\psdots[linecolor=black, dotsize=0.4](7.397115,-5.0028844)
    			\psdots[linecolor=black, dotsize=0.4](6.1971154,-5.0028844)
    			\psdots[linecolor=black, dotsize=0.4](5.397115,-3.8028846)
    			\psdots[linecolor=black, dotsize=0.4](6.1971154,-2.6028845)
    			\psdots[linecolor=black, dotsize=0.4](7.397115,-2.6028845)
    			\psdots[linecolor=black, dotsize=0.4](6.997115,-1.4028845)
    			\psdots[linecolor=black, dotsize=0.4](6.1971154,-0.20288453)
    			\psdots[linecolor=black, dotsize=0.4](4.997115,-0.20288453)
    			\psdots[linecolor=black, dotsize=0.4](4.1971154,-1.4028845)
    			\psdots[linecolor=black, dotsize=0.4](4.997115,-2.6028845)
    			\psdots[linecolor=black, dotsize=0.4](3.3971152,-2.6028845)
    			\psdots[linecolor=black, dotsize=0.4](3.3971152,-0.20288453)
    			\psdots[linecolor=black, dotsize=0.4](2.1971154,-0.20288453)
    			\psdots[linecolor=black, dotsize=0.4](1.3971153,-1.4028845)
    			\psdots[linecolor=black, dotsize=0.4](2.1971154,-2.6028845)
    			\psdots[linecolor=black, dotsize=0.4](0.9971153,-2.6028845)
    			\psdots[linecolor=black, dotsize=0.4](0.19711533,-3.8028846)
    			\psdots[linecolor=black, dotsize=0.4](0.9971153,-5.0028844)
    			\psdots[linecolor=black, dotsize=0.4](2.1971154,-5.0028844)
    			\psdots[linecolor=black, dotsize=0.4](2.9971154,-3.8028846)
    			\end{pspicture}
    		}
    	\end{center}
		\caption{The graph $S_{6,2,8}$.}\label{S628}
    \end{figure}
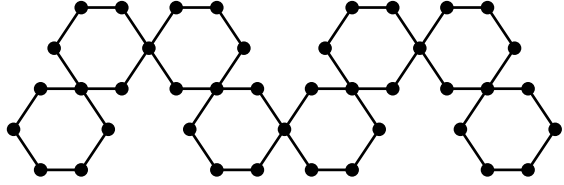

    \begin{theorem} \label{thm-l(qhk)}
    	For the graph  $S_{q,h,k}$ ($h\geq 2$), we have:
    	\begin{align*}
    	ABC(S_{q,h,k})=\frac{qk}{\sqrt{2}}.
    	\end{align*}
    \end{theorem}

    \begin{proof}
    	There are $4(k-1)$ edges with endpoints of degree 2 and 4. Also  there are $qk-4(k-1)$ edges with endpoints of degree 2. Therefore 
    	\begin{align*}
    	ABC(S_{q,h,k})=4(k-1)\sqrt{\frac{2+4-2}{2(4)}}+(qk-4(k-1))\sqrt{\frac{2+2-2}{2(2)}},
    	\end{align*}
    	and we have the result.	
    	\qed
    \end{proof}

    \begin{theorem} \label{thm-l(q1k)}
    	For the graph  $S_{q,1,k}$, we have:
    	\begin{align*}
    	ABC(S_{q,1,k})=\frac{qk-k+2}{\sqrt{2}}+\frac{(k-2)\sqrt{6}}{4}.
    	\end{align*}
    \end{theorem}

    \begin{proof}
    	There are $k-2$ edges with endpoints of degree 4. Also there are $2k$ edges with endpoints of degree 4 and 2, and there are $qk-3k+2$ edges with endpoints of degree 2. Therefore 
by the definition of the atom-bond connectivity,
we have the result.	
    	\qed
    \end{proof}

	\begin{theorem}
Let $T_n$ be the chain triangular graph  of order $n$. Then,
\begin{itemize}
\item[(i)] for every $n\geq 2$, and $k \geq 1$, if $n=2k$, we have:

\begin{align*}
	ABC_{GG}(T_{n})&=2\sum_{i=1}^{k}\left(\sqrt{\frac{2i-2}{2i-1}}+ \sqrt{\frac{4k-2i}{4k-2i+1}}+ \sqrt{\frac{4k-2}{(4k-2i+1)(2i-1)}}\right),\\
\end{align*}

and if $n=2k+1$, we have:

\begin{align*}
	ABC_{GG}(T_{n})&=2\sum_{i=1}^{k}\left(\sqrt{\frac{2i-2}{2i-1}}+\sqrt{\frac{4k-2i+2}{4k-2i+3}}+\sqrt{\frac{4k}{(4k-2i+3)(2i-1)}}\right)\\
	&\quad +2\sqrt{\frac{2k}{2k+1}}+\frac{2\sqrt{k}}{2k+1}.
\end{align*}

\item[(ii)] for every $n\geq 2$,
 		$ABC(T_n)=\frac{2n+2}{\sqrt{2}}+\frac{(n-2)\sqrt{6}}{4}.$
\end{itemize}

\end{theorem}

\begin{figure}
\begin{center}
\psscalebox{0.7 0.7}
{
\begin{pspicture}(0,-8.22)(19.59423,-4.56)
\definecolor{colour0}{rgb}{1.0,0.4,0.4}
\psdots[linecolor=black, dotsize=0.4](0.9971154,-5.37)
\psdots[linecolor=black, dotsize=0.4](0.1971154,-6.57)
\psdots[linecolor=black, dotsize=0.4](1.7971154,-6.57)
\psdots[linecolor=black, dotsize=0.4](2.5971155,-5.37)
\psdots[linecolor=black, dotsize=0.4](3.3971155,-6.57)
\psdots[linecolor=black, dotsize=0.1](3.7971153,-6.57)
\psdots[linecolor=black, dotsize=0.1](4.1971154,-6.57)
\psdots[linecolor=black, dotsize=0.1](4.5971155,-6.57)
\psdots[linecolor=black, dotsize=0.1](6.9971156,-6.57)
\psdots[linecolor=black, dotsize=0.1](7.397115,-6.57)
\psdots[linecolor=black, dotsize=0.1](7.7971153,-6.57)
\psdots[linecolor=black, dotsize=0.4](8.997115,-5.37)
\psdots[linecolor=black, dotsize=0.4](9.797115,-6.57)
\psline[linecolor=black, linewidth=0.08](8.997115,-5.37)(8.197115,-6.57)(9.797115,-6.57)(8.997115,-5.37)(8.997115,-5.37)
\psdots[linecolor=black, dotsize=0.4](8.197115,-6.57)
\psdots[linecolor=black, dotsize=0.4](10.5971155,-5.37)
\psdots[linecolor=black, dotsize=0.4](11.397116,-6.57)
\psline[linecolor=black, linewidth=0.08](10.5971155,-5.37)(9.797115,-6.57)(11.397116,-6.57)(10.5971155,-5.37)(10.5971155,-5.37)
\psdots[linecolor=black, dotsize=0.4](9.797115,-6.57)
\psdots[linecolor=black, dotsize=0.1](11.797115,-6.57)
\psdots[linecolor=black, dotsize=0.1](12.197115,-6.57)
\psdots[linecolor=black, dotsize=0.1](12.5971155,-6.57)
\psdots[linecolor=black, dotsize=0.1](14.997115,-6.57)
\psdots[linecolor=black, dotsize=0.1](15.397116,-6.57)
\psdots[linecolor=black, dotsize=0.1](15.797115,-6.57)
\psdots[linecolor=black, dotsize=0.4](16.997116,-5.37)
\psdots[linecolor=black, dotsize=0.4](16.197115,-6.57)
\psdots[linecolor=black, dotsize=0.4](17.797115,-6.57)
\psdots[linecolor=black, dotsize=0.4](18.597115,-5.37)
\psdots[linecolor=black, dotsize=0.4](19.397116,-6.57)
\psline[linecolor=blue, linewidth=0.08](4.9971156,-6.57)(5.7971153,-5.37)(5.7971153,-5.37)
\psline[linecolor=blue, linewidth=0.08](14.5971155,-6.57)(13.797115,-5.37)(13.797115,-5.37)
\psline[linecolor=green, linewidth=0.08](5.7971153,-5.37)(6.5971155,-6.57)(6.5971155,-6.57)
\psline[linecolor=green, linewidth=0.08](13.797115,-5.37)(12.997115,-6.57)(12.997115,-6.57)
\psline[linecolor=red, linewidth=0.08](4.9971156,-6.57)(6.5971155,-6.57)(6.5971155,-6.57)
\psline[linecolor=red, linewidth=0.08](12.997115,-6.57)(14.5971155,-6.57)(14.5971155,-6.57)
\psdots[linecolor=black, dotsize=0.4](5.7971153,-5.37)
\psdots[linecolor=black, dotsize=0.4](4.9971156,-6.57)
\psdots[linecolor=black, dotsize=0.4](6.5971155,-6.57)
\psdots[linecolor=black, dotsize=0.4](12.997115,-6.57)
\psdots[linecolor=black, dotsize=0.4](13.797115,-5.37)
\psdots[linecolor=black, dotsize=0.4](14.5971155,-6.57)
\psline[linecolor=colour0, linewidth=0.08, linestyle=dashed, dash=0.17638889cm 0.10583334cm](9.797115,-4.57)(9.797115,-8.17)(9.797115,-8.17)
\rput[bl](5.6371155,-5.09){$u_i$}
\rput[bl](4.6771154,-7.13){$v_i$}
\rput[bl](6.5371156,-7.11){$w_i$}
\rput[bl](0.6971154,-7.95){(1)}
\rput[bl](2.4171154,-7.95){(2)}
\rput[bl](5.5571156,-7.97){(i)}
\rput[bl](8.717115,-8.01){(k)}
\rput[bl](10.197115,-8.03){(k+1)}
\rput[bl](16.597115,-8.05){(2k-1)}
\rput[bl](18.337116,-8.03){(2k)}
\psline[linecolor=black, linewidth=0.08](0.1971154,-6.57)(3.3971155,-6.57)(2.5971155,-5.37)(1.7971154,-6.57)(0.9971154,-5.37)(0.1971154,-6.57)(0.1971154,-6.57)
\psline[linecolor=black, linewidth=0.08](16.197115,-6.57)(16.997116,-5.37)(17.797115,-6.57)(18.597115,-5.37)(19.397116,-6.57)(16.197115,-6.57)(16.197115,-6.57)
\rput[bl](13.177115,-5.01){$u_{2k-i+1}$}
\rput[bl](14.197115,-7.09){$v_{2k-i+1}$}
\rput[bl](12.617115,-7.13){$w_{2k-i+1}$}
\rput[bl](13.257115,-8.05){(2k-i+1)}
\end{pspicture}
}
\end{center}
\caption{Chain triangular cactus $T_{2k}$. } \label{Chaintri2k}
\end{figure}
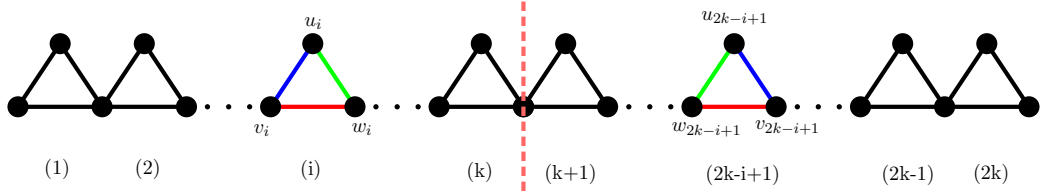

	\begin{proof}
	\begin{itemize}
\item[(i)]
	We consider the following cases:

\begin{itemize}
\item[\textbf{Case 1.}] Suppose that $n$ is even, and  $n=2k$ for some $k\in \mathbb{N}$.
Consider the $T_{2k}$ as shown in Figure \ref{Chaintri2k}. One can easily check that whatever happens to computation of Graovac-Ghorbani index related to the edge $u_iv_i$ in the $(i)$-th triangle in $T_{2k}$, is the same as computation of Graovac-Ghorbani index related to the edge $u_{2k-i+1}v_{2k-i+1}$ in the $(2k-i+1)$-th triangle. The same goes for $w_iv_i$ and $w_{2k-i+1}v_{2k-i+1}$, and also for $w_iu_i$ and $w_{2k-i+1}u_{2k-i+1}$. So for computing Graovac-Ghorbani index, it suffices to compute the $\sqrt{\frac{n_u(uv,T_{2k})+n_v(uv,T_{2k})-2}{n_u(uv,T_{2k})n_v(uv,T_{2k})}}$  for every $uv \in E(T_{2k})$ in the first $k$ triangles and then multiple that by 2. So from now, we only consider the  first $k$ triangles. 

Consider the blue edge $u_iv_i$ in the $(i)$-th triangle. There are $2(i-1)+1$ vertices which are closer to $v_i$ than $u_i$, and there is one vertex closer to $u_i$ than $v_i$. So, $\sqrt{\frac{n_{u_i}(u_iv_i,T_{2k})+n_{v_i}(u_iv_i,T_{2k})-2}{n_{u_i}(u_iv_i,T_{2k})n_{v_i}(u_iv_i,T_{2k})}}=\sqrt{\frac{2i-2}{2i-1}}$.

Now consider the green edge $u_iw_i$ in the $(i)$-th triangle. There are $2(2k-i)+1$ vertices which are closer to $w_i$ than $u_i$, and there is one vertex closer to $u_i$ than $w_i$. So, $\sqrt{\frac{n_{u_i}(u_iw_i,T_{2k})+n_{w_i}(u_iw_i,T_{2k})-2}{n_{u_i}(u_iw_i,T_{2k})n_{w_i}(u_iw_i,T_{2k})}}=\sqrt{\frac{4k-2i}{4k-2i+1}}$.

Finally, consider the red edge $v_iw_i$ in the $(i)$-th triangle. There are $2(2k-i)+1$ vertices which are closer to $w_i$ than $v_i$, and there are $2(i-1)+1$ vertices closer to $v_i$ than $w_i$. So, $\sqrt{\frac{n_{v_i}(v_iw_i,T_{2k})+n_{w_i}(v_iw_i,T_{2k})-2}{n_{v_i}(v_iw_i,T_{2k})n_{w_i}(v_iw_i,T_{2k})}}=\sqrt{\frac{4k-2}{(4k-2i+1)(2i-1)}}$.

Since we have $k$ edges like blue one, $k$ edges like green one and $k$ edges like red one, then by our argument, we have:

\begin{align*}
	ABC_{GG}(T_{2k})&=2\sum_{i=1}^{k}\left(\sqrt{\frac{2i-2}{2i-1}}+ \sqrt{\frac{4k-2i}{4k-2i+1}}+ \sqrt{\frac{4k-2}{(4k-2i+1)(2i-1)}}\right)\\
\end{align*}

\item[\textbf{Case 2.}] Suppose that $n$ is odd and  $n=2k+1$ for some $k\in \mathbb{N}$.
Now consider the $T_{2k+1}$ as shown in Figure \ref{Chaintri2k+1}. One can easily check that whatever happens to computation of Graovac-Ghorbani index related to the edge $u_iv_i$ in the $(i)$-th triangle in $T_{2k+1}$, is the same as computation of Graovac-Ghorbani index related to the edge $u_{2k-i+2}v_{2k-i+2}$ in the $(2k-i+2)$-th triangle. The same goes for $w_iv_i$ and $w_{2k-i+2}v_{2k-i+2}$, and also for $w_iu_i$ and $w_{2k-i+2}u_{2k-i+2}$. So for computing Graovac-Ghorbani index, it suffices to compute the $\sqrt{\frac{n_u(uv,T_{2k+1})+n_v(uv,T_{2k+1})-2}{n_u(uv,T_{2k+1})n_v(uv,T_{2k+1})}}$  for every $uv \in E(T_{2k+1})$ in the first $k$ triangles and then multiple that by 2 and add it to 
$$\sum_{uv\in A}^{}\sqrt{\frac{n_u(uv,T_{2k+1})+n_v(uv,T_{2k+1})-2}{n_u(uv,T_{2k+1})n_v(uv,T_{2k+1})}},$$
 where  $A = \{ab,bc,ac\}$. So from now, we only consider the  first $k$ triangles and the middle one. 

Consider the blue edge $u_iv_i$ in the $(i)$-th triangle. There are $2(i-1)+1$ vertices which are closer to $v_i$ than $u_i$, and  there is one vertex closer to $u_i$ than $v_i$. So, $\sqrt{\frac{n_{u_i}(u_iv_i,T_{2k+1})+n_{v_i}(u_iv_i,T_{2k+1})-2}{n_{u_i}(u_iv_i,T_{2k+1})n_{v_i}(u_iv_i,T_{2k+1})}}=\sqrt{\frac{2i-2}{2i-1}}$.

Now consider the green edge $u_iw_i$ in the $(i)$-th triangle. There are $4k-2i+3$ vertices which are closer to $w_i$ than $u_i$, and  there is one vertex closer to $u_i$ than $w_i$. So, $\sqrt{\frac{n_{u_i}(u_iw_i,T_{2k+1})+n_{w_i}(u_iw_i,T_{2k+1})-2}{n_{u_i}(u_iw_i,T_{2k+1})n_{w_i}(u_iw_i,T_{2k+1})}}=\sqrt{\frac{4k-2i+2}{4k-2i+3}}$.

Now consider the red edge $v_iw_i$ in the $(i)$-th triangle. There are $2(2k-i+1)+1$ vertices which are closer to $w_i$ than $v_i$, and there are $2(i-1)+1$ vertices closer to $v_i$ than $w_i$. So, $\sqrt{\frac{n_{v_i}(v_iw_i,T_{2k+1})+n_{w_i}(v_iw_i,T_{2k+1})-2}{n_{v_i}(v_iw_i,T_{2k+1})n_{w_i}(v_iw_i,T_{2k+1})}}=\sqrt{\frac{4k}{(4k-2i+3)(2i-1)}}$.

Finally, consider  the middle triangle. For the  edge $ab$, there are $2k+1$ vertices which are closer to $b$ than $a$, and there is one vertex closer to $a$ than $b$. Also for the  edge $ac$, there are $2k+1$ vertices which are closer to $c$ than $a$, and there is one vertex closer to $a$ than $c$ and for the  edge $bc$, there are $2k+1$ vertices which are closer to $b$ than $c$, and  there are $2k+1$ vertices closer to $c$ than $b$. Hence, $\sum_{uv\in A}^{}\sqrt{\frac{n_u(uv,T_{2k+1})+n_v(uv,T_{2k+1})-2}{n_u(uv,T_{2k+1})n_v(uv,T_{2k+1})}}=2\sqrt{\frac{2k}{2k+1}}+\frac{\sqrt{4k}}{2k+1}$, where  $A = \{ab,bc,ac\}$.

Since we have $k$ edges like blue one, $k$ edges like green one and $k$ edges like red one, then by our argument, we have:

\begin{align*}
	ABC_{GG}(T_{2k+1})&=2\sum_{i=1}^{k}\left(\sqrt{\frac{2i-2}{2i-1}}+\sqrt{\frac{4k-2i+2}{4k-2i+3}}+\sqrt{\frac{4k}{(4k-2i+3)(2i-1)}}\right)\\
	&\quad +2\sqrt{\frac{2k}{2k+1}}+\frac{2\sqrt{k}}{2k+1}.
\end{align*}

\end{itemize}	
Therefore, we have the result. 
\item[(ii)] It follows from Theorem  \ref{thm-l(q1k)}.
\qed
\end{itemize}
	\end{proof}

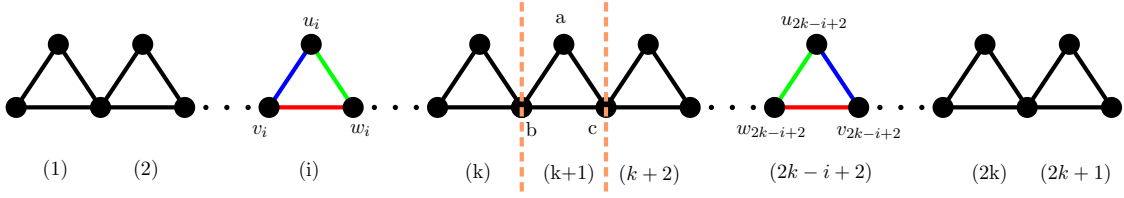
\begin{figure}
\begin{center}
\psscalebox{0.7 0.7}
{
\begin{pspicture}(0,-8.22)(21.19423,-4.56)
\definecolor{colour0}{rgb}{1.0,0.6,0.4}
\psdots[linecolor=black, dotsize=0.4](0.9971154,-5.37)
\psdots[linecolor=black, dotsize=0.4](0.19711538,-6.57)
\psdots[linecolor=black, dotsize=0.4](1.7971153,-6.57)
\psdots[linecolor=black, dotsize=0.4](2.5971153,-5.37)
\psdots[linecolor=black, dotsize=0.4](3.3971155,-6.57)
\psdots[linecolor=black, dotsize=0.1](3.7971153,-6.57)
\psdots[linecolor=black, dotsize=0.1](4.1971154,-6.57)
\psdots[linecolor=black, dotsize=0.1](4.5971155,-6.57)
\psdots[linecolor=black, dotsize=0.1](6.9971156,-6.57)
\psdots[linecolor=black, dotsize=0.1](7.397115,-6.57)
\psdots[linecolor=black, dotsize=0.1](7.7971153,-6.57)
\psdots[linecolor=black, dotsize=0.4](8.997115,-5.37)
\psdots[linecolor=black, dotsize=0.4](9.797115,-6.57)
\psline[linecolor=black, linewidth=0.08](8.997115,-5.37)(8.197115,-6.57)(9.797115,-6.57)(8.997115,-5.37)(8.997115,-5.37)
\psdots[linecolor=black, dotsize=0.4](8.197115,-6.57)
\psdots[linecolor=black, dotsize=0.4](10.5971155,-5.37)
\psdots[linecolor=black, dotsize=0.4](11.397116,-6.57)
\psline[linecolor=black, linewidth=0.08](10.5971155,-5.37)(9.797115,-6.57)(11.397116,-6.57)(10.5971155,-5.37)(10.5971155,-5.37)
\psdots[linecolor=black, dotsize=0.4](9.797115,-6.57)
\psdots[linecolor=black, dotsize=0.1](13.397116,-6.57)
\psdots[linecolor=black, dotsize=0.1](13.797115,-6.57)
\psdots[linecolor=black, dotsize=0.1](14.197115,-6.57)
\psdots[linecolor=black, dotsize=0.1](16.597115,-6.57)
\psdots[linecolor=black, dotsize=0.1](16.997116,-6.57)
\psdots[linecolor=black, dotsize=0.1](17.397116,-6.57)
\psdots[linecolor=black, dotsize=0.4](18.597115,-5.37)
\psdots[linecolor=black, dotsize=0.4](17.797115,-6.57)
\psdots[linecolor=black, dotsize=0.4](19.397116,-6.57)
\psdots[linecolor=black, dotsize=0.4](20.197115,-5.37)
\psdots[linecolor=black, dotsize=0.4](20.997116,-6.57)
\psline[linecolor=blue, linewidth=0.08](4.9971156,-6.57)(5.7971153,-5.37)(5.7971153,-5.37)
\psline[linecolor=blue, linewidth=0.08](16.197115,-6.57)(15.397116,-5.37)(15.397116,-5.37)
\psline[linecolor=green, linewidth=0.08](5.7971153,-5.37)(6.5971155,-6.57)(6.5971155,-6.57)
\psline[linecolor=green, linewidth=0.08](15.397116,-5.37)(14.5971155,-6.57)(14.5971155,-6.57)
\psline[linecolor=red, linewidth=0.08](4.9971156,-6.57)(6.5971155,-6.57)(6.5971155,-6.57)
\psline[linecolor=red, linewidth=0.08](14.5971155,-6.57)(16.197115,-6.57)(16.197115,-6.57)
\psdots[linecolor=black, dotsize=0.4](5.7971153,-5.37)
\psdots[linecolor=black, dotsize=0.4](4.9971156,-6.57)
\psdots[linecolor=black, dotsize=0.4](6.5971155,-6.57)
\psdots[linecolor=black, dotsize=0.4](14.5971155,-6.57)
\psdots[linecolor=black, dotsize=0.4](15.397116,-5.37)
\psdots[linecolor=black, dotsize=0.4](16.197115,-6.57)
\rput[bl](5.6371155,-5.09){$u_i$}
\rput[bl](4.6771154,-7.13){$v_i$}
\rput[bl](6.5371156,-7.11){$w_i$}
\rput[bl](0.69711536,-7.95){(1)}
\rput[bl](2.4171154,-7.95){(2)}
\rput[bl](5.5571156,-7.97){(i)}
\rput[bl](8.717115,-8.01){(k)}
\rput[bl](10.197115,-8.03){(k+1)}
\rput[bl](18.337116,-8.03){(2k)}
\psline[linecolor=black, linewidth=0.08](0.19711538,-6.57)(3.3971155,-6.57)(2.5971153,-5.37)(1.7971153,-6.57)(0.9971154,-5.37)(0.19711538,-6.57)(0.19711538,-6.57)
\psline[linecolor=black, linewidth=0.08](17.797115,-6.57)(18.597115,-5.37)(19.397116,-6.57)(20.197115,-5.37)(20.997116,-6.57)(17.797115,-6.57)(17.797115,-6.57)
\psline[linecolor=black, linewidth=0.08](11.397116,-6.57)(12.197115,-5.37)(12.997115,-6.57)(11.397116,-6.57)(11.397116,-6.57)
\psdots[linecolor=black, dotsize=0.4](12.197115,-5.37)
\psdots[linecolor=black, dotsize=0.4](12.997115,-6.57)
\psline[linecolor=colour0, linewidth=0.08, linestyle=dashed, dash=0.17638889cm 0.10583334cm](11.397116,-4.57)(11.397116,-8.17)(11.397116,-8.17)
\psline[linecolor=colour0, linewidth=0.08, linestyle=dashed, dash=0.17638889cm 0.10583334cm](9.797115,-4.57)(9.797115,-8.17)(9.797115,-8.17)
\rput[bl](11.6371155,-8.05){$(k+2)$}
\rput[bl](14.497115,-8.01){$(2k-i+2)$}
\rput[bl](19.637115,-8.03){$(2k+1)$}
\rput[bl](14.717115,-5.09){$u_{2k-i+2}$}
\rput[bl](15.777116,-7.19){$v_{2k-i+2}$}
\rput[bl](13.877115,-7.17){$w_{2k-i+2}$}
\rput[bl](10.437116,-4.97){a}
\rput[bl](9.877115,-7.11){b}
\rput[bl](11.057116,-7.07){c}
\end{pspicture}
}
\end{center}
\caption{Chain triangular cactus $T_{2k+1}$. } \label{Chaintri2k+1}
\end{figure}

	\begin{theorem} \label{thm-para-Q}
Let $Q_n$ be the para-chain square cactus graph  of order $n$. Then,
\begin{itemize}

 \item[(i)] for every $n\geq 1$, and $k \geq 1$, we have:

	 \[
 	ABC_{GG}(Q_n)=\left\{
  	\begin{array}{ll}
  	{\displaystyle
  		8\sum_{i=1}^{k}\sqrt{\frac{6k-1}{(6k-3i+2)(3i-1)}}}&
  		\quad\mbox{if $n=2k$, }\\[15pt]
  		{\displaystyle
  			8\left(\sum_{i=1}^{k}\sqrt{\frac{6k+2}{(6k-3i+5)(3i-1)}}\right)+\frac{4\sqrt{6k+2}}{3k+2}}&
  			\quad\mbox{if $n=2k+1$.}
  				  \end{array}
  					\right.	
  					\]

\item[(ii)] for every $n\geq 2$, $ABC(Q_n)=2n\sqrt{2}.$
\end{itemize}

	\end{theorem}

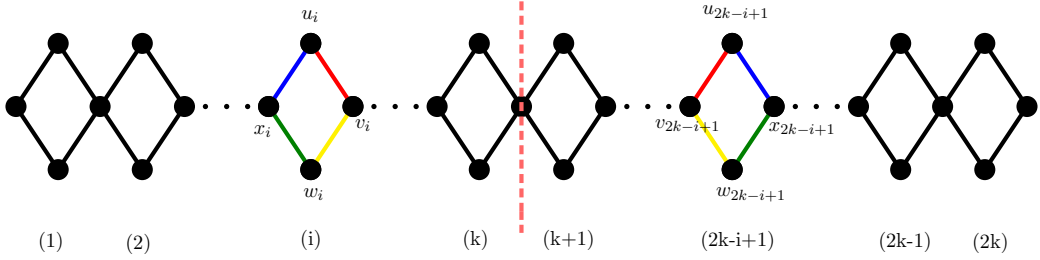
\begin{figure}
\begin{center}
\psscalebox{0.7 0.7}
{
\begin{pspicture}(0,-5.195)(19.59423,-0.385)
\definecolor{colour1}{rgb}{0.0,0.5019608,0.0}
\definecolor{colour0}{rgb}{1.0,0.4,0.4}
\psdots[linecolor=black, dotsize=0.4](0.9971153,-1.195)
\psdots[linecolor=black, dotsize=0.4](0.19711533,-2.395)
\psdots[linecolor=black, dotsize=0.4](0.9971153,-3.595)
\psdots[linecolor=black, dotsize=0.4](1.7971153,-2.395)
\psdots[linecolor=black, dotsize=0.4](2.5971153,-1.195)
\psdots[linecolor=black, dotsize=0.4](3.3971152,-2.395)
\psdots[linecolor=black, dotsize=0.4](2.5971153,-3.595)
\psdots[linecolor=black, dotsize=0.1](3.7971153,-2.395)
\psdots[linecolor=black, dotsize=0.1](4.1971154,-2.395)
\psdots[linecolor=black, dotsize=0.1](4.5971155,-2.395)
\psdots[linecolor=black, dotsize=0.4](4.997115,-2.395)
\psdots[linecolor=black, dotsize=0.4](5.7971153,-1.195)
\psdots[linecolor=black, dotsize=0.4](6.5971155,-2.395)
\psdots[linecolor=black, dotsize=0.4](5.7971153,-3.595)
\psdots[linecolor=black, dotsize=0.1](6.997115,-2.395)
\psdots[linecolor=black, dotsize=0.1](7.397115,-2.395)
\psdots[linecolor=black, dotsize=0.1](7.7971153,-2.395)
\psdots[linecolor=black, dotsize=0.4](8.997115,-1.195)
\psdots[linecolor=black, dotsize=0.4](8.197115,-2.395)
\psdots[linecolor=black, dotsize=0.4](8.997115,-3.595)
\psdots[linecolor=black, dotsize=0.4](9.797115,-2.395)
\psdots[linecolor=black, dotsize=0.4](10.5971155,-1.195)
\psdots[linecolor=black, dotsize=0.4](11.397116,-2.395)
\psdots[linecolor=black, dotsize=0.4](10.5971155,-3.595)
\psdots[linecolor=black, dotsize=0.1](11.797115,-2.395)
\psdots[linecolor=black, dotsize=0.1](12.197115,-2.395)
\psdots[linecolor=black, dotsize=0.1](12.5971155,-2.395)
\psdots[linecolor=black, dotsize=0.4](12.997115,-2.395)
\psdots[linecolor=black, dotsize=0.4](13.797115,-1.195)
\psdots[linecolor=black, dotsize=0.4](14.5971155,-2.395)
\psdots[linecolor=black, dotsize=0.4](13.797115,-3.595)
\psdots[linecolor=black, dotsize=0.1](14.997115,-2.395)
\psdots[linecolor=black, dotsize=0.1](15.397116,-2.395)
\psdots[linecolor=black, dotsize=0.1](15.797115,-2.395)
\psdots[linecolor=black, dotsize=0.4](16.997116,-1.195)
\psdots[linecolor=black, dotsize=0.4](16.197115,-2.395)
\psdots[linecolor=black, dotsize=0.4](16.997116,-3.595)
\psdots[linecolor=black, dotsize=0.4](17.797115,-2.395)
\psdots[linecolor=black, dotsize=0.4](18.597115,-1.195)
\psdots[linecolor=black, dotsize=0.4](19.397116,-2.395)
\psdots[linecolor=black, dotsize=0.4](18.597115,-3.595)
\psline[linecolor=black, linewidth=0.08](0.9971153,-1.195)(0.19711533,-2.395)(0.9971153,-3.595)(2.5971153,-1.195)(3.3971152,-2.395)(2.5971153,-3.595)(0.9971153,-1.195)(0.9971153,-1.195)
\psline[linecolor=black, linewidth=0.08](8.197115,-2.395)(8.997115,-1.195)(10.5971155,-3.595)(11.397116,-2.395)(10.5971155,-1.195)(8.997115,-3.595)(8.197115,-2.395)(8.197115,-2.395)
\psline[linecolor=black, linewidth=0.08](16.197115,-2.395)(16.997116,-1.195)(18.597115,-3.595)(19.397116,-2.395)(18.597115,-1.195)(16.997116,-3.595)(16.197115,-2.395)(16.197115,-2.395)
\psline[linecolor=blue, linewidth=0.08](4.997115,-2.395)(5.7971153,-1.195)(5.7971153,-1.195)
\psline[linecolor=blue, linewidth=0.08](13.797115,-1.195)(14.5971155,-2.395)(14.5971155,-2.395)
\psline[linecolor=red, linewidth=0.08](5.7971153,-1.195)(6.5971155,-2.395)(6.5971155,-2.395)
\psline[linecolor=red, linewidth=0.08](13.797115,-1.195)(12.997115,-2.395)(12.997115,-2.395)
\psline[linecolor=colour1, linewidth=0.08](4.997115,-2.395)(5.7971153,-3.595)(5.7971153,-3.595)
\psline[linecolor=colour1, linewidth=0.08](14.5971155,-2.395)(13.797115,-3.595)(13.797115,-3.595)
\psline[linecolor=yellow, linewidth=0.08](12.997115,-2.395)(13.797115,-3.595)(13.797115,-3.595)
\psline[linecolor=yellow, linewidth=0.08](6.5971155,-2.395)(5.7971153,-3.595)(5.7971153,-3.595)
\psdots[linecolor=black, dotsize=0.1](4.997115,-2.395)
\psdots[linecolor=black, dotsize=0.1](4.997115,-2.395)
\psdots[linecolor=black, dotsize=0.1](4.997115,-2.395)
\psdots[linecolor=black, dotsize=0.1](4.997115,-2.395)
\psdots[linecolor=black, dotsize=0.4](4.997115,-2.395)
\psdots[linecolor=black, dotsize=0.4](5.7971153,-1.195)
\psdots[linecolor=black, dotsize=0.4](6.5971155,-2.395)
\psdots[linecolor=black, dotsize=0.4](5.7971153,-3.595)
\psdots[linecolor=black, dotsize=0.4](12.997115,-2.395)
\psdots[linecolor=black, dotsize=0.4](13.797115,-1.195)
\psdots[linecolor=black, dotsize=0.4](14.5971155,-2.395)
\psdots[linecolor=black, dotsize=0.4](13.797115,-3.595)
\psline[linecolor=colour0, linewidth=0.08, linestyle=dashed, dash=0.17638889cm 0.10583334cm](9.797115,-0.395)(9.797115,-4.395)(9.797115,-4.395)
\psline[linecolor=colour0, linewidth=0.08, linestyle=dashed, dash=0.17638889cm 0.10583334cm](9.797115,-4.395)(9.797115,-4.795)(9.797115,-4.795)
\rput[bl](0.6171153,-5.155){(1)}
\rput[bl](2.2771153,-5.195){(2)}
\rput[bl](5.5971155,-5.115){(i)}
\rput[bl](8.677115,-5.115){(k)}
\rput[bl](10.197115,-5.115){(k+1)}
\rput[bl](16.577116,-5.155){(2k-1)}
\rput[bl](18.357115,-5.175){(2k)}
\rput[bl](13.197115,-5.135){(2k-i+1)}
\rput[bl](5.6171155,-0.795){$u_i$}
\rput[bl](6.6371155,-2.855){$v_i$}
\rput[bl](5.6571155,-4.175){$w_i$}
\rput[bl](4.7171154,-2.975){$x_i$}
\rput[bl](13.237115,-0.735){$u_{2k-i+1}$}
\rput[bl](12.337115,-2.895){$v_{2k-i+1}$}
\rput[bl](13.497115,-4.195){$w_{2k-i+1}$}
\rput[bl](14.497115,-2.955){$x_{2k-i+1}$}
\end{pspicture}
}
\end{center}
\caption{Para-chain square cactus $Q_{2k}$. } \label{paraChainsqu2k}
\end{figure}

	\begin{proof}
	\begin{itemize}

 \item[(i)]
We consider the following cases:

\begin{itemize}
\item[\textbf{Case 1.}] Suppose that $n$ is even and  $n=2k$ for some $k\in \mathbb{N}$.
Now consider the $Q_{2k}$ as shown in Figure \ref{paraChainsqu2k}. One can easily check that whatever happens to computation of Graovac-Ghorbani index related to the edge $u_iv_i$ in the $(i)$-th rhombus in $Q_{2k}$, is the same as computation of Graovac-Ghorbani index related to the edge $u_{2k-i+1}v_{2k-i+1}$ in the $(2k-i+1)$-th rhombus. The same goes for $w_iv_i$ and $w_{2k-i+1}v_{2k-i+1}$, for $w_ix_i$ and $w_{2k-i+1}x_{2k-i+1}$, and also for $x_iu_i$ and $x_{2k-i+1}u_{2k-i+1}$. So for computing Graovac-Ghorbani index, it suffices to compute the $\sqrt{\frac{n_u(uv,Q_{2k})+n_v(uv,Q_{2k})-2}{n_u(uv,Q_{2k})n_v(uv,Q_{2k})}}$  for every $uv \in E(Q_{2k})$ in the first $k$ rhombus and then multiple that by 2. So from now, we only consider the  first $k$ rhombus.

Consider the red edge $u_iv_i$ in the $(i)$-th rhombus. There are $3k+3(k-i)+2$ vertices which are closer to $v_i$ than $u_i$, and there are $3i-1$ vertices closer to $u_i$ than $v_i$. So, $\sqrt{\frac{n_{u_i}(u_iv_i,Q_{2k})+n_{v_i}(u_iv_i,Q_{2k})-2}{n_{u_i}(u_iv_i,Q_{2k})n_{v_i}(u_iv_i,Q_{2k})}}=\sqrt{\frac{6k-1}{(6k-3i+2)(3i-1)}}$.

One can easily check that the edges $w_iv_i$, $w_ix_i$ and $x_iu_i$ have the same attitude as $u_iv_i$. Since we have $k$ edges like blue one, $k$ edges like green one, $k$ edges like yellow one and $k$ edges like red one, then by our argument, we have:

\begin{align*}
	ABC_{GG}(Q_{2k})=2\left(4\sum_{i=1}^{k}\sqrt{\frac{6k-1}{(6k-3i+2)(3i-1)}}\right).
\end{align*}

\item[\textbf{Case 2.}] Suppose that $n$ is odd and $n=2k+1$ for some $k\in \mathbb{N}$.
Now consider the $Q_{2k+1}$ as shown in Figure \ref{paraChainsqu2k+1}. One can easily check that whatever happens to computation of Graovac-Ghorbani index related to the edge $u_iv_i$ in the $(i)$-th rhombus in $Q_{2k+1}$, is the same as computation of Graovac-Ghorbani index related to the edge $u_{2k-i+2}v_{2k-i+2}$ in the $(2k-i+2)$-th rhombus. The same goes for $w_iv_i$ and $w_{2k-i+2}v_{2k-i+2}$, for $w_ix_i$ and $w_{2k-i+2}x_{2k-i+2}$, and also for $x_iu_i$ and $x_{2k-i+2}u_{2k-i+2}$. So for computing Graovac-Ghorbani index, it suffices to compute the $\sqrt{\frac{n_u(uv,Q_{2k+1})+n_v(uv,Q_{2k+1})-2}{n_u(uv,Q_{2k+1})n_v(uv,Q_{2k+1})}}$  for every $uv \in E(Q_{2k+1})$ in the first $k$ rhombus and then multiple that by 2 and add it to $\sum_{uv\in A}^{}\sqrt{\frac{n_u(uv,Q_{2k+1})+n_v(uv,Q_{2k+1})-2}{n_u(uv,Q_{2k+1})n_v(uv,Q_{2k+1})}}$, where  $A = \{ab,bc,cd,da\}$. So from now, we only consider the  first $k+1$ rhombus.

Consider the red edge $u_iv_i$ in the $(i)$-th rhombus. There are $3(k+1)+3(k-i)+2$ vertices which are closer to $v_i$ than $u_i$, and there are $3i-1$ vertices closer to $u_i$ than $v_i$. So, $\sqrt{\frac{n_{u_i}(u_iv_i,Q_{2k+1})+n_{v_i}(u_iv_i,Q_{2k+1})-2}{n_{u_i}(u_iv_i,Q_{2k+1})n_{v_i}(u_iv_i,Q_{2k+1})}}=\sqrt{\frac{6k+2}{(6k-3i+5)(3i-1)}}$.

One can easily check that the edges $w_iv_i$, $w_ix_i$ and $x_iu_i$ have the same attitude as $u_iv_i$. 

Now consider the middle rhombus. For the  edge $ab$, there are $3k+2$ vertices which are closer to $b$ than $a$, and there are $3k+2$ vertices closer to $a$ than $b$. the edges $bc$, $cd$ and $da$ have the same attitude as $ab$. Hence, $\sum_{uv\in A}^{}\sqrt{\frac{n_u(uv,Q_{2k+1})+n_v(uv,Q_{2k+1})-2}{n_u(uv,Q_{2k+1})n_v(uv,Q_{2k+1})}}=\frac{4\sqrt{6k+2}}{3k+2}$, where  $A = \{ab,bc,cd,da\}$.

Since we have $k$ edges like blue one, $k$ edges like green one, $k$ edges like yellow one and $k$ edges like red one, then by our argument, we have:

\begin{align*}
	ABC_{GG}(Q_{2k+1})=2\left(4\sum_{i=1}^{k}\sqrt{\frac{6k+2}{(6k-3i+5)(3i-1)}}\right)+\frac{4\sqrt{6k+2}}{3k+2}.
\end{align*}

\end{itemize}	
Therefore, we have the result.
\item[(ii)] It follows from Theorem  \ref{thm-l(qhk)}.
\qed
\end{itemize}
	\end{proof}

\begin{figure}
\begin{center}
\psscalebox{0.7 0.7}
{
\begin{pspicture}(0,-5.195)(21.19423,-0.385)
\definecolor{colour0}{rgb}{0.0,0.5019608,0.0}
\definecolor{colour1}{rgb}{1.0,0.4,0.4}
\psdots[linecolor=black, dotsize=0.4](0.9971154,-1.195)
\psdots[linecolor=black, dotsize=0.4](0.19711538,-2.395)
\psdots[linecolor=black, dotsize=0.4](0.9971154,-3.595)
\psdots[linecolor=black, dotsize=0.4](1.7971153,-2.395)
\psdots[linecolor=black, dotsize=0.4](2.5971153,-1.195)
\psdots[linecolor=black, dotsize=0.4](3.3971155,-2.395)
\psdots[linecolor=black, dotsize=0.4](2.5971153,-3.595)
\psdots[linecolor=black, dotsize=0.1](3.7971153,-2.395)
\psdots[linecolor=black, dotsize=0.1](4.1971154,-2.395)
\psdots[linecolor=black, dotsize=0.1](4.5971155,-2.395)
\psdots[linecolor=black, dotsize=0.4](4.9971156,-2.395)
\psdots[linecolor=black, dotsize=0.4](5.7971153,-1.195)
\psdots[linecolor=black, dotsize=0.4](6.5971155,-2.395)
\psdots[linecolor=black, dotsize=0.4](5.7971153,-3.595)
\psdots[linecolor=black, dotsize=0.1](6.9971156,-2.395)
\psdots[linecolor=black, dotsize=0.1](7.397115,-2.395)
\psdots[linecolor=black, dotsize=0.1](7.7971153,-2.395)
\psdots[linecolor=black, dotsize=0.4](8.997115,-1.195)
\psdots[linecolor=black, dotsize=0.4](8.197115,-2.395)
\psdots[linecolor=black, dotsize=0.4](8.997115,-3.595)
\psdots[linecolor=black, dotsize=0.4](9.797115,-2.395)
\psdots[linecolor=black, dotsize=0.4](10.5971155,-1.195)
\psdots[linecolor=black, dotsize=0.4](11.397116,-2.395)
\psdots[linecolor=black, dotsize=0.4](10.5971155,-3.595)
\psdots[linecolor=black, dotsize=0.1](13.397116,-2.395)
\psdots[linecolor=black, dotsize=0.1](13.797115,-2.395)
\psdots[linecolor=black, dotsize=0.1](14.197115,-2.395)
\psdots[linecolor=black, dotsize=0.4](14.5971155,-2.395)
\psdots[linecolor=black, dotsize=0.4](15.397116,-1.195)
\psdots[linecolor=black, dotsize=0.4](16.197115,-2.395)
\psdots[linecolor=black, dotsize=0.4](15.397116,-3.595)
\psdots[linecolor=black, dotsize=0.1](16.597115,-2.395)
\psdots[linecolor=black, dotsize=0.1](16.997116,-2.395)
\psdots[linecolor=black, dotsize=0.1](17.397116,-2.395)
\psdots[linecolor=black, dotsize=0.4](18.597115,-1.195)
\psdots[linecolor=black, dotsize=0.4](17.797115,-2.395)
\psdots[linecolor=black, dotsize=0.4](18.597115,-3.595)
\psdots[linecolor=black, dotsize=0.4](19.397116,-2.395)
\psdots[linecolor=black, dotsize=0.4](20.197115,-1.195)
\psdots[linecolor=black, dotsize=0.4](20.997116,-2.395)
\psdots[linecolor=black, dotsize=0.4](20.197115,-3.595)
\psline[linecolor=black, linewidth=0.08](0.9971154,-1.195)(0.19711538,-2.395)(0.9971154,-3.595)(2.5971153,-1.195)(3.3971155,-2.395)(2.5971153,-3.595)(0.9971154,-1.195)(0.9971154,-1.195)
\psline[linecolor=black, linewidth=0.08](8.197115,-2.395)(8.997115,-1.195)(10.5971155,-3.595)(11.397116,-2.395)(10.5971155,-1.195)(8.997115,-3.595)(8.197115,-2.395)(8.197115,-2.395)
\psline[linecolor=black, linewidth=0.08](17.797115,-2.395)(18.597115,-1.195)(20.197115,-3.595)(20.997116,-2.395)(20.197115,-1.195)(18.597115,-3.595)(17.797115,-2.395)(17.797115,-2.395)
\psline[linecolor=blue, linewidth=0.08](4.9971156,-2.395)(5.7971153,-1.195)(5.7971153,-1.195)
\psline[linecolor=blue, linewidth=0.08](15.397116,-1.195)(16.197115,-2.395)(16.197115,-2.395)
\psline[linecolor=red, linewidth=0.08](5.7971153,-1.195)(6.5971155,-2.395)(6.5971155,-2.395)
\psline[linecolor=red, linewidth=0.08](15.397116,-1.195)(14.5971155,-2.395)(14.5971155,-2.395)
\psline[linecolor=colour0, linewidth=0.08](4.9971156,-2.395)(5.7971153,-3.595)(5.7971153,-3.595)
\psline[linecolor=colour0, linewidth=0.08](16.197115,-2.395)(15.397116,-3.595)(15.397116,-3.595)
\psline[linecolor=yellow, linewidth=0.08](14.5971155,-2.395)(15.397116,-3.595)(15.397116,-3.595)
\psline[linecolor=yellow, linewidth=0.08](6.5971155,-2.395)(5.7971153,-3.595)(5.7971153,-3.595)
\psdots[linecolor=black, dotsize=0.1](4.9971156,-2.395)
\psdots[linecolor=black, dotsize=0.1](4.9971156,-2.395)
\psdots[linecolor=black, dotsize=0.1](4.9971156,-2.395)
\psdots[linecolor=black, dotsize=0.1](4.9971156,-2.395)
\psdots[linecolor=black, dotsize=0.4](4.9971156,-2.395)
\psdots[linecolor=black, dotsize=0.4](5.7971153,-1.195)
\psdots[linecolor=black, dotsize=0.4](6.5971155,-2.395)
\psdots[linecolor=black, dotsize=0.4](5.7971153,-3.595)
\psdots[linecolor=black, dotsize=0.4](14.5971155,-2.395)
\psdots[linecolor=black, dotsize=0.4](15.397116,-1.195)
\psdots[linecolor=black, dotsize=0.4](16.197115,-2.395)
\psdots[linecolor=black, dotsize=0.4](15.397116,-3.595)
\psline[linecolor=colour1, linewidth=0.08, linestyle=dashed, dash=0.17638889cm 0.10583334cm](9.797115,-0.395)(9.797115,-4.395)(9.797115,-4.395)
\psline[linecolor=colour1, linewidth=0.08, linestyle=dashed, dash=0.17638889cm 0.10583334cm](9.797115,-4.395)(9.797115,-4.795)(9.797115,-4.795)
\rput[bl](0.6171154,-5.155){(1)}
\rput[bl](2.2771153,-5.195){(2)}
\rput[bl](5.5971155,-5.115){(i)}
\rput[bl](8.677115,-5.115){(k)}
\rput[bl](10.197115,-5.115){(k+1)}
\rput[bl](18.297115,-5.155){(2k)}
\rput[bl](5.6171155,-0.795){$u_i$}
\rput[bl](6.6371155,-2.855){$v_i$}
\rput[bl](5.6571155,-4.175){$w_i$}
\rput[bl](4.7171154,-2.975){$x_i$}
\psdots[linecolor=black, dotsize=0.4](12.197115,-1.195)
\psdots[linecolor=black, dotsize=0.4](12.997115,-2.395)
\psdots[linecolor=black, dotsize=0.4](12.197115,-3.595)
\psline[linecolor=black, linewidth=0.08](11.397116,-2.395)(12.197115,-1.195)(12.997115,-2.395)(12.197115,-3.595)(11.397116,-2.395)(11.397116,-2.395)
\psline[linecolor=colour1, linewidth=0.08, linestyle=dashed, dash=0.17638889cm 0.10583334cm](11.397116,-0.395)(11.397116,-4.795)(11.397116,-4.795)
\rput[bl](14.997115,-5.195){(2k-i+2)}
\rput[bl](14.837115,-0.895){$u_{2k-i+2}$}
\rput[bl](14.157115,-2.935){$v_{2k-i+2}$}
\rput[bl](15.017116,-4.175){$w_{2k-i+2}$}
\rput[bl](16.017115,-2.955){$x_{2k-i+2}$}
\rput[bl](11.837115,-5.135){(k+2)}
\rput[bl](10.457115,-0.855){a}
\rput[bl](10.937116,-2.535){b}
\rput[bl](10.497115,-4.115){c}
\rput[bl](10.057116,-2.515){d}
\rput[bl](19.857115,-5.135){(2k+1)}
\end{pspicture}
}
\end{center}
\caption{Para-chain square cactus $Q_{2k+1}$. } \label{paraChainsqu2k+1}
\end{figure}
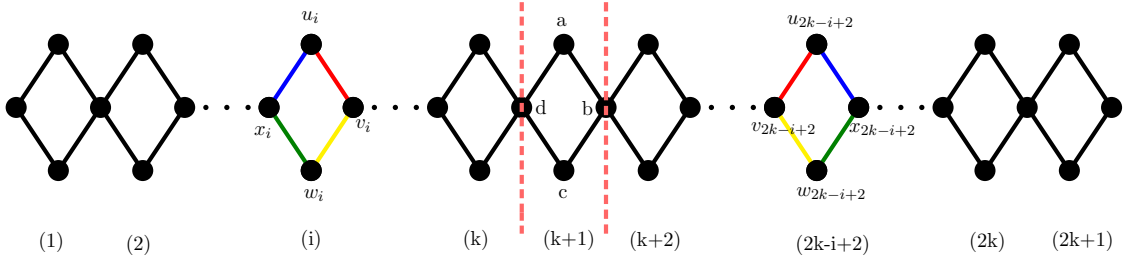

	\begin{theorem} \label{thm-para-O}
Let $O_n$ be the para-chain square cactus graph  of order $n$. Then,
\begin{itemize}
\item[(i)] for every $n\geq 2$, and $k \geq 1$, if $n=2k$, we have:

\begin{align*}
	ABC_{GG}(O_{n})&=2k\sqrt{2}+4\left(\sum_{i=1}^{k}\sqrt{\frac{6k-1}{(6k-3i+2)(3i-1)}}\right),\\
\end{align*}

and if $n=2k+1$, we have:

\begin{align*}
	ABC_{GG}(O_{n})&=(2k+1)\sqrt{2}+\frac{2\sqrt{6k+2}}{3k+2}+4\left(\sum_{i=1}^{k}\sqrt{\frac{6k+2}{(6k-3i+5)(3i-1)}}\right).
\end{align*}

\item[(ii)] for every $n\geq 2$,
 		$ABC(O_n)=\frac{3n+2}{\sqrt{2}}+\frac{(n-2)\sqrt{6}}{4}.$
\end{itemize}
	\end{theorem}

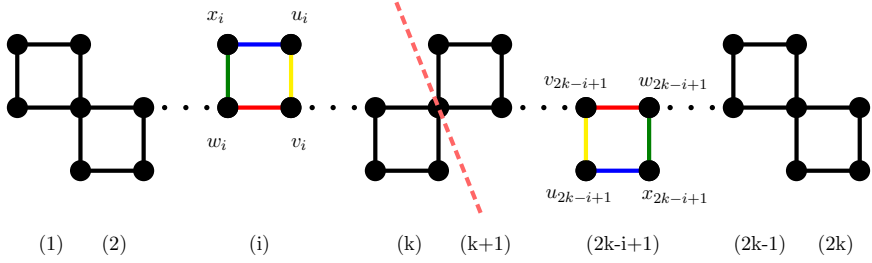
\begin{figure}
\begin{center}
\psscalebox{0.7 0.7}
{
\begin{pspicture}(0,-5.2176895)(16.394232,-0.33259934)
\definecolor{colour0}{rgb}{0.0,0.5019608,0.0}
\definecolor{colour1}{rgb}{1.0,0.4,0.4}
\psdots[linecolor=black, dotsize=0.4](0.1971154,-1.157455)
\psdots[linecolor=black, dotsize=0.4](1.3971153,-1.157455)
\psdots[linecolor=black, dotsize=0.4](0.1971154,-2.357455)
\psdots[linecolor=black, dotsize=0.4](1.3971153,-2.357455)
\psdots[linecolor=black, dotsize=0.4](2.5971155,-2.357455)
\psdots[linecolor=black, dotsize=0.4](1.3971153,-3.557455)
\psdots[linecolor=black, dotsize=0.4](2.5971155,-3.557455)
\psdots[linecolor=black, dotsize=0.1](2.9971154,-2.357455)
\psdots[linecolor=black, dotsize=0.1](3.3971155,-2.357455)
\psdots[linecolor=black, dotsize=0.1](3.7971153,-2.357455)
\psdots[linecolor=black, dotsize=0.4](4.1971154,-2.357455)
\psdots[linecolor=black, dotsize=0.4](4.1971154,-1.157455)
\psdots[linecolor=black, dotsize=0.4](5.397115,-1.157455)
\psdots[linecolor=black, dotsize=0.4](5.397115,-2.357455)
\psdots[linecolor=black, dotsize=0.1](5.7971153,-2.357455)
\psdots[linecolor=black, dotsize=0.1](6.1971154,-2.357455)
\psdots[linecolor=black, dotsize=0.1](6.5971155,-2.357455)
\psdots[linecolor=black, dotsize=0.4](6.9971156,-2.357455)
\psdots[linecolor=black, dotsize=0.4](8.197115,-2.357455)
\psdots[linecolor=black, dotsize=0.4](6.9971156,-3.557455)
\psdots[linecolor=black, dotsize=0.4](8.197115,-3.557455)
\psdots[linecolor=black, dotsize=0.4](8.197115,-1.157455)
\psdots[linecolor=black, dotsize=0.4](9.397116,-1.157455)
\psdots[linecolor=black, dotsize=0.4](8.197115,-2.357455)
\psdots[linecolor=black, dotsize=0.4](9.397116,-2.357455)
\psdots[linecolor=black, dotsize=0.1](9.797115,-2.357455)
\psdots[linecolor=black, dotsize=0.1](10.197115,-2.357455)
\psdots[linecolor=black, dotsize=0.1](10.5971155,-2.357455)
\psdots[linecolor=black, dotsize=0.4](10.997115,-2.357455)
\psdots[linecolor=black, dotsize=0.4](10.997115,-3.557455)
\psdots[linecolor=black, dotsize=0.4](12.197115,-3.557455)
\psdots[linecolor=black, dotsize=0.4](12.197115,-2.357455)
\psdots[linecolor=black, dotsize=0.1](12.5971155,-2.357455)
\psdots[linecolor=black, dotsize=0.1](12.997115,-2.357455)
\psdots[linecolor=black, dotsize=0.1](13.397116,-2.357455)
\psdots[linecolor=black, dotsize=0.4](13.797115,-2.357455)
\psdots[linecolor=black, dotsize=0.4](14.997115,-2.357455)
\psdots[linecolor=black, dotsize=0.4](13.797115,-1.157455)
\psdots[linecolor=black, dotsize=0.4](14.997115,-1.157455)
\psdots[linecolor=black, dotsize=0.4](14.997115,-3.557455)
\psdots[linecolor=black, dotsize=0.4](16.197115,-3.557455)
\psdots[linecolor=black, dotsize=0.4](16.197115,-2.357455)
\psline[linecolor=black, linewidth=0.08](0.1971154,-2.357455)(2.5971155,-2.357455)(2.5971155,-3.557455)(1.3971153,-3.557455)(1.3971153,-1.157455)(0.1971154,-1.157455)(0.1971154,-2.357455)(0.1971154,-2.357455)
\psline[linecolor=black, linewidth=0.08](10.997115,-2.357455)(10.997115,-2.357455)(10.997115,-2.357455)
\psline[linecolor=black, linewidth=0.08](6.9971156,-2.357455)(9.397116,-2.357455)(9.397116,-1.157455)(8.197115,-1.157455)(8.197115,-3.557455)(6.9971156,-3.557455)(6.9971156,-2.357455)(6.9971156,-2.357455)
\psline[linecolor=black, linewidth=0.08](13.797115,-2.357455)(13.797115,-1.157455)(14.997115,-1.157455)(14.997115,-2.357455)(14.997115,-3.557455)(16.197115,-3.557455)(16.197115,-2.357455)(13.797115,-2.357455)(13.797115,-2.357455)
\psline[linecolor=blue, linewidth=0.08](4.1971154,-1.157455)(5.397115,-1.157455)(5.397115,-1.157455)
\psline[linecolor=blue, linewidth=0.08](10.997115,-3.557455)(12.197115,-3.557455)(12.197115,-3.557455)
\psline[linecolor=red, linewidth=0.08](4.1971154,-2.357455)(5.397115,-2.357455)(5.397115,-2.357455)
\psline[linecolor=red, linewidth=0.08](10.997115,-2.357455)(12.197115,-2.357455)(12.197115,-2.357455)
\psline[linecolor=colour0, linewidth=0.08](4.1971154,-1.157455)(4.1971154,-2.357455)(4.1971154,-2.357455)
\psline[linecolor=colour0, linewidth=0.08](12.197115,-2.357455)(12.197115,-3.557455)(12.197115,-3.557455)
\psline[linecolor=yellow, linewidth=0.08](5.397115,-1.157455)(5.397115,-2.357455)(5.397115,-2.357455)
\psline[linecolor=yellow, linewidth=0.08](10.997115,-2.357455)(10.997115,-3.557455)(10.997115,-3.557455)
\psdots[linecolor=black, dotsize=0.4](4.1971154,-2.357455)
\psdots[linecolor=black, dotsize=0.4](4.1971154,-1.157455)
\psdots[linecolor=black, dotsize=0.4](5.397115,-1.157455)
\psdots[linecolor=black, dotsize=0.4](5.397115,-2.357455)
\psdots[linecolor=black, dotsize=0.4](10.997115,-2.357455)
\psdots[linecolor=black, dotsize=0.4](12.197115,-2.357455)
\psdots[linecolor=black, dotsize=0.4](12.197115,-3.557455)
\psdots[linecolor=black, dotsize=0.4](10.997115,-3.557455)
\psline[linecolor=colour1, linewidth=0.08, linestyle=dashed, dash=0.17638889cm 0.10583334cm](7.397115,-0.357455)(8.997115,-4.357455)(8.997115,-4.357455)
\rput[bl](0.5971154,-5.157455){(1)}
\rput[bl](1.7971154,-5.157455){(2)}
\rput[bl](4.5971155,-5.157455){(i)}
\rput[bl](7.397115,-5.157455){(k)}
\rput[bl](8.5971155,-5.157455){(k+1)}
\rput[bl](15.397116,-5.157455){(2k)}
\rput[bl](13.797115,-5.157455){(2k-1)}
\rput[bl](10.997115,-5.157455){(2k-i+1)}
\rput[bl](5.397115,-0.757455){$u_i$}
\rput[bl](5.397115,-3.157455){$v_i$}
\rput[bl](3.7971153,-3.157455){$w_i$}
\rput[bl](3.7971153,-0.757455){$x_i$}
\rput[bl](10.237116,-4.177455){$u_{2k-i+1}$}
\rput[bl](10.197115,-2.017455){$v_{2k-i+1}$}
\rput[bl](11.997115,-2.037455){$w_{2k-i+1}$}
\rput[bl](12.057116,-4.197455){$x_{2k-i+1}$}
\end{pspicture}
}
\end{center}
\caption{Para-chain square cactus $O_{2k}$. } \label{o2k-paraChainsqu2k}
\end{figure}

	\begin{proof}
	\begin{itemize}
\item[(i)] We consider the following cases:

\begin{itemize}
\item[\textbf{Case 1.}] Suppose that $n$ is even and  $n=2k$ for some $k\in \mathbb{N}$.
Now consider the $O_{2k}$ as shown in Figure \ref{o2k-paraChainsqu2k}. One can easily check that whatever happens to computation of Graovac-Ghorbani index related to the edge $u_iv_i$ in the $(i)$-th square in $O_{2k}$, is the same as computation of Graovac-Ghorbani index related to the edge $u_{2k-i+1}v_{2k-i+1}$ in the $(2k-i+1)$-th square. The same goes for $w_iv_i$ and $w_{2k-i+1}v_{2k-i+1}$, for $w_ix_i$ and $w_{2k-i+1}x_{2k-i+1}$, and also for $x_iu_i$ and $x_{2k-i+1}u_{2k-i+1}$. So for computing Graovac-Ghorbani index, it suffices to compute the 
$$\sqrt{\frac{n_u(uv,O_{2k})+n_v(uv,O_{2k})-2}{n_u(uv,O_{2k})n_v(uv,O_{2k})}}$$
  for every $uv \in E(O_{2k})$ in the first $k$ squares and then multiple that by 2. So from now, we only consider the  first $k$ squares.

Consider the yellow edge $u_iv_i$ in the $(i)$-th square. There are $3(2k)-1$ vertices which are closer to $v_i$ than $u_i$, and there are $2$ vertices closer to $u_i$ than $v_i$ which is $x_i$. So, $\sqrt{\frac{n_{u_i}(u_iv_i,O_{2k})+n_{v_i}(u_iv_i,O_{2k})-2}{n_{u_i}(u_iv_i,O_{2k})n_{v_i}(u_iv_i,O_{2k})}}=\frac{\sqrt{2}}{2}$. By the same argument, the same happens to the edge $x_iw_i$.

Now consider the blue edge $u_ix_i$ in the $(i)$-th square. There are $3i-1$ vertices which are closer to $x_i$ than $u_i$, and there are $3k+3(k-i)+2$ vertices closer to $u_i$ than $x_i$. So, $\sqrt{\frac{n_{u_i}(u_ix_i,O_{2k})+n_{x_i}(u_ix_i,O_{2k})-2}{n_{u_i}(u_ix_i,O_{2k})n_{x_i}(u_ix_i,O_{2k})}}=\sqrt{\frac{6k-1}{(6k-3i+2)(3i-1)}}$. By the same argument, the same happens to the edge $v_iw_i$.

Since we have $k$ edges like blue one, $k$ edges like green one, $k$ edges like yellow one and $k$ edges like red one, then by our argument, we have:

\begin{align*}
	ABC_{GG}(O_{2k})=2\left(2\sum_{i=1}^{k}\frac{\sqrt{2}}{2}+2\sum_{i=1}^{k}\sqrt{\frac{6k-1}{(6k-3i+2)(3i-1)}}\right).
\end{align*}

\item[\textbf{Case 2.}] Suppose that $n$ is odd and $n=2k+1$ for some $k\in \mathbb{N}$.
Now consider  the $O_{2k+1}$ as shown in Figure \ref{o2k+1-paraChainsqu2k}. One can easily check that whatever happens to computation of Graovac-Ghorbani index related to the edge $u_iv_i$ in the $(i)$-th square in $O_{2k+1}$, is the same as computation of Graovac-Ghorbani index related to the edge $u_{2k-i+2}v_{2k-i+2}$ in the $(2k-i+2)$-th square. The same goes for $w_iv_i$ and $w_{2k-i+2}v_{2k-i+2}$, for $w_ix_i$ and $w_{2k-i+2}x_{2k-i+2}$, and also for $x_iu_i$ and $x_{2k-i+2}u_{2k-i+2}$. So for computing Graovac-Ghorbani index, it suffices to compute the $\sqrt{\frac{n_u(uv,O_{2k+1})+n_v(uv,O_{2k+1})-2}{n_u(uv,O_{2k+1})n_v(uv,O_{2k+1})}}$  for every $uv \in E(O_{2k+1})$ in the first $k$ squares and then multiple that by 2 and add it to $\sum_{uv\in A}^{}\sqrt{\frac{n_u(uv,O_{2k+1})+n_v(uv,O_{2k+1})-2}{n_u(uv,O_{2k+1})n_v(uv,O_{2k+1})}}$, where  $A = \{ab,bc,cd,da\}$. So from now, we only consider the  first $k+1$ squares. 

Consider the yellow edge $u_iv_i$ in the $(i)$-th square. There are $3(2k+1)-1$ vertices which are closer to $v_i$ than $u_i$, and there are $2$ vertices closer to $u_i$ than $v_i$. So, $\sqrt{\frac{n_{u_i}(u_iv_i,O_{2k+1})+n_{v_i}(u_iv_i,O_{2k+1})-2}{n_{u_i}(u_iv_i,O_{2k+1})n_{v_i}(u_iv_i,O_{2k+1})}}=\frac{\sqrt{2}}{2}$. By the same argument, the same happens to the edge $x_iw_i$.

Now consider the blue edge $u_ix_i$ in the $(i)$-th square. There are $3i-1$ vertices which are closer to $x_i$ than $u_i$, and there are $3(k+1)+3(k-i)+2$ vertices closer to $u_i$ than $x_i$. So, $\sqrt{\frac{n_{u_i}(u_ix_i,O_{2k+1})+n_{x_i}(u_ix_i,O_{2k+1})-2}{n_{u_i}(u_ix_i,O_{2k+1})n_{x_i}(u_ix_i,O_{2k+1})}}=\sqrt{\frac{6k+2}{(6k-3i+5)(3i-1)}}$. By the same argument, the same happens to the edge $v_iw_i$.

Now consider  the middle square. For the  edge $ab$, there are $3k+2$ vertices which are closer to $b$ than $a$, and there are $3k+2$ vertices closer to $a$ than $b$. The edge  $cd$ has the same attitude as $ab$. But for the  edge $ad$, there are $3(2k+1)-1$ vertices which are closer to $d$ than $a$, and there are $2$ vertices closer to $a$ than $d$, and the edge  $bc$ has the same attitude as $ad$. Hence, $\sum_{uv\in A}^{}\sqrt{\frac{n_u(uv,O_{2k+1})+n_v(uv,O_{2k+1})-2}{n_u(uv,O_{2k+1})n_v(uv,O_{2k+1})}}=\frac{2\sqrt{6k+2}}{3k+2}+\sqrt{2}$, where  $A = \{ab,bc,cd,da\}$.

Since we have $k$ edges like blue one, $k$ edges like green one, $k$ edges like yellow one and $k$ edges like red one, then by our argument, we have:

\begin{align*}
	ABC_{GG}(O_{2k+1})&=2\left(2\sum_{i=1}^{k}\frac{\sqrt{2}}{2}+2\sum_{i=1}^{k}\sqrt{\frac{6k+2}{(6k-3i+5)(3i-1)}}\right)\\
	&\quad +\frac{2\sqrt{6k+2}}{3k+2}+\sqrt{2}.
\end{align*}

\end{itemize}	
Therefore, we have the result.
\item[(ii)] It follows from Theorem  \ref{thm-l(q1k)}.
\qed
\end{itemize}
	\end{proof}

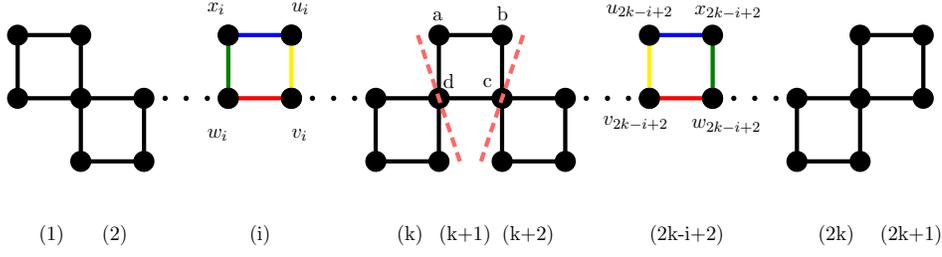
\begin{figure}
\begin{center}
\psscalebox{0.7 0.7}
{
\begin{pspicture}(0,-5.275)(17.707115,-0.625)
\definecolor{colour0}{rgb}{0.0,0.5019608,0.0}
\definecolor{colour1}{rgb}{1.0,0.4,0.4}
\psdots[linecolor=black, dotsize=0.4](0.1971154,-1.275)
\psdots[linecolor=black, dotsize=0.4](1.3971153,-1.275)
\psdots[linecolor=black, dotsize=0.4](0.1971154,-2.475)
\psdots[linecolor=black, dotsize=0.4](1.3971153,-2.475)
\psdots[linecolor=black, dotsize=0.4](2.5971155,-2.475)
\psdots[linecolor=black, dotsize=0.4](1.3971153,-3.675)
\psdots[linecolor=black, dotsize=0.4](2.5971155,-3.675)
\psdots[linecolor=black, dotsize=0.1](2.9971154,-2.475)
\psdots[linecolor=black, dotsize=0.1](3.3971155,-2.475)
\psdots[linecolor=black, dotsize=0.1](3.7971153,-2.475)
\psdots[linecolor=black, dotsize=0.4](4.1971154,-2.475)
\psdots[linecolor=black, dotsize=0.4](4.1971154,-1.275)
\psdots[linecolor=black, dotsize=0.4](5.397115,-1.275)
\psdots[linecolor=black, dotsize=0.4](5.397115,-2.475)
\psdots[linecolor=black, dotsize=0.1](5.7971153,-2.475)
\psdots[linecolor=black, dotsize=0.1](6.1971154,-2.475)
\psdots[linecolor=black, dotsize=0.1](6.5971155,-2.475)
\psdots[linecolor=black, dotsize=0.4](6.9971156,-2.475)
\psdots[linecolor=black, dotsize=0.4](8.197115,-2.475)
\psdots[linecolor=black, dotsize=0.4](6.9971156,-3.675)
\psdots[linecolor=black, dotsize=0.4](8.197115,-3.675)
\psdots[linecolor=black, dotsize=0.4](8.197115,-1.275)
\psdots[linecolor=black, dotsize=0.4](9.397116,-1.275)
\psdots[linecolor=black, dotsize=0.4](8.197115,-2.475)
\psdots[linecolor=black, dotsize=0.4](9.397116,-2.475)
\psdots[linecolor=black, dotsize=0.1](10.997115,-2.475)
\psdots[linecolor=black, dotsize=0.1](11.397116,-2.475)
\psdots[linecolor=black, dotsize=0.1](11.797115,-2.475)
\psdots[linecolor=black, dotsize=0.4](12.197115,-2.475)
\psdots[linecolor=black, dotsize=0.4](13.397116,-2.475)
\psdots[linecolor=black, dotsize=0.1](13.797115,-2.475)
\psdots[linecolor=black, dotsize=0.1](14.197115,-2.475)
\psline[linecolor=black, linewidth=0.08](0.1971154,-2.475)(2.5971155,-2.475)(2.5971155,-3.675)(1.3971153,-3.675)(1.3971153,-1.275)(0.1971154,-1.275)(0.1971154,-2.475)(0.1971154,-2.475)
\psline[linecolor=black, linewidth=0.08](10.997115,-2.475)(10.997115,-2.475)(10.997115,-2.475)
\psline[linecolor=black, linewidth=0.08](6.9971156,-2.475)(9.397116,-2.475)(9.397116,-1.275)(8.197115,-1.275)(8.197115,-3.675)(6.9971156,-3.675)(6.9971156,-2.475)(6.9971156,-2.475)
\psline[linecolor=blue, linewidth=0.08](4.1971154,-1.275)(5.397115,-1.275)(5.397115,-1.275)
\psline[linecolor=blue, linewidth=0.08](12.197115,-1.275)(13.397116,-1.275)(13.397116,-1.275)
\psline[linecolor=red, linewidth=0.08](4.1971154,-2.475)(5.397115,-2.475)(5.397115,-2.475)
\psline[linecolor=red, linewidth=0.08](12.197115,-2.475)(13.397116,-2.475)(13.397116,-2.475)
\psline[linecolor=colour0, linewidth=0.08](4.1971154,-1.275)(4.1971154,-2.475)(4.1971154,-2.475)
\psline[linecolor=colour0, linewidth=0.08](13.397116,-1.275)(13.397116,-2.475)(13.397116,-2.475)
\psline[linecolor=yellow, linewidth=0.08](5.397115,-1.275)(5.397115,-2.475)(5.397115,-2.475)
\psline[linecolor=yellow, linewidth=0.08](12.197115,-1.275)(12.197115,-2.475)(12.197115,-2.475)
\psdots[linecolor=black, dotsize=0.4](4.1971154,-2.475)
\psdots[linecolor=black, dotsize=0.4](4.1971154,-1.275)
\psdots[linecolor=black, dotsize=0.4](5.397115,-1.275)
\psdots[linecolor=black, dotsize=0.4](5.397115,-2.475)
\psdots[linecolor=black, dotsize=0.4](12.197115,-2.475)
\psdots[linecolor=black, dotsize=0.4](13.397116,-2.475)
\psdots[linecolor=black, dotsize=0.4](13.397116,-1.275)
\psdots[linecolor=black, dotsize=0.4](12.197115,-1.275)
\rput[bl](0.5971154,-5.275){(1)}
\rput[bl](1.7971154,-5.275){(2)}
\rput[bl](4.5971155,-5.275){(i)}
\rput[bl](7.397115,-5.275){(k)}
\rput[bl](8.197115,-5.275){(k+1)}
\rput[bl](5.397115,-0.875){$u_i$}
\rput[bl](5.397115,-3.275){$v_i$}
\rput[bl](3.7971153,-3.275){$w_i$}
\rput[bl](3.7971153,-0.875){$x_i$}
\rput[bl](12.197115,-5.275){(2k-i+2)}
\psdots[linecolor=black, dotsize=0.4](10.5971155,-2.475)
\psdots[linecolor=black, dotsize=0.4](9.397116,-3.675)
\psdots[linecolor=black, dotsize=0.4](10.5971155,-3.675)
\psline[linecolor=black, linewidth=0.08](9.397116,-2.475)(10.5971155,-2.475)(10.5971155,-3.675)(9.397116,-3.675)(9.397116,-2.475)(9.397116,-2.475)
\psline[linecolor=colour1, linewidth=0.08, linestyle=dashed, dash=0.17638889cm 0.10583334cm](7.7971153,-1.275)(8.5971155,-3.675)(8.5971155,-3.675)
\psline[linecolor=colour1, linewidth=0.08, linestyle=dashed, dash=0.17638889cm 0.10583334cm](9.797115,-1.275)(8.997115,-3.675)(8.997115,-3.675)
\rput[bl](8.077115,-0.995){a}
\rput[bl](9.297115,-0.995){b}
\rput[bl](9.037115,-2.255){c}
\rput[bl](8.277116,-2.275){d}
\rput[bl](9.397116,-5.275){(k+2)}
\rput[bl](11.377115,-0.955){$u_{2k-i+2}$}
\rput[bl](11.317116,-3.075){$v_{2k-i+2}$}
\rput[bl](12.997115,-3.135){$w_{2k-i+2}$}
\rput[bl](13.057116,-0.975){$x_{2k-i+2}$}
\rput[bl](16.577116,-5.275){(2k+1)}
\rput[bl](15.397116,-5.275){(2k)}
\psdots[linecolor=black, dotsize=0.4](14.997115,-2.475)
\psdots[linecolor=black, dotsize=0.4](16.197115,-2.475)
\psdots[linecolor=black, dotsize=0.4](14.997115,-3.675)
\psdots[linecolor=black, dotsize=0.4](16.197115,-3.675)
\psdots[linecolor=black, dotsize=0.4](16.197115,-1.275)
\psdots[linecolor=black, dotsize=0.4](17.397116,-1.275)
\psdots[linecolor=black, dotsize=0.4](17.397116,-2.475)
\psline[linecolor=black, linewidth=0.08](14.997115,-2.475)(17.397116,-2.475)(17.397116,-1.275)(16.197115,-1.275)(16.197115,-3.675)(14.997115,-3.675)(14.997115,-2.475)(14.997115,-2.475)
\psdots[linecolor=black, dotsize=0.1](14.5971155,-2.475)
\end{pspicture}
}
\end{center}
\caption{Para-chain square cactus $O_{2k+1}$. } \label{o2k+1-paraChainsqu2k}
\end{figure}

	\begin{theorem}
Let $O^h_n$ be the Ortho-chain graph  of order $n$ (See Figure \ref{ortho-chain}). Then,
\begin{itemize}

\item[(i)]
 for every $n\geq 2$, and $k \geq 1$, if $n=2k$, we have:

\begin{align*}
	ABC_{GG}(O^h_n)&=4\left(\sum_{i=1}^{k}\sqrt{\frac{10k-1}{(10k-5i+3)(5i-2)}}\right)+ 8k\sqrt{\frac{10k-1}{30k-6}},\\
\end{align*}

and if $n=2k+1$, we have:

\begin{align*}
	ABC_{GG}(O^h_n)&=4\left(\sum_{i=1}^{k}\sqrt{\frac{10k+4}{(10k-5i+8)(5i-2)}}\right)+ (8k+4)\sqrt{\frac{10k+4}{30k+9}}+\frac{2\sqrt{10k+4}}{5k+3}.\\
\end{align*}

\item[(ii)]
for every $n\geq 2$,
 		$ABC(O_n^h)=\frac{5n+2}{\sqrt{2}}+\frac{(n-2)\sqrt{6}}{4}.$
\end{itemize}
\end{theorem}

\begin{proof}
	\begin{itemize}
	\item[(i)]
It is similar to the proof of Theorem \ref{thm-para-O}.
	\item[(ii)] It follows from Theorem  \ref{thm-l(q1k)}.		\qed
	\end{itemize}
\end{proof}

\begin{figure}
\begin{center}
\psscalebox{0.5 0.5}
{
\begin{pspicture}(0,-6.8)(13.194231,-1.605769)
\psdots[linecolor=black, dotsize=0.4](2.1971154,-1.8028846)
\psdots[linecolor=black, dotsize=0.4](2.1971154,-4.2028847)
\psdots[linecolor=black, dotsize=0.4](2.5971155,-3.0028846)
\psdots[linecolor=black, dotsize=0.4](3.3971155,-3.0028846)
\psdots[linecolor=black, dotsize=0.4](3.7971156,-4.2028847)
\psdots[linecolor=black, dotsize=0.4](3.7971156,-1.8028846)
\psdots[linecolor=black, dotsize=0.4](5.3971157,-1.8028846)
\psdots[linecolor=black, dotsize=0.4](5.7971153,-3.0028846)
\psdots[linecolor=black, dotsize=0.4](5.3971157,-4.2028847)
\psdots[linecolor=black, dotsize=0.4](0.59711546,-1.8028846)
\psdots[linecolor=black, dotsize=0.4](0.19711548,-3.0028846)
\psdots[linecolor=black, dotsize=0.4](0.59711546,-4.2028847)
\psdots[linecolor=black, dotsize=0.4](1.7971154,-5.4028845)
\psdots[linecolor=black, dotsize=0.4](2.1971154,-6.6028843)
\psdots[linecolor=black, dotsize=0.4](3.7971156,-6.6028843)
\psdots[linecolor=black, dotsize=0.4](4.1971154,-5.4028845)
\psdots[linecolor=black, dotsize=0.4](6.9971156,-4.2028847)
\psdots[linecolor=black, dotsize=0.4](4.9971156,-5.4028845)
\psdots[linecolor=black, dotsize=0.4](5.3971157,-6.6028843)
\psdots[linecolor=black, dotsize=0.4](7.3971157,-5.4028845)
\psdots[linecolor=black, dotsize=0.4](6.9971156,-6.6028843)
\psdots[linecolor=black, dotsize=0.4](10.997115,-1.8028846)
\psdots[linecolor=black, dotsize=0.4](10.997115,-4.2028847)
\psdots[linecolor=black, dotsize=0.4](11.397116,-3.0028846)
\psdots[linecolor=black, dotsize=0.4](12.5971155,-4.2028847)
\psdots[linecolor=black, dotsize=0.4](9.397116,-1.8028846)
\psdots[linecolor=black, dotsize=0.4](8.997115,-3.0028846)
\psdots[linecolor=black, dotsize=0.4](9.397116,-4.2028847)
\psdots[linecolor=black, dotsize=0.4](10.5971155,-5.4028845)
\psdots[linecolor=black, dotsize=0.4](10.997115,-6.6028843)
\psdots[linecolor=black, dotsize=0.4](12.5971155,-6.6028843)
\psdots[linecolor=black, dotsize=0.4](12.997115,-5.4028845)
\psline[linecolor=black, linewidth=0.08](6.9971156,-4.2028847)(0.59711546,-4.2028847)(0.19711548,-3.0028846)(0.59711546,-1.8028846)(2.1971154,-1.8028846)(2.5971155,-3.0028846)(2.1971154,-4.2028847)(1.7971154,-5.4028845)(2.1971154,-6.6028843)(3.7971156,-6.6028843)(4.1971154,-5.4028845)(3.7971156,-4.2028847)(3.3971155,-3.0028846)(3.7971156,-1.8028846)(5.3971157,-1.8028846)(5.7971153,-3.0028846)(5.3971157,-4.2028847)(4.9971156,-5.4028845)(5.3971157,-6.6028843)(6.9971156,-6.6028843)(7.3971157,-5.4028845)(6.9971156,-4.2028847)(6.9971156,-4.2028847)
\psline[linecolor=black, linewidth=0.08](9.397116,-4.2028847)(12.5971155,-4.2028847)(12.997115,-5.4028845)(12.5971155,-6.6028843)(10.997115,-6.6028843)(10.5971155,-5.4028845)(11.397116,-3.0028846)(10.997115,-1.8028846)(9.397116,-1.8028846)(8.997115,-3.0028846)(9.397116,-4.2028847)(9.397116,-4.2028847)
\psline[linecolor=black, linewidth=0.08](7.3971157,-4.2028847)(6.9971156,-4.2028847)(6.9971156,-4.2028847)
\psline[linecolor=black, linewidth=0.08](8.997115,-4.2028847)(9.397116,-4.2028847)(9.397116,-4.2028847)
\psdots[linecolor=black, dotsize=0.1](8.197116,-4.2028847)
\psdots[linecolor=black, dotsize=0.1](7.7971153,-4.2028847)
\psdots[linecolor=black, dotsize=0.1](8.5971155,-4.2028847)
\end{pspicture}
}
\end{center}
\caption{Ortho-chain graph  $O^h_n$. } \label{ortho-chain}
\end{figure}
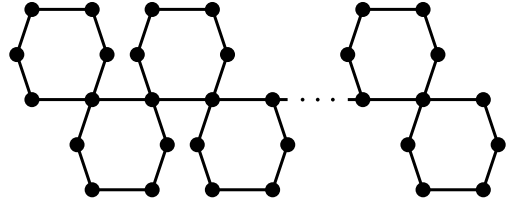

	\begin{theorem}
Let $L_n$ be the para-chain hexagonal graph  of order $n$ (See Figure \ref{para-chain}). Then,
\begin{itemize}
\item[(i)]
 for every $n\geq 1$, and $k \geq 1$, we have:

 	\[
 	ABC_{GG}(L_n)=\left\{
  	\begin{array}{ll}
  	{\displaystyle
  		12\sum_{i=1}^{k}\sqrt{\frac{10k-1}{(10k-5i+3)(5i-2)}}}&
  		\quad\mbox{if $n=2k$, }\\[15pt]
  		{\displaystyle
  			12\left(\sum_{i=1}^{k}\sqrt{\frac{10k+4}{(10k-5i+8)(5i-2)}}\right)+\frac{6\sqrt{10k+4}}{5k+3}}&
  			\quad\mbox{if $n=2k+1$.}
  				  \end{array}
  					\right.	
  					\]
\item[(ii)] for every $n\geq 2$,
 		$ABC(L_n)=3n\sqrt{2}.$
\end{itemize}
\end{theorem}

\begin{proof}
	\begin{itemize}
	\item[(i)]
It is similar to the proof of Theorem \ref{thm-para-Q}.
	\item[(ii)] It follows from Theorem  \ref{thm-l(qhk)}.		\qed
	\end{itemize}
\end{proof}

\begin{figure}
\begin{center}
\psscalebox{0.5 0.5}
{
\begin{pspicture}(0,-5.6)(16.794231,-2.805769)
\psdots[linecolor=black, dotsize=0.4](2.1971154,-3.0028846)
\psdots[linecolor=black, dotsize=0.4](2.1971154,-5.4028845)
\psdots[linecolor=black, dotsize=0.4](2.9971154,-4.2028847)
\psdots[linecolor=black, dotsize=0.4](2.9971154,-4.2028847)
\psdots[linecolor=black, dotsize=0.4](3.7971153,-5.4028845)
\psdots[linecolor=black, dotsize=0.4](3.7971153,-3.0028846)
\psdots[linecolor=black, dotsize=0.4](4.9971156,-3.0028846)
\psdots[linecolor=black, dotsize=0.4](5.7971153,-4.2028847)
\psdots[linecolor=black, dotsize=0.4](4.9971156,-5.4028845)
\psdots[linecolor=black, dotsize=0.4](0.9971154,-3.0028846)
\psdots[linecolor=black, dotsize=0.4](0.19711538,-4.2028847)
\psdots[linecolor=black, dotsize=0.4](0.9971154,-5.4028845)
\psdots[linecolor=black, dotsize=0.4](7.7971153,-3.0028846)
\psdots[linecolor=black, dotsize=0.4](7.7971153,-5.4028845)
\psdots[linecolor=black, dotsize=0.4](8.5971155,-4.2028847)
\psdots[linecolor=black, dotsize=0.4](8.5971155,-4.2028847)
\psdots[linecolor=black, dotsize=0.4](6.5971155,-3.0028846)
\psdots[linecolor=black, dotsize=0.4](5.7971153,-4.2028847)
\psdots[linecolor=black, dotsize=0.4](6.5971155,-5.4028845)
\psdots[linecolor=black, dotsize=0.4](12.997115,-3.0028846)
\psdots[linecolor=black, dotsize=0.4](12.997115,-5.4028845)
\psdots[linecolor=black, dotsize=0.4](13.797115,-4.2028847)
\psdots[linecolor=black, dotsize=0.4](13.797115,-4.2028847)
\psdots[linecolor=black, dotsize=0.4](14.5971155,-5.4028845)
\psdots[linecolor=black, dotsize=0.4](14.5971155,-3.0028846)
\psdots[linecolor=black, dotsize=0.4](15.797115,-3.0028846)
\psdots[linecolor=black, dotsize=0.4](16.597115,-4.2028847)
\psdots[linecolor=black, dotsize=0.4](15.797115,-5.4028845)
\psdots[linecolor=black, dotsize=0.4](11.797115,-3.0028846)
\psdots[linecolor=black, dotsize=0.4](10.997115,-4.2028847)
\psdots[linecolor=black, dotsize=0.4](11.797115,-5.4028845)
\psdots[linecolor=black, dotsize=0.4](8.5971155,-4.2028847)
\psdots[linecolor=black, dotsize=0.4](8.5971155,-4.2028847)
\psline[linecolor=black, linewidth=0.08](0.19711538,-4.2028847)(0.9971154,-3.0028846)(2.1971154,-3.0028846)(2.9971154,-4.2028847)(3.7971153,-3.0028846)(4.9971156,-3.0028846)(5.7971153,-4.2028847)(6.5971155,-3.0028846)(7.7971153,-3.0028846)(8.5971155,-4.2028847)(7.7971153,-5.4028845)(6.5971155,-5.4028845)(5.7971153,-4.2028847)(4.9971156,-5.4028845)(3.7971153,-5.4028845)(2.9971154,-4.2028847)(2.1971154,-5.4028845)(0.9971154,-5.4028845)(0.19711538,-4.2028847)(0.19711538,-4.2028847)
\psline[linecolor=black, linewidth=0.08](10.997115,-4.2028847)(11.797115,-3.0028846)(12.997115,-3.0028846)(13.797115,-4.2028847)(14.5971155,-3.0028846)(15.797115,-3.0028846)(16.597115,-4.2028847)(15.797115,-5.4028845)(14.5971155,-5.4028845)(13.797115,-4.2028847)(12.997115,-5.4028845)(11.797115,-5.4028845)(10.997115,-4.2028847)(10.997115,-4.2028847)
\psdots[linecolor=black, dotsize=0.1](9.397116,-4.2028847)
\psdots[linecolor=black, dotsize=0.1](9.797115,-4.2028847)
\psdots[linecolor=black, dotsize=0.1](10.197115,-4.2028847)
\end{pspicture}
}
\end{center}
\caption{Para-chain hexagonal graph  $L_n$. } \label{para-chain}
\end{figure}
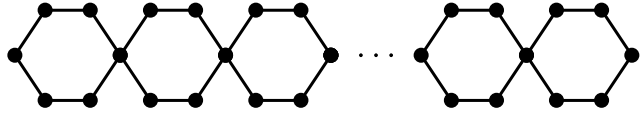

	\begin{theorem}
Let $M_n$ be the Meta-chain hexagonal  of order $n$ (See Figure \ref{Meta-chain}).  Then,
\begin{itemize}
\item[(i)]
for every $n\geq 2$, and $k \geq 1$, if $n=2k$, we have:

\begin{align*}
	ABC_{GG}(M_n)&=8\left(\sum_{i=1}^{k}\sqrt{\frac{10k-1}{(10k-5i+3)(5i-2)}}\right)+ 4k\sqrt{\frac{10k-1}{30k-6}},\\
\end{align*}

and if $n=2k+1$, we have:

\begin{align*}
	ABC_{GG}(M_n)&=8\left(\sum_{i=1}^{k}\sqrt{\frac{10k+4}{(10k-5i+8)(5i-2)}}\right)+ (2k+2)\sqrt{\frac{10k+4}{30k+9}}+\frac{4\sqrt{10k+4}}{5k+3}.\\
\end{align*}
\item[(ii)]  for every $n\geq 2$,
 		$ABC(M_n)=3n\sqrt{2}.$
\end{itemize}
	\end{theorem}

\begin{proof}
	\begin{itemize}
	\item[(i)]
It is similar to the proof of Theorem \ref{thm-para-O}.
	\item[(ii)] It follows from Theorem  \ref{thm-l(qhk)}.		\qed
	\end{itemize}
\end{proof}

\begin{corollary}
	Meta-chain hexagonal
	cactus graphs  and para-chain hexagonal
	cactus graphs of the same order, have the same atom-bond connectivity index. But they do not have the same Graovac-Ghorbani index.
\end{corollary}

 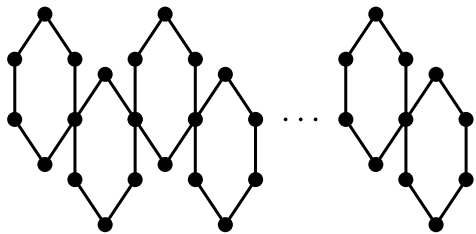
\begin{figure}
	\begin{center}
		\psscalebox{0.5 0.5}
		{
			\begin{pspicture}(0,-6.4)(12.394231,-0.40576905)
			\psdots[linecolor=black, dotsize=0.4](0.9971155,-0.60288453)
			\psdots[linecolor=black, dotsize=0.4](1.7971154,-1.8028846)
			\psdots[linecolor=black, dotsize=0.4](1.7971154,-3.4028845)
			\psdots[linecolor=black, dotsize=0.4](0.9971155,-4.6028843)
			\psdots[linecolor=black, dotsize=0.4](0.19711548,-1.8028846)
			\psdots[linecolor=black, dotsize=0.4](0.19711548,-3.4028845)
			\psdots[linecolor=black, dotsize=0.4](2.5971155,-2.2028844)
			\psdots[linecolor=black, dotsize=0.4](3.3971155,-3.4028845)
			\psdots[linecolor=black, dotsize=0.4](1.7971154,-5.0028844)
			\psdots[linecolor=black, dotsize=0.4](3.3971155,-5.0028844)
			\psdots[linecolor=black, dotsize=0.4](2.5971155,-6.2028847)
			\psdots[linecolor=black, dotsize=0.4](4.1971154,-0.60288453)
			\psdots[linecolor=black, dotsize=0.4](4.9971156,-1.8028846)
			\psdots[linecolor=black, dotsize=0.4](4.9971156,-3.4028845)
			\psdots[linecolor=black, dotsize=0.4](4.1971154,-4.6028843)
			\psdots[linecolor=black, dotsize=0.4](3.3971155,-1.8028846)
			\psdots[linecolor=black, dotsize=0.4](3.3971155,-3.4028845)
			\psdots[linecolor=black, dotsize=0.4](5.7971153,-2.2028844)
			\psdots[linecolor=black, dotsize=0.4](6.5971155,-3.4028845)
			\psdots[linecolor=black, dotsize=0.4](4.9971156,-5.0028844)
			\psdots[linecolor=black, dotsize=0.4](6.5971155,-5.0028844)
			\psdots[linecolor=black, dotsize=0.4](5.7971153,-6.2028847)
			\psdots[linecolor=black, dotsize=0.1](7.3971157,-3.4028845)
			\psdots[linecolor=black, dotsize=0.1](7.7971153,-3.4028845)
			\psdots[linecolor=black, dotsize=0.1](8.197116,-3.4028845)
			\psdots[linecolor=black, dotsize=0.4](9.797115,-0.60288453)
			\psdots[linecolor=black, dotsize=0.4](10.5971155,-1.8028846)
			\psdots[linecolor=black, dotsize=0.4](10.5971155,-3.4028845)
			\psdots[linecolor=black, dotsize=0.4](9.797115,-4.6028843)
			\psdots[linecolor=black, dotsize=0.4](8.997115,-1.8028846)
			\psdots[linecolor=black, dotsize=0.4](8.997115,-3.4028845)
			\psdots[linecolor=black, dotsize=0.4](11.397116,-2.2028844)
			\psdots[linecolor=black, dotsize=0.4](12.197116,-3.4028845)
			\psdots[linecolor=black, dotsize=0.4](10.5971155,-5.0028844)
			\psdots[linecolor=black, dotsize=0.4](12.197116,-5.0028844)
			\psdots[linecolor=black, dotsize=0.4](11.397116,-6.2028847)
			\psline[linecolor=black, linewidth=0.08](0.9971155,-0.60288453)(1.7971154,-1.8028846)(1.7971154,-5.0028844)(2.5971155,-6.2028847)(3.3971155,-5.0028844)(3.3971155,-1.8028846)(4.1971154,-0.60288453)(4.9971156,-1.8028846)(4.9971156,-5.0028844)(5.7971153,-6.2028847)(6.5971155,-5.0028844)(6.5971155,-3.4028845)(5.7971153,-2.2028844)(4.9971156,-3.4028845)(4.1971154,-4.6028843)(2.5971155,-2.2028844)(0.9971155,-4.6028843)(0.19711548,-3.4028845)(0.19711548,-1.8028846)(0.9971155,-0.60288453)(0.9971155,-0.60288453)
			\psline[linecolor=black, linewidth=0.08](11.397116,-2.2028844)(9.797115,-4.6028843)(8.997115,-3.4028845)(8.997115,-1.8028846)(9.797115,-0.60288453)(10.5971155,-1.8028846)(10.5971155,-5.0028844)(11.397116,-6.2028847)(12.197116,-5.0028844)(12.197116,-3.4028845)(11.397116,-2.2028844)(11.397116,-2.2028844)
			\end{pspicture}
		}
	\end{center}
	\caption{Meta-chain hexagonal graph  $M_n$. } \label{Meta-chain}
\end{figure}

\subsection{Polyphenylenes}
Similar to the above definition of the spiro-chain $S_{q,h,k}$, we can define the graph 
$L_{q,h,k}$ 
as the link of $k$ cycles $C_q$ in which the distance between the two contact vertices in the
same cycle is $h$ (see $L_{6,2,4}$ in Figure \ref{L624}). 
	\begin{theorem}
	For the graph  $L_{q,h,k}$, when $h\geq 2$, we have:
\begin{align*}
ABC(L_{q,h,k})=\frac{2(k-1)}{3}+\frac{qk}{\sqrt{2}}.
	\end{align*}
	\end{theorem}
	
		\begin{proof}
	There are $k-1$ edges with endpoints of degree 3. Also there are $4(k-1)$ edges with endpoints of degree 3 and 2, and there are $qk-4(k-1)$ edges with endpoints of degree 2. Therefore 
	\begin{align*}
	ABC(L_{q,h,k})=(k-1)\sqrt{\frac{3+3-2}{3(3)}}+4(k-1)\sqrt{\frac{3+2-2}{3(2)}}+ (qk-4(k-1))\sqrt{\frac{2+2-2}{2(2)}},
	\end{align*}
	and we have the result.	
	\qed
	\end{proof}

	\begin{theorem}
	For the graph  $L_{q,1,k}$, we have:
\begin{align*}
ABC(L_{q,1,k})=\frac{4k-6}{3}+\frac{qk-k+2}{\sqrt{2}}.
	\end{align*}
	\end{theorem}

	\begin{proof}
	There are $2k-3$ edges with endpoints of degree 3. Also there are $2k$ edges with endpoints of degree 3 and 2, and there are $qk-3k+2$ edges with endpoints of degree 2. Therefore, by the definition of the atom-bond connectivity index,
we have the result.	
	\qed
	\end{proof}

\begin{figure}[!h]
\begin{center}
\psscalebox{0.6 0.6}
{
\begin{pspicture}(0,-3.8000002)(10.394231,-0.6057692)
\psdots[linecolor=black, dotsize=0.4](0.9971155,-0.8028845)
\psdots[linecolor=black, dotsize=0.4](0.19711548,-1.6028845)
\psdots[linecolor=black, dotsize=0.4](1.7971154,-1.6028845)
\psdots[linecolor=black, dotsize=0.4](0.19711548,-2.8028846)
\psdots[linecolor=black, dotsize=0.4](1.7971154,-2.8028846)
\psdots[linecolor=black, dotsize=0.4](0.9971155,-3.6028845)
\psdots[linecolor=black, dotsize=0.4](2.9971154,-2.8028846)
\psline[linecolor=black, linewidth=0.08](0.9971155,-0.8028845)(0.19711548,-1.6028845)(0.19711548,-1.6028845)
\psline[linecolor=black, linewidth=0.08](0.9971155,-0.8028845)(1.7971154,-1.6028845)(1.7971154,-1.6028845)(1.7971154,-2.8028846)(1.7971154,-2.8028846)
\psline[linecolor=black, linewidth=0.08](2.9971154,-2.8028846)(1.7971154,-2.8028846)(0.9971155,-3.6028845)(0.19711548,-2.8028846)(0.19711548,-1.6028845)(0.19711548,-1.6028845)
\psdots[linecolor=black, dotsize=0.4](3.7971156,-0.8028845)
\psdots[linecolor=black, dotsize=0.4](2.9971154,-1.6028845)
\psdots[linecolor=black, dotsize=0.4](4.5971155,-1.6028845)
\psdots[linecolor=black, dotsize=0.4](2.9971154,-2.8028846)
\psdots[linecolor=black, dotsize=0.4](4.5971155,-2.8028846)
\psdots[linecolor=black, dotsize=0.4](3.7971156,-3.6028845)
\psdots[linecolor=black, dotsize=0.4](5.7971153,-2.8028846)
\psline[linecolor=black, linewidth=0.08](3.7971156,-0.8028845)(2.9971154,-1.6028845)(2.9971154,-1.6028845)
\psline[linecolor=black, linewidth=0.08](3.7971156,-0.8028845)(4.5971155,-1.6028845)(4.5971155,-1.6028845)(4.5971155,-2.8028846)(4.5971155,-2.8028846)
\psline[linecolor=black, linewidth=0.08](5.7971153,-2.8028846)(4.5971155,-2.8028846)(3.7971156,-3.6028845)(2.9971154,-2.8028846)(2.9971154,-1.6028845)(2.9971154,-1.6028845)
\psdots[linecolor=black, dotsize=0.4](6.5971155,-0.8028845)
\psdots[linecolor=black, dotsize=0.4](5.7971153,-1.6028845)
\psdots[linecolor=black, dotsize=0.4](7.3971157,-1.6028845)
\psdots[linecolor=black, dotsize=0.4](5.7971153,-2.8028846)
\psdots[linecolor=black, dotsize=0.4](7.3971157,-2.8028846)
\psdots[linecolor=black, dotsize=0.4](6.5971155,-3.6028845)
\psdots[linecolor=black, dotsize=0.4](8.5971155,-2.8028846)
\psline[linecolor=black, linewidth=0.08](6.5971155,-0.8028845)(5.7971153,-1.6028845)(5.7971153,-1.6028845)
\psline[linecolor=black, linewidth=0.08](6.5971155,-0.8028845)(7.3971157,-1.6028845)(7.3971157,-1.6028845)(7.3971157,-2.8028846)(7.3971157,-2.8028846)
\psline[linecolor=black, linewidth=0.08](8.5971155,-2.8028846)(7.3971157,-2.8028846)(6.5971155,-3.6028845)(5.7971153,-2.8028846)(5.7971153,-1.6028845)(5.7971153,-1.6028845)
\psdots[linecolor=black, dotsize=0.4](9.397116,-0.8028845)
\psdots[linecolor=black, dotsize=0.4](8.5971155,-1.6028845)
\psdots[linecolor=black, dotsize=0.4](10.197116,-1.6028845)
\psdots[linecolor=black, dotsize=0.4](8.5971155,-2.8028846)
\psdots[linecolor=black, dotsize=0.4](10.197116,-2.8028846)
\psdots[linecolor=black, dotsize=0.4](9.397116,-3.6028845)
\psline[linecolor=black, linewidth=0.08](9.397116,-0.8028845)(8.5971155,-1.6028845)(8.5971155,-1.6028845)
\psline[linecolor=black, linewidth=0.08](9.397116,-0.8028845)(10.197116,-1.6028845)(10.197116,-1.6028845)(10.197116,-2.8028846)(10.197116,-2.8028846)
\psline[linecolor=black, linewidth=0.08](10.197116,-2.8028846)(9.397116,-3.6028845)(8.5971155,-2.8028846)(8.5971155,-1.6028845)(8.5971155,-1.6028845)
\end{pspicture}
}
\end{center}
	\caption{The graph $L_{6,2,4}$.}\label{L624}
\end{figure}
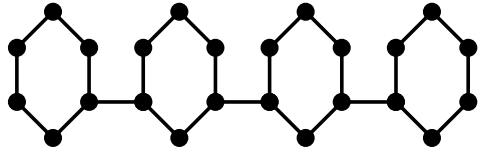

\subsection{Triangulanes}
We intend to derive the atom-bond connectivity of the triangulane $T_k$ defined pictorially in \cite{Khalifeh}.
We define $T_k$ recursively in a manner that will be useful in our approach. First we define
recursively an auxiliary family of triangulanes $G_k$ $(k\geq 1)$. Let $G_1$ be a triangle and denote one of its vertices by $y_1$. We define $G_k$ $(k\geq 2)$ as the circuit of the graphs $G_{k-1}, G_{k-1}$,
and $K_1$ and denote by $y_k$ the vertex where $K_1$ has been placed. The graphs $G_1, G_2$ and $G_3$ 
are shown in Figure \ref{triang}.

	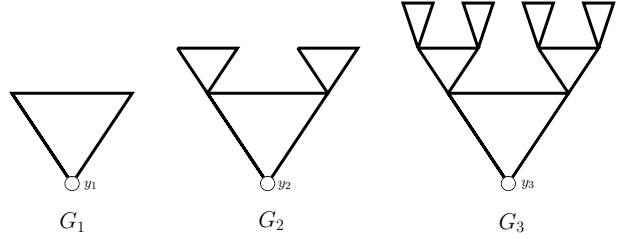
\begin{figure}[!h]
		\begin{center}
			\psscalebox{0.5 0.5}
			{
				\begin{pspicture}(0,-7.8216667)(16.150236,-1.6783332)
				\psline[linecolor=black, linewidth=0.08](7.6847405,-2.928333)(9.28474,-2.928333)(8.48474,-4.128333)(7.6847405,-2.928333)(7.6847405,-2.928333)
				\psline[linecolor=black, linewidth=0.08](10.884741,-2.928333)(11.28474,-1.7283331)(10.48474,-1.7283331)(10.884741,-2.928333)(10.884741,-2.928333)
				\psline[linecolor=black, linewidth=0.08](5.2847404,-4.128333)(6.884741,-6.528333)(8.48474,-4.128333)(5.2847404,-4.128333)(6.884741,-6.528333)(6.884741,-6.528333)
				\psline[linecolor=black, linewidth=0.08](0.0847406,-4.128333)(1.6847405,-6.528333)(3.2847407,-4.128333)(0.0847406,-4.128333)(1.6847405,-6.528333)(1.6847405,-6.528333)
				\psline[linecolor=black, linewidth=0.08](4.4847407,-2.928333)(6.0847406,-2.928333)(5.2847404,-4.128333)(4.4847407,-2.928333)(4.4847407,-2.928333)
				\psline[linecolor=black, linewidth=0.08](14.084741,-2.928333)(15.684741,-2.928333)(14.884741,-4.128333)(14.084741,-2.928333)(14.084741,-2.928333)
				\psline[linecolor=black, linewidth=0.08](11.684741,-4.128333)(13.28474,-6.528333)(14.884741,-4.128333)(11.684741,-4.128333)(13.28474,-6.528333)(13.28474,-6.528333)
				\psline[linecolor=black, linewidth=0.08](10.884741,-2.928333)(12.48474,-2.928333)(11.684741,-4.128333)(10.884741,-2.928333)(10.884741,-2.928333)
				\psline[linecolor=black, linewidth=0.08](12.48474,-2.928333)(12.884741,-1.7283331)(12.084741,-1.7283331)(12.48474,-2.928333)(12.48474,-2.928333)
				\psline[linecolor=black, linewidth=0.08](14.084741,-2.928333)(14.48474,-1.7283331)(13.684741,-1.7283331)(14.084741,-2.928333)(14.084741,-2.928333)
				\psline[linecolor=black, linewidth=0.08](15.684741,-2.928333)(16.08474,-1.7283331)(15.28474,-1.7283331)(15.684741,-2.928333)(15.684741,-2.928333)
				\psdots[linecolor=black, dotstyle=o, dotsize=0.4, fillcolor=white](1.6847405,-6.528333)
				\psdots[linecolor=black, dotstyle=o, dotsize=0.4, fillcolor=white](6.884741,-6.528333)
				\psdots[linecolor=black, dotstyle=o, dotsize=0.4, fillcolor=white](13.28474,-6.528333)
				\rput[bl](2.0047407,-6.7016664){$y_1$}
				\rput[bl](7.1647406,-6.688333){$y_2$}
				\rput[bl](13.631408,-6.6749997){$y_3$}
				\rput[bl](1.1647406,-7.808333){\begin{LARGE}
					$G_1$
					\end{LARGE}}
				\rput[bl](6.458074,-7.768333){\begin{LARGE}
					$G_2$
					\end{LARGE}}
				\rput[bl](12.844741,-7.8216662){\begin{LARGE}
					$G_3$
					\end{LARGE}}
				\end{pspicture}
			}
		\end{center}
		\caption{Graphs $G_1$, $G_2$ and $G_3$.}\label{triang}
	\end{figure}

	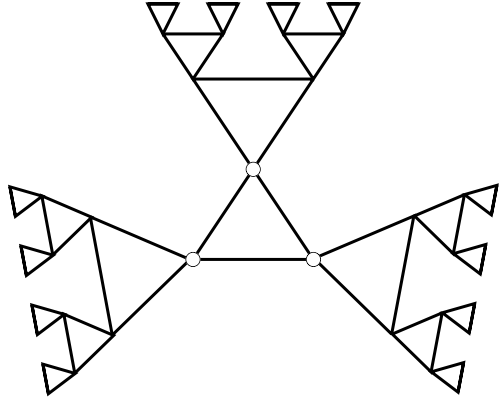
\begin{figure}
		\begin{center}
			\psscalebox{0.5 0.5}
			{
				\begin{pspicture}(0,-8.002529)(13.070843,2.5025294)
				\psline[linecolor=black, linewidth=0.08](3.7284591,2.4525292)(4.528459,2.4525292)(4.128459,1.6525292)(3.7284591,2.4525292)(4.128459,2.4525292)(4.528459,2.4525292)
				\psline[linecolor=black, linewidth=0.08](5.328459,2.4525292)(6.128459,2.4525292)(5.728459,1.6525292)(5.328459,2.4525292)(5.728459,2.4525292)(6.128459,2.4525292)
				\psline[linecolor=black, linewidth=0.08](4.128459,1.6525292)(5.728459,1.6525292)(4.928459,0.4525293)(4.128459,1.6525292)(4.128459,1.6525292)
				\psline[linecolor=black, linewidth=0.08](6.928459,2.4525292)(7.728459,2.4525292)(7.328459,1.6525292)(6.928459,2.4525292)(7.328459,2.4525292)(7.728459,2.4525292)
				\psline[linecolor=black, linewidth=0.08](8.528459,2.4525292)(9.328459,2.4525292)(8.928459,1.6525292)(8.528459,2.4525292)(8.928459,2.4525292)(9.328459,2.4525292)
				\psline[linecolor=black, linewidth=0.08](7.328459,1.6525292)(8.928459,1.6525292)(8.128459,0.4525293)(7.328459,1.6525292)(7.328459,1.6525292)
				\psline[linecolor=black, linewidth=0.08](8.128459,0.4525293)(4.928459,0.4525293)(6.528459,-1.9474707)(8.128459,0.4525293)(8.128459,0.4525293)
				\psline[linecolor=black, linewidth=0.08](13.0095825,-2.376164)(12.861903,-3.162415)(12.149491,-2.62161)(13.0095825,-2.376164)(12.935742,-2.7692895)(12.861903,-3.162415)
				\psline[linecolor=black, linewidth=0.08](12.714223,-3.948666)(12.566544,-4.734917)(11.854133,-4.194112)(12.714223,-3.948666)(12.640384,-4.3417916)(12.566544,-4.734917)
				\psline[linecolor=black, linewidth=0.08](12.144952,-2.6238508)(11.849592,-4.196353)(10.817896,-3.1885824)(12.144952,-2.6238508)(12.144952,-2.6238508)
				\psline[linecolor=black, linewidth=0.08](12.418864,-5.521168)(12.271184,-6.3074193)(11.558773,-5.766614)(12.418864,-5.521168)(12.345024,-5.914294)(12.271184,-6.3074193)
				\psline[linecolor=black, linewidth=0.08](12.123505,-7.0936704)(11.975825,-7.8799214)(11.263413,-7.339116)(12.123505,-7.0936704)(12.0496645,-7.486796)(11.975825,-7.8799214)
				\psline[linecolor=black, linewidth=0.08](11.554234,-5.768855)(11.258874,-7.341357)(10.227177,-6.3335867)(11.554234,-5.768855)(11.554234,-5.768855)
				\psline[linecolor=black, linewidth=0.08](10.210442,-6.336506)(10.80116,-3.1915019)(8.147048,-4.320965)(10.210442,-6.336506)(10.210442,-6.336506)
				\psline[linecolor=black, linewidth=0.08](1.0839291,-7.92159)(0.93781745,-7.135046)(1.7974172,-7.3822064)(1.0839291,-7.92159)(1.0108732,-7.528318)(0.93781745,-7.135046)
				\psline[linecolor=black, linewidth=0.08](0.7917058,-6.348502)(0.6455942,-5.5619583)(1.505194,-5.8091187)(0.7917058,-6.348502)(0.71865,-5.95523)(0.6455942,-5.5619583)
				\psline[linecolor=black, linewidth=0.08](1.7854072,-7.3759704)(1.4931839,-5.802882)(2.8191113,-6.370259)(1.7854072,-7.3759704)(1.7854072,-7.3759704)
				\psline[linecolor=black, linewidth=0.08](0.49948257,-4.7754145)(0.35337096,-3.9888704)(1.2129707,-4.2360306)(0.49948257,-4.7754145)(0.42642677,-4.382142)(0.35337096,-3.9888704)
				\psline[linecolor=black, linewidth=0.08](0.20725933,-3.2023263)(0.061147712,-2.4157825)(0.92074746,-2.6629426)(0.20725933,-3.2023263)(0.13420352,-2.8090544)(0.061147712,-2.4157825)
				\psline[linecolor=black, linewidth=0.08](1.2009606,-4.2297945)(0.9087374,-2.6567066)(2.234665,-3.224083)(1.2009606,-4.2297945)(1.2009606,-4.2297945)
				\psline[linecolor=black, linewidth=0.08](2.197415,-3.215118)(2.7818615,-6.361294)(4.8492703,-4.3498707)(2.197415,-3.215118)(2.197415,-3.215118)
				\psline[linecolor=black, linewidth=0.08](6.528459,-1.9474707)(4.928459,-4.3474708)(8.128459,-4.3474708)(6.528459,-1.9474707)(6.528459,-1.9474707)
				\psdots[linecolor=black, dotstyle=o, dotsize=0.4, fillcolor=white](6.528459,-1.9474707)
				\psdots[linecolor=black, dotstyle=o, dotsize=0.4, fillcolor=white](4.928459,-4.3474708)
				\psdots[linecolor=black, dotstyle=o, dotsize=0.4, fillcolor=white](8.128459,-4.3474708)
				\end{pspicture}
			}
		\end{center}
		\caption{Graph $T_3$.}\label{T3}
	\end{figure}

	\begin{theorem} 
	For the graph  $T_k$ (see $T_3$ in Figure \ref{T3}), we have:
	\begin{itemize}
	\item[(i)]
\begin{align*}
ABC(T_k)=\frac{9(2^{k-1})\sqrt{2}}{2}+\frac{(9(2^k)-6)\sqrt{6}}{4}.
	\end{align*}
	\item[(ii)]
			\begin{align*}
		ABC_{GG}(T_n)&=6\sqrt{\frac{2^{n+2}+2^{n}-4}{(2^{n+2}-1)(2^{n}-1)}}+\frac{3\sqrt{2^{n+2}-4}}{2^{n+1}-1}\\
		&\quad+ \sum_{i=2}^{n}3(2^i)\left(\sqrt{\frac{2^{n+2}+(\sum_{t=0}^{i-2}2^{n-t})+2^{n-i+1}-4}{(2^{n+2}-1+\sum_{t=0}^{i-2}2^{n-t})(2^{n-i+1}-1)}}\right)\\
		&\quad+ \sum_{i=1}^{n}3(2^{i-1})\left(\frac{\sqrt{2^{n-i+2}-4}}{2^{n-i+1}-1}\right).
		\end{align*}
	\end{itemize}	
	\end{theorem}

	\begin{proof}
		\begin{itemize}
	\item[(i)] Since creating such a graph is recursive, then there are $3+3\sum_{n=0}^{k-1}3(2^n)$ edges with endpoints of degree 4. Also there are $3(2^k)$ edges with endpoints of degree 4 and 2, and there are $3(2^{k-1})$ edges with endpoints of degree 2. Therefore, by the definition of the atom-bond connectivity index,
	and we have the result.	
	\item[(ii)] Consider the graph $T_n$ in Figure \ref{Tn}. First we consider the edge $x_0x_1$. There are $2^{n+2}-1$ vertices which are closer to $x_o$ than $x_1$, and there are $2^n-1$ vertices closer to $x_1$ than $x_o$. So, $\sqrt{\frac{n_{x_o}(x_0x_1,T_{n})+n_{x_1}(x_0x_1,T_{n})-2}{n_{x_o}(x_0x_1,T_{n})n_{x_1}(x_0x_1,T_{n})}}=\sqrt{\frac{2^{n+2}+2^{n}-4}{(2^{n+2}-1)(2^{n}-1)}}$. The edge $ax_0$ has the same attitude as the blue edge $x_0x_1$. In total there are 6 edges with this value related to Graovac-Ghorbani index. 
			The number of vertices closer to vertex $a$ is the same as the number of vertices closer to vertex $x_1$ and are $2^{n}-1$ vertices. So, $\sqrt{\frac{n_{a}(ax_1,T_{n})+n_{x_1}(ax_1,T_{n})-2}{n_{a}(ax_1,T_{n})n_{x_1}(ax_1,T_{n})}}=\frac{\sqrt{2^{n+1}-4}}{2^{n}-1}$, and in total, we have 3 edges like this one.
		
		Now consider  the edge $x_1x_2$. There are $2(2^{n+1}-1)+2^n+1$ vertices which are closer to $x_1$ than $x_2$, and there are $2^{n-1}-1$ vertices closer to $x_2$ than $x_1$. So, $\sqrt{\frac{n_{x_o}(x_0x_1,T_{n})+n_{x_1}(x_0x_1,T_{n})-2}{n_{x_o}(x_0x_1,T_{n})n_{x_1}(x_0x_1,T_{n})}}=\sqrt{\frac{2^{n+2}+2^{n}+2^{n-1}-4}{(2^{n+2}+2^{n}-1)(2^{n-1}-1)}}$. The edge $bx_1$ has the same attitude as the red edge $x_1x_2$. In total there are 12 edges with this value related to Graovac-Ghorbani index. The number of vertices closer to vertex $b$ is the same as the number of vertices closer to vertex $x_2$, and in total, and are $2^{n-1}-1$ vertices. So, $\sqrt{\frac{n_{b}(bx_1,T_{n})+n_{x_1}(bx_1,T_{n})-2}{n_{b}(bx_1,T_{n})n_{x_1}(bx_1,T_{n})}}=\frac{\sqrt{2^{n}-4}}{2^{n-1}-1}$, and in total, we have 6 edges like this one.
		
		By continuing this process in the $i$-th level ($i>1$), we have:
		$$\sqrt{\frac{n_{x_{i-1}}(x_{i-1}x_i,T_{n})+n_{x_i}(x_{i-1}x_i,T_{n})-2}{n_{x_{i-1}}(x_{i-1}x_i,T_{n})n_{x_i}(x_{i-1}x_i,T_{n})}}=\sqrt{\frac{2^{n+2}+(\sum_{t=0}^{i-2}2^{n-t})+2^{n-i+1}-4}{(2^{n+2}-1+\sum_{t=0}^{i-2}2^{n-t})(2^{n-i+1}-1)}}.$$
		We have $3(2^i)$ edges like this one. The number of vertices closer to vertex $x_i$ is the same as the number of vertices closer to its neighbour in horizontal edge with one endpoint $x_i$ (suppose $l$), and are $2^{n-i+2}-1$ vertices. So, $\sqrt{\frac{n_{l}(lx_1,T_{n})+n_{x_1}(lx_1,T_{n})-2}{n_{l}(lx_1,T_{n})n_{x_1}(lx_1,T_{n})}}=\frac{\sqrt{2^{n-i+2}-4}}{2^{n-i+1}-1}$, and in total, we have $3(2^{i-1})$ edges like this one.
		
		Finally, the number of vertices closer to vertex $x_0$ is the same as the number of vertices closer to vertex $u$,  the number of vertices closer to vertex $x_0$ is the same as the number of vertices closer to vertex $v$, and the number of vertices closer to vertex $v$ is the same as the number of vertices closer to vertex $u$, and are $2^{n+1}-1$ vertices. 
		
		So by the definition of the Graovac-Ghorbani index and our argument, we have 
		
		\begin{align*}
		ABC_{GG}(T_n)&=6\sqrt{\frac{2^{n+2}+2^{n}-4}{(2^{n+2}-1)(2^{n}-1)}}\\
		&\quad+ \sum_{i=2}^{n}3(2^i)\left(\sqrt{\frac{2^{n+2}+(\sum_{t=0}^{i-2}2^{n-t})+2^{n-i+1}-4}{(2^{n+2}-1+\sum_{t=0}^{i-2}2^{n-t})(2^{n-i+1}-1)}}\right)\\
		&\quad+ \left(\sum_{i=1}^{n}3(2^{i-1})\left(\frac{\sqrt{2^{n-i+2}-4}}{2^{n-i+1}-1}\right)\right)+\frac{3\sqrt{2^{n+2}-4}}{2^{n+1}-1},
		\end{align*}
		and therefore we have the result. 	\qed
			\end{itemize}
	\end{proof}	

	\begin{figure}[!h]
		\begin{center}
			\psscalebox{0.45 0.45}
			{
				\begin{pspicture}(0,-8.2)(17.6,5.8)
				\definecolor{colour0}{rgb}{0.0,0.5019608,0.0}
				\psline[linecolor=black, linewidth=0.08](8.8,-1.4)(7.2,-3.8)(10.4,-3.8)(8.8,-1.4)(8.8,-1.4)
				\psline[linecolor=blue, linewidth=0.08](8.8,-1.4)(10.0,0.2)(10.0,0.2)
				\psline[linecolor=red, linewidth=0.08](10.0,0.2)(10.8,1.0)(10.8,1.0)
				\psline[linecolor=colour0, linewidth=0.08](10.8,1.0)(11.6,1.8)(11.6,1.8)
				\pscircle[linecolor=black, linewidth=0.08, linestyle=dotted, dotsep=0.10583334cm, dimen=outer](8.8,2.2){3.6}
				\pscircle[linecolor=black, linewidth=0.08, linestyle=dotted, dotsep=0.10583334cm, dimen=outer](14.0,-4.6){3.6}
				\pscircle[linecolor=black, linewidth=0.08, linestyle=dotted, dotsep=0.10583334cm, dimen=outer](3.6,-4.6){3.6}
				\psline[linecolor=black, linewidth=0.08](10.0,0.2)(8.0,0.2)(7.6,0.2)(8.8,-1.4)(8.8,-1.4)
				\psline[linecolor=black, linewidth=0.08](10.8,1.0)(9.2,1.0)(10.0,0.2)(10.0,0.2)
				\psline[linecolor=black, linewidth=0.08](11.6,1.8)(10.4,1.8)(10.0,1.8)(10.8,1.0)(10.8,1.0)
				\psdots[linecolor=black, dotsize=0.1](7.2,0.6)
				\psdots[linecolor=black, dotsize=0.1](6.8,1.0)
				\psdots[linecolor=black, dotsize=0.1](6.4,1.4)
				\psdots[linecolor=black, dotsize=0.1](8.8,1.4)
				\psdots[linecolor=black, dotsize=0.1](8.4,1.8)
				\psdots[linecolor=black, dotsize=0.1](8.0,2.2)
				\psdots[linecolor=black, dotsize=0.1](9.6,2.2)
				\psdots[linecolor=black, dotsize=0.1](9.2,2.6)
				\psdots[linecolor=black, dotsize=0.1](8.8,3.0)
				\psdots[linecolor=black, dotsize=0.1](11.74,2.02)
				\psdots[linecolor=black, dotsize=0.1](11.82,2.18)
				\psdots[linecolor=black, dotsize=0.1](11.92,2.36)
				\psdots[linecolor=black, dotstyle=o, dotsize=0.4, fillcolor=white](8.8,-1.4)
				\psdots[linecolor=black, dotstyle=o, dotsize=0.4, fillcolor=white](7.2,-3.8)
				\psdots[linecolor=black, dotstyle=o, dotsize=0.4, fillcolor=white](10.4,-3.8)
				\psdots[linecolor=black, dotstyle=o, dotsize=0.4, fillcolor=white](10.0,0.2)
				\psdots[linecolor=black, dotstyle=o, dotsize=0.4, fillcolor=white](10.8,1.0)
				\psdots[linecolor=black, dotstyle=o, dotsize=0.4, fillcolor=white](11.6,1.8)
				\rput[bl](8.4,3.8){$G_n$}
				\rput[bl](13.86,-4.84){$G_n$}
				\rput[bl](3.16,-4.74){$G_n$}
				\rput[bl](8.66,-2.02){$x_0$}
				\rput[bl](10.14,-0.22){$x_1$}
				\rput[bl](10.98,0.62){$x_2$}
				\rput[bl](11.7,1.28){$x_3$}
				\rput[bl](9.86,-3.7){u}
				\rput[bl](7.5,-3.66){v}
				\psdots[linecolor=black, dotstyle=o, dotsize=0.4, fillcolor=white](7.6,0.2)
				\psdots[linecolor=black, dotstyle=o, dotsize=0.4, fillcolor=white](9.2,1.0)
				\psdots[linecolor=black, dotstyle=o, dotsize=0.4, fillcolor=white](10.0,1.8)
				\rput[bl](7.1,0.06){a}
				\rput[bl](8.7,0.82){b}
				\rput[bl](9.58,1.48){c}
				\end{pspicture}
			}
		\end{center}
		\caption{Graph $T_n$.}\label{Tn}
	\end{figure}
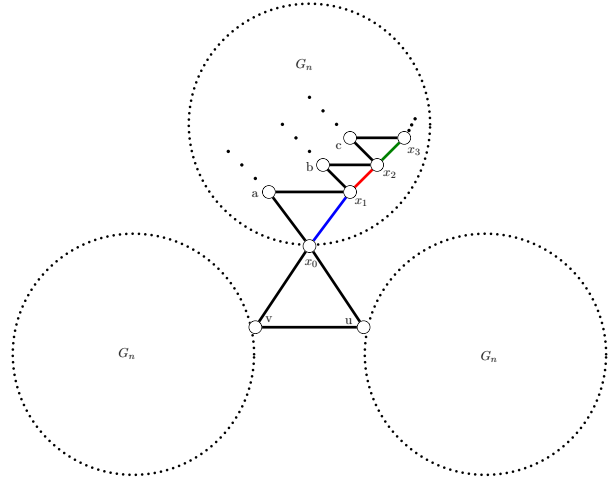

\subsection{Nanostar dendrimers}
We want to compute  the atom-bond connectivity of the nanostar dendrimer $D_k$ defined in \cite{Khalifeh}. First we define recursively an auxiliary family of rooted dendrimers $G_k$ 
$(k\geq 1)$. We need
a fixed graph $F$ defined in Figure \ref{FG1}, we consider one of its endpoint to be the root of $F$.
The graph $G_1$ is defined in Figure \ref{FG1}, the leaf being its root. Now we define $G_k$ $(k\geq2)$ the bouquet of the following 3 graphs: $G_{k-1}, G_{k-1}$, and $F$ with respect to their roots; the
root of $G_k$ is taken to be its unique leaf (see $G_2$ and $G_3$ in Figure \ref{G2G3}). Finally, we define $D_k$ $(k\geq 1)$ as the bouquet of 3 copies of $G_k$ with respect to their roots ($D_2$ is shown in Figure \ref{nanostar}, where the circles represent hexagons).

\begin{figure}
\hspace{2.7cm}	\begin{minipage}{6.5cm}
		\psscalebox{0.47 0.47}
		{
			\begin{pspicture}(0,-4.8)(5.62,-3.14)
			\psline[linecolor=black, linewidth=0.04](0.01,-3.97)(0.81,-3.97)(1.21,-3.17)(2.01,-3.17)(2.41,-3.97)(2.01,-4.77)(1.21,-4.77)(0.81,-3.97)(0.81,-3.97)
			\psline[linecolor=black, linewidth=0.04](2.41,-3.97)(3.21,-3.97)(3.61,-3.17)(4.41,-3.17)(4.81,-3.97)(4.41,-4.77)(3.61,-4.77)(3.21,-3.97)(3.21,-3.97)
			\psline[linecolor=black, linewidth=0.04](5.21,-3.97)(4.81,-3.97)(4.81,-3.97)
			\psline[linecolor=black, linewidth=0.04](5.61,-3.97)(5.21,-3.97)(5.21,-3.97)
			\end{pspicture}
		}
	\end{minipage}
	\begin{minipage}{5cm}
		\psscalebox{0.47 0.47}
		{
			\begin{pspicture}(0,-4.8)(2.4423609,-3.14)
			\psline[linecolor=black, linewidth=0.04](0.01,-3.97)(0.81,-3.97)(1.21,-3.17)(2.01,-3.17)(2.41,-3.97)(2.01,-4.77)(1.21,-4.77)(0.81,-3.97)(0.81,-3.97)
			\end{pspicture}
		}
	\end{minipage}
 	\caption{Graphs $F$ and $G_1$, respectively.}\label{FG1}
\end{figure}

\begin{figure}
\hspace{1.1cm} 	\begin{minipage}{7cm}
		\psscalebox{0.46 0.46}
		{
			\begin{pspicture}(0,-6.79)(4.86,2.85)
			\psline[linecolor=black, linewidth=0.04](2.43,-0.38)(1.63,-1.18)(1.63,-1.98)(2.43,-2.78)(3.23,-1.98)(3.23,-1.18)(2.43,-0.38)(2.43,-0.38)
			\psline[linecolor=black, linewidth=0.04](2.43,-3.58)(2.43,-2.78)(2.43,-2.78)
			\psline[linecolor=black, linewidth=0.04](2.43,-3.58)(1.63,-4.38)(1.63,-5.18)(2.43,-5.98)(3.23,-5.18)(3.23,-4.38)(2.43,-3.58)(2.43,-3.58)
			\psline[linecolor=black, linewidth=0.04](2.43,-6.78)(2.43,-5.98)(2.43,-5.98)
			\psline[linecolor=black, linewidth=0.04](2.43,0.42)(2.43,-0.38)(2.43,-0.38)
			\psline[linecolor=black, linewidth=0.04](2.43,0.42)(3.23,1.22)(3.23,1.22)
			\psline[linecolor=black, linewidth=0.04](2.43,0.42)(1.63,1.22)(1.63,1.22)
			\psline[linecolor=black, linewidth=0.04](3.23,1.22)(3.23,2.02)(4.03,2.82)(4.83,2.82)(4.83,2.02)(4.03,1.22)(3.23,1.22)(3.23,1.22)
			\psline[linecolor=black, linewidth=0.04](1.63,1.22)(1.63,2.02)(0.83,2.82)(0.03,2.82)(0.03,2.02)(0.83,1.22)(1.63,1.22)(1.63,1.22)
			\end{pspicture}
		}
	\end{minipage}
	\begin{minipage}{7cm}
		\psscalebox{0.4 0.4}
		{
			\begin{pspicture}(0,-11.59)(17.276567,4.45)
			\psline[linecolor=black, linewidth=0.04](8.438284,-5.18)(7.638284,-5.98)(7.638284,-6.78)(8.438284,-7.58)(9.238284,-6.78)(9.238284,-5.98)(8.438284,-5.18)(8.438284,-5.18)
			\psline[linecolor=black, linewidth=0.04](8.438284,-8.38)(8.438284,-7.58)(8.438284,-7.58)
			\psline[linecolor=black, linewidth=0.04](8.438284,-8.38)(7.638284,-9.18)(7.638284,-9.98)(8.438284,-10.78)(9.238284,-9.98)(9.238284,-9.18)(8.438284,-8.38)(8.438284,-8.38)
			\psline[linecolor=black, linewidth=0.04](8.438284,-11.58)(8.438284,-10.78)(8.438284,-10.78)
			\psline[linecolor=black, linewidth=0.04](8.438284,-4.38)(8.438284,-5.18)(8.438284,-5.18)
			\psline[linecolor=black, linewidth=0.04](8.438284,-4.38)(9.238284,-3.58)(9.238284,-3.58)
			\psline[linecolor=black, linewidth=0.04](8.438284,-4.38)(7.638284,-3.58)(7.638284,-3.58)
			\psline[linecolor=black, linewidth=0.04](9.238284,-3.58)(9.238284,-2.78)(10.038284,-1.98)(10.838284,-1.98)(10.838284,-2.78)(10.038284,-3.58)(9.238284,-3.58)(9.238284,-3.58)
			\psline[linecolor=black, linewidth=0.04](7.638284,-3.58)(7.638284,-2.78)(6.838284,-1.98)(6.038284,-1.98)(6.038284,-2.78)(6.838284,-3.58)(7.638284,-3.58)(7.638284,-3.58)
			\psline[linecolor=black, linewidth=0.04](10.838284,-1.98)(11.638284,-1.18)(11.638284,-1.18)
			\psline[linecolor=black, linewidth=0.04](11.638284,-1.18)(11.638284,-0.38)(12.438284,0.42)(13.238284,0.42)(13.238284,-0.38)(12.438284,-1.18)(11.638284,-1.18)(11.638284,-1.18)
			\psline[linecolor=black, linewidth=0.04](6.038284,-1.98)(5.238284,-1.18)(5.238284,-1.18)
			\psline[linecolor=black, linewidth=0.04](5.238284,-1.18)(5.238284,-0.38)(4.438284,0.42)(3.638284,0.42)(3.638284,-0.38)(4.438284,-1.18)(5.238284,-1.18)(5.238284,-1.18)
			\psline[linecolor=black, linewidth=0.04](14.038284,4.42)(13.238284,3.62)(13.238284,2.82)(14.038284,2.02)(14.838284,2.82)(14.838284,3.62)(14.038284,4.42)(14.038284,4.42)
			\psline[linecolor=black, linewidth=0.04](14.038284,1.22)(14.038284,2.02)(14.038284,2.02)
			\psline[linecolor=black, linewidth=0.04](14.038284,1.22)(13.238284,0.42)(13.238284,0.42)
			\psline[linecolor=black, linewidth=0.04](14.038284,1.22)(14.838284,1.22)(14.838284,1.22)
			\psline[linecolor=black, linewidth=0.04](3.238284,4.42)(2.438284,3.62)(2.438284,2.82)(3.238284,2.02)(4.038284,2.82)(4.038284,3.62)(3.238284,4.42)(3.238284,4.42)
			\psline[linecolor=black, linewidth=0.04](3.238284,1.22)(3.238284,2.02)(3.238284,2.02)
			\psline[linecolor=black, linewidth=0.04](3.638284,0.42)(3.238284,1.22)(2.438284,1.22)(2.438284,1.22)
			\psline[linecolor=black, linewidth=0.04](14.838284,1.22)(15.638284,2.02)(16.438284,2.02)(17.238283,1.22)(16.438284,0.42)(15.638284,0.42)(14.838284,1.22)(14.838284,1.22)
			\psline[linecolor=black, linewidth=0.04](2.438284,1.22)(1.638284,2.02)(0.838284,2.02)(0.038283996,1.22)(0.838284,0.42)(1.638284,0.42)(2.438284,1.22)(2.438284,1.22)
			\end{pspicture}
		}
	\end{minipage}
	\caption{Graphs $G_2$ and $G_3$, respectively.}\label{G2G3}
\end{figure}

\begin{theorem}
	For the dendrimer $D_3[n]$ we have:
	\begin{align*}
	ABC(D_3[n])=6(2^n)-4+(18(2^n)-9)\sqrt{2}.
	\end{align*}
\end{theorem}

\begin{proof}
	There are $3+9\displaystyle\sum_{k=0}^{n-1}2^k$ edges with endpoints of degree 3. Also there are $6+18\displaystyle\sum_{k=0}^{n-1}2^k$ edges with endpoints of degree 3 and 2, and there are $12+18\displaystyle\sum_{k=0}^{n-1}2^k$ edges with endpoints of degree 2. Therefore 
	\begin{align*}
	ABC(D_3[n])&=\left(3+9\displaystyle\sum_{k=0}^{n-1}2^k\right)\sqrt{\frac{3+3-2}{3(3)}}+\left(6+18\displaystyle\sum_{k=0}^{n-1}2^k\right)\sqrt{\frac{3+2-2}{3(2)}}\\
	& \quad + \left(12+18\displaystyle\sum_{k=0}^{n-1}2^k\right)\sqrt{\frac{2+2-2}{2(3)}},
	\end{align*}
	and we have the result.	
	\qed
\end{proof}

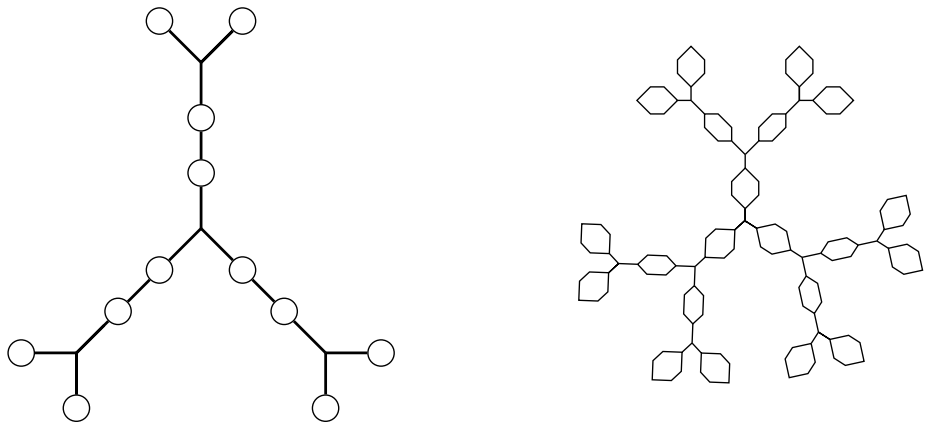
\begin{figure}
\hspace{1cm}	\begin{minipage}{7.5cm}
		\psscalebox{0.46 0.46}
{
\begin{pspicture}(0,-6.0000005)(11.205557,6.005556)
\psdots[linecolor=black, fillstyle=solid, dotstyle=o, dotsize=0.8, fillcolor=white](5.6027775,1.2027783)
\psdots[linecolor=black, fillstyle=solid, dotstyle=o, dotsize=0.8, fillcolor=white](5.6027775,2.8027782)
\psdots[linecolor=black, fillstyle=solid, dotstyle=o, dotsize=0.8, fillcolor=white](10.802777,-3.9972217)
\psdots[linecolor=black, fillstyle=solid, dotstyle=o, dotsize=0.8, fillcolor=white](9.202778,-5.597222)
\psdots[linecolor=black, fillstyle=solid, dotstyle=o, dotsize=0.8, fillcolor=white](0.4027777,-3.9972217)
\psdots[linecolor=black, fillstyle=solid, dotstyle=o, dotsize=0.8, fillcolor=white](2.0027778,-5.597222)
\psline[linecolor=black, linewidth=0.08](0.8027777,-3.9972217)(2.0027778,-3.9972217)(2.0027778,-5.1972218)(2.0027778,-3.9972217)(2.8027778,-3.1972218)(2.8027778,-3.1972218)
\psline[linecolor=black, linewidth=0.08](9.202778,-3.9972217)(9.202778,-5.1972218)(9.202778,-5.1972218)
\psline[linecolor=black, linewidth=0.08](9.202778,-3.9972217)(10.402778,-3.9972217)(10.402778,-3.9972217)
\psline[linecolor=black, linewidth=0.08](5.6027775,-0.39722168)(5.6027775,0.8027783)(5.6027775,0.8027783)
\psline[linecolor=black, linewidth=0.08](5.6027775,1.6027783)(5.6027775,2.4027784)(5.6027775,2.4027784)
\psline[linecolor=black, linewidth=0.08](5.6027775,4.402778)(5.6027775,3.2027783)(5.6027775,3.2027783)
\psline[linecolor=black, linewidth=0.08](5.6027775,4.402778)(6.802778,5.6027784)(5.6027775,4.402778)(4.4027777,5.6027784)(4.4027777,5.6027784)
\psdots[linecolor=black, fillstyle=solid, dotstyle=o, dotsize=0.8, fillcolor=white](6.802778,5.6027784)
\psdots[linecolor=black, fillstyle=solid, dotstyle=o, dotsize=0.8, fillcolor=white](4.4027777,5.6027784)
\psline[linecolor=black, linewidth=0.08](2.0027778,-3.9972217)(5.6027775,-0.39722168)(5.6027775,-0.39722168)(9.202778,-3.9972217)(9.202778,-3.9972217)
\psdots[linecolor=black, fillstyle=solid, dotstyle=o, dotsize=0.8, fillcolor=white](6.802778,-1.5972217)
\psdots[linecolor=black, fillstyle=solid, dotstyle=o, dotsize=0.8, fillcolor=white](8.002778,-2.7972217)
\psdots[linecolor=black, fillstyle=solid, dotstyle=o, dotsize=0.8, fillcolor=white](4.4027777,-1.5972217)
\psdots[linecolor=black, fillstyle=solid, dotstyle=o, dotsize=0.8, fillcolor=white](3.2027776,-2.7972217)
\end{pspicture}
}
	\end{minipage}
		\begin{minipage}{7.5cm}
		\psscalebox{0.45 0.45}
{
\begin{pspicture}(0,-7.3653517)(10.220022,2.6543653)
\psline[linecolor=black, linewidth=0.04](4.946707,-2.1839187)(4.946707,-2.5839188)(4.946707,-2.1839187)
\psline[linecolor=black, linewidth=0.04](4.946707,-2.1839187)(4.5467067,-1.7839187)(4.5467067,-1.3839188)(4.946707,-0.9839188)(5.3467064,-1.3839188)(5.3467064,-1.7839187)(4.946707,-2.1839187)(4.946707,-2.1839187)
\psline[linecolor=black, linewidth=0.04](4.946707,-0.9839188)(4.946707,-0.58391875)(4.946707,-0.58391875)
\psline[linecolor=black, linewidth=0.04](4.946707,-0.58391875)(5.3467064,-0.18391877)(5.3467064,-0.18391877)
\psline[linecolor=black, linewidth=0.04](4.946707,-0.58391875)(4.5467067,-0.18391877)(4.5467067,-0.18391877)
\psline[linecolor=black, linewidth=0.04](5.3467064,-0.18391877)(5.3467064,0.21608122)(5.7467065,0.61608124)(6.1467066,0.61608124)(6.1467066,0.21608122)(5.7467065,-0.18391877)(5.3467064,-0.18391877)(5.3467064,-0.18391877)
\psline[linecolor=black, linewidth=0.04](4.5467067,-0.18391877)(4.5467067,0.21608122)(4.1467066,0.61608124)(3.7467065,0.61608124)(3.7467065,0.21608122)(4.1467066,-0.18391877)(4.5467067,-0.18391877)(4.5467067,-0.18391877)
\psline[linecolor=black, linewidth=0.04](6.1467066,0.61608124)(6.5467067,1.0160812)(6.5467067,1.0160812)
\psline[linecolor=black, linewidth=0.04](3.7467065,0.61608124)(3.3467066,1.0160812)(3.3467066,1.0160812)
\psline[linecolor=black, linewidth=0.04](6.5479274,1.016033)(6.5479274,1.4160331)(6.5479274,1.016033)(6.9479275,1.016033)(6.9479275,1.016033)
\psline[linecolor=black, linewidth=0.04](6.5467067,1.4160812)(6.1467066,1.8160812)(6.1467066,2.2160811)(6.5467067,2.6160812)(6.946707,2.2160811)(6.946707,1.8160812)(6.5467067,1.4160812)(6.5467067,1.4160812)
\psline[linecolor=black, linewidth=0.04](6.946707,1.0160812)(7.3467064,1.4160812)(7.7467065,1.4160812)(8.146707,1.0160812)(7.7467065,0.61608124)(7.3467064,0.61608124)(6.946707,1.0160812)(6.946707,1.0160812)
\psline[linecolor=black, linewidth=0.04](3.346579,1.0147685)(3.346579,1.4147685)(3.7465792,1.8147686)(3.7465792,2.2147684)(3.346579,2.6147685)(2.946579,2.2147684)(2.946579,1.8147686)(3.346579,1.4147685)(3.346579,1.4147685)
\psline[linecolor=black, linewidth=0.04](3.3480194,1.0159538)(2.9480193,1.0159538)(2.5480192,1.4159538)(2.1480193,1.4159538)(1.7480192,1.0159538)(2.1480193,0.61595374)(2.5480192,0.61595374)(2.9480193,1.0159538)(2.9480193,1.0159538)
\psline[linecolor=black, linewidth=0.04](5.281036,-2.760522)(4.9442973,-2.5446353)(5.281036,-2.760522)
\psline[linecolor=black, linewidth=0.04](5.281036,-2.760522)(5.8336616,-2.6396697)(6.1704,-2.8555562)(6.2912526,-3.4081817)(5.738627,-3.529034)(5.4018884,-3.3131473)(5.281036,-2.760522)(5.281036,-2.760522)
\psline[linecolor=black, linewidth=0.04](6.2912526,-3.4081817)(6.627991,-3.624068)(6.627991,-3.624068)
\psline[linecolor=black, linewidth=0.04](6.627991,-3.624068)(6.748843,-4.1766934)(6.748843,-4.1766934)
\psline[linecolor=black, linewidth=0.04](6.627991,-3.624068)(7.1806164,-3.5032158)(7.1806164,-3.5032158)
\psline[linecolor=black, linewidth=0.04](6.748843,-4.1766934)(7.0855823,-4.39258)(7.2064342,-4.945205)(6.9905477,-5.2819443)(6.653809,-5.0660577)(6.532957,-4.513432)(6.748843,-4.1766934)(6.748843,-4.1766934)
\psline[linecolor=black, linewidth=0.04](7.1806164,-3.5032158)(7.5173554,-3.7191024)(8.069981,-3.5982502)(8.285867,-3.2615113)(7.949128,-3.045625)(7.396503,-3.1664772)(7.1806164,-3.5032158)(7.1806164,-3.5032158)
\psline[linecolor=black, linewidth=0.04](6.9905477,-5.2819443)(7.1114,-5.8345695)(7.1114,-5.8345695)
\psline[linecolor=black, linewidth=0.04](8.285867,-3.2615113)(8.838492,-3.140659)(8.838492,-3.140659)
\psline[linecolor=black, linewidth=0.04](7.1118913,-5.8347993)(7.44863,-6.050686)(7.1118913,-5.8347993)(6.8960047,-6.171538)(6.8960047,-6.171538)
\psline[linecolor=black, linewidth=0.04](7.4481387,-6.050456)(8.000764,-5.9296036)(8.3375025,-6.14549)(8.458355,-6.6981153)(7.90573,-6.818968)(7.568991,-6.603081)(7.4481387,-6.050456)(7.4481387,-6.050456)
\psline[linecolor=black, linewidth=0.04](6.8955135,-6.171308)(7.0163655,-6.7239337)(6.8004794,-7.0606723)(6.2478538,-7.1815243)(6.127002,-6.628899)(6.3428884,-6.2921605)(6.8955135,-6.171308)(6.8955135,-6.171308)
\psline[linecolor=black, linewidth=0.04](8.839256,-3.1404498)(9.175995,-3.3563364)(9.296847,-3.9089615)(9.633586,-4.1248484)(10.186212,-4.003996)(10.065359,-3.4513707)(9.728621,-3.2354841)(9.175995,-3.3563364)(9.175995,-3.3563364)
\psline[linecolor=black, linewidth=0.04](8.838283,-3.1398954)(9.05417,-2.8031566)(9.606794,-2.6823044)(9.822681,-2.3455656)(9.701829,-1.7929403)(9.149204,-1.9137925)(8.933317,-2.2505312)(9.05417,-2.8031566)(9.05417,-2.8031566)
\psline[linecolor=black, linewidth=0.04](4.6404505,-2.8087912)(4.9341197,-2.5372064)(4.6404505,-2.8087912)
\psline[linecolor=black, linewidth=0.04](4.6411376,-2.8091736)(4.619053,-3.3744278)(4.325383,-3.6460125)(3.7601292,-3.6239278)(3.782214,-3.0586736)(4.0758834,-2.7870889)(4.6411376,-2.8091736)(4.6411376,-2.8091736)
\psline[linecolor=black, linewidth=0.04](3.7601292,-3.6239278)(3.4664598,-3.8955126)(3.4664598,-3.8955126)
\psline[linecolor=black, linewidth=0.04](3.4664598,-3.8955126)(2.9012055,-3.8734279)(2.9012055,-3.8734279)
\psline[linecolor=black, linewidth=0.04](3.4664598,-3.8955126)(3.444375,-4.460767)(3.444375,-4.460767)
\psline[linecolor=black, linewidth=0.04](2.9012055,-3.8734279)(2.607536,-4.1450124)(2.0422819,-4.1229277)(1.7706972,-3.8292582)(2.0643668,-3.5576737)(2.6296208,-3.5797584)(2.9012055,-3.8734279)(2.9012055,-3.8734279)
\psline[linecolor=black, linewidth=0.04](3.444375,-4.460767)(3.1507056,-4.7323513)(3.1286209,-5.2976055)(3.4002056,-5.591275)(3.693875,-5.31969)(3.7159598,-4.754436)(3.444375,-4.460767)(3.444375,-4.460767)
\psline[linecolor=black, linewidth=0.04](1.7706972,-3.8292582)(1.2054431,-3.8071735)(1.2054431,-3.8071735)
\psline[linecolor=black, linewidth=0.04](3.4002056,-5.591275)(3.378121,-6.156529)(3.378121,-6.156529)
\psline[linecolor=black, linewidth=0.04](1.2070984,-3.8073893)(0.91342884,-4.078974)(1.2070984,-3.8073893)(0.9355136,-3.5137198)(0.9355136,-3.5137198)
\psline[linecolor=black, linewidth=0.04](0.9117737,-4.0787582)(0.8896889,-4.6440125)(0.59601945,-4.915597)(0.030765306,-4.8935122)(0.052850045,-4.328258)(0.3465195,-4.0566735)(0.9117737,-4.0787582)(0.9117737,-4.0787582)
\psline[linecolor=black, linewidth=0.04](0.9338584,-3.513504)(0.36860424,-3.4914193)(0.09701952,-3.1977499)(0.11910426,-2.6324956)(0.6843584,-2.6545806)(0.9559431,-2.94825)(0.9338584,-3.513504)(0.9338584,-3.513504)
\psline[linecolor=black, linewidth=0.04](3.3747423,-6.157989)(3.0810728,-6.429574)(2.5158186,-6.4074893)(2.2221491,-6.679074)(2.2000644,-7.244328)(2.7653186,-7.2664127)(3.058988,-6.994828)(3.0810728,-6.429574)(3.0810728,-6.429574)
\psline[linecolor=black, linewidth=0.04](3.379581,-6.159908)(3.6511655,-6.453577)(3.6290808,-7.0188313)(3.9006655,-7.3125005)(4.4659195,-7.334585)(4.4880047,-6.7693315)(4.2164197,-6.4756618)(3.6511655,-6.453577)(3.6511655,-6.453577)
\end{pspicture}
}
\end{minipage} 
	\caption{Nanostar $D_2$ and $D_3[2]$, respectively.}\label{nanostar} 
\end{figure}

\section{Acknowledgements} 
	
The  author would like to thank the Research Council of Norway and Department of Informatics, University of Bergen for their support.

\end{document}